\documentclass[reqno]{amsart}

\usepackage{amsmath,amsthm,amsfonts,amssymb,bm,graphicx, mathrsfs}

\usepackage{comment} \usepackage[normalem]{ulem} \usepackage
{hyperref}

\usepackage{tikz}
\usepackage[all]{xy}
\usetikzlibrary{matrix,arrows,decorations.pathmorphing}
\usepackage{tikz-cd}

\hypersetup{colorlinks=true,citecolor=blue,linkcolor=blue,urlcolor=blue, pdfstartview=FitH }

\theoremstyle{plain} \newtheorem{lemma}{Lemma}[section] \newtheorem{theorem}{Theorem}  \newtheorem{cor}[lemma]{Corollary}
\newtheorem{prop}[lemma]{Proposition}

\numberwithin{equation}{section} \newtheorem{quest}{Question}

\theoremstyle{definition}

\theoremstyle{remark} \newtheorem{remark}{Remark}[section]

\newcommand{\B}{\mathcal{B}}  \newcommand{\gen}{\mathrm{gen}}   \newcommand{\I}{\mathcal{I}} \renewcommand{\geq}{\geqslant} \renewcommand{\leq}{\leqslant}  \newcommand{\D}{\mathcal{D}}

\DeclareMathOperator{\GL}{GL} \DeclareMathOperator{\SL}{SL}      \DeclareMathOperator{\Ad}{Ad} \DeclareMathOperator{\Eis}{Eis} 
\DeclareMathOperator{\vol}{vol}

 \newcommand{\eps}{\varepsilon} \def\PGL{\operatorname{PGL}}         

\newcommand{\changed}[1]{{\color{black} #1}}  \usepackage{calc} \newsavebox\CBox \newcommand\hcancel[2][0.5pt]{%
  \changed{\ifmmode\sbox\CBox{$#2$}\else\sbox\CBox{#2}\fi%
    \makebox[0pt][l]{\usebox\CBox}%
    \rule[0.5\ht\CBox-#1/2]{\wd\CBox}{#1}}}

\makeatletter \DeclareRobustCommand\widecheck[1]{{\mathpalette\@widecheck{#1}}} \def\@widecheck#1#2{%
  \setbox\z@\hbox{\m@th$#1#2$}%
  \setbox\tw@\hbox{\m@th$#1%
    \widehat{%
      \vrule\@width\z@\@height\ht\z@ \vrule\@height\z@\@width\wd\z@}$}%
  \dp\tw@-\ht\z@ \@tempdima\ht\z@ \advance\@tempdima2\ht\tw@ \divide\@tempdima\thr@@ \setbox\tw@\hbox{%
    \raise\@tempdima\hbox{\scalebox{1}[-1]{\lower\@tempdima\box \tw@}}}%
  {\ooalign{\box\tw@ \cr \box\z@}}} \makeatother

             \newcommand{\tr}{\operatorname{tr}}

\renewcommand{\Bbb}{\mathbb}

\usepackage
{color}

\usepackage{enumerate}

\begin{document}

\author{Valentin Blomer} \author{Subhajit Jana} \author{Paul D. Nelson}

\address{Mathematisches Institut, Endenicher Allee 60, 53115 Bonn, Germany} \email{blomer@math.uni-bonn.de}

\address{Queen Mary University of London, Mile End Road, London E1 4NS, UK} \email{s.jana@qmul.ac.uk}

\address{Aarhus University, Ny Munkegade 118,
  DK-8000 Aarhus, Denmark}
\email{paul.nelson@math.au.dk}

\title[Local integral transforms]{Local integral transforms and global spectral decomposition}

\thanks{The first author was supported in part by the DFG under Germany's Excellence Strategy - EXC-2047/1 - 390685813 and ERC Advanced Grant 101054336.  The second author is supported by the EPSRC research grant EP/Z536611/1 (NIA) from UKRI. The third author is supported by a research grant (VIL54509) from VILLUM FONDEN.  Some work towards this paper was carried out while the third author was at the Institute for Advanced Study, supported by the National Science Foundation under Grant No. DMS-1926686.  The paper was finished when all three authors attended the program ``Analytic Number Theory'' at the  Institut Mittag-Leffler, whose excellent working conditions are acknowledged.}

\begin{abstract} 
  We establish an explicit global spectral decomposition of shifted convolution sums and the second moment of automorphic $L$-functions for Maa{\ss} forms with explicit integral transforms as well as explicit inversion formulae over every local field.
\end{abstract}


\setcounter{tocdepth}{2} \maketitle

\maketitle

\section{Introduction}

\subsection{Motohashi's formula}\label{sec:cngtt989nr}
The connection between the cubic moment of $L$-functions for ${\rm GL}(2)$ and the fourth moment of $L$-functions for ${\rm GL}(1)$ has been a very successful theme at the interface of analytic number theory and automorphic forms for several decades. The starting point is Motohashi's beautiful formula \cite[Theorem 4.2]{Mo1}: for a sufficiently nice test function $V : \mathbb{R} \rightarrow \mathbb{R}$, we have
\begin{equation}\label{fourth}
  \begin{split}
    \int_{\Bbb{R}} &  |\zeta(\tfrac{1}{2} + it)|^4 V(t) dt = \text{ main term } + \sum_j\frac{ {L(\tfrac{1}{2}, \psi_j)}^3}{L(1,  \psi_j, \text{Ad})}  \widecheck{V}(t_j) + \text{ similar,}
  \end{split}
\end{equation}
where the sum runs over Maa{\ss} forms $\psi_j$ with spectral parameter $t_j$ for the group ${\rm SL}_2(\Bbb{Z})$. It is remarkable that the integral transform $V \mapsto \widecheck{V}$ is given in a completely explicit, but rather complicated way in terms of hypergeometric functions (see \eqref{trafo} below).  The formula~\eqref{fourth} has a precursor, namely the spectral decomposition of the binary additive divisor problem (\cite[Theorem 3]{Mo}): for $b \in \Bbb{N}$ and a sufficiently nice test function $V : \mathbb{R} \rightarrow \mathbb{R}$, we have
\begin{equation}\label{shifted}
  \begin{split}
    \sum_{n \geq 1} &  \tau(n) \tau(n+b) V\Big(\frac{n}{b}\Big) = \text{ main term} + \sum_j\frac{ b^{1/2} \lambda_j(b) {L(\tfrac{1}{2}, \psi_j)}^2}{L(1,  \psi_j, \text{Ad})}  \widetilde{V}(t_j)  + \text{ similar,}
  \end{split}
\end{equation}
where $\lambda_j(b)$ denotes the corresponding Hecke eigenvalue and $\widetilde{V}$ is as in \eqref{trafo1} below.

The similarity between~\eqref{fourth} and~\eqref{shifted} is no coincidence. The left-hand side of~\eqref{shifted} is essentially an off-diagonal term in~\eqref{fourth} 
and indeed Motohashi uses 
\eqref{shifted} in the proof of~\eqref{fourth}.

Neither~\eqref{fourth} nor~\eqref{shifted} are easy to predict and lie very deep. Only much later, Michel and Venkatesh~\cite[Section 4.5]{MV} provided a conceptual explanation for why a formula of type~\eqref{fourth} should be expected. See~\cite{Wu} for a somewhat different approach, and~\cite{BHKM} for a version with non-archimedean test functions in the context of spectral reciprocity.

In addition to their structural interest, such formulae are extremely useful in analytic number theory. For instance, Iwaniec~\cite{Iw} uses the technology behind~\eqref{fourth} to establish the best results for the fourth moment of the Riemann zeta function in short intervals, and variations of the additive divisor problem have been used as an ingredient in subconvexity bounds~\cite{DFI}.
 
One can turn the tables and try to understand the cubic moment of automorphic $L$-functions in terms of $L$-functions on ${\rm GL}(1)$. The first important works in this direction are by Ivi\'c~\cite{Iv} and, most notably, by Conrey--Iwaniec~\cite{CI} in a twisted setting, leading to a Weyl-type subconvexity bound for automorphic $L$-functions twisted by real characters. The approach of Conrey--Iwaniec was extended in important works of Petrow~\cite{Pe} and Petrow--Young~\cite{PY1, PY2}, while the most general approach to the cubic moment of automorphic $L$-function is contained in~\cite{Ne} and~\cite{BFW}.

Further recent  variations   include  for instance \cite{Fr, Ka, Kw1, Kw2, Kw3, To, Wu2}.  \\

Looking back at Motohashi's formula~\eqref{fourth}, at least two questions spring to mind. The first concerns the integral transforms $\widecheck{V}$ and $\widetilde{V}$ in \eqref{trafo} and~\eqref{trafo1} below that at first sight look rather artificial.
\begin{quest}\label{Q1}
  Do these integral transforms have a structural interpretation as a map on the unitary dual of ${\rm PGL}_2(\Bbb{R})$ in terms of representation-theoretic data, are there analogous formulae for other local fields, and is there an explicit inversion formula of similar shape?
\end{quest}

The second question focuses on    non-split situations where Motohashi's approach is not possible:
\begin{quest}\label{Q2}
  Is there a formula similar to~\eqref{fourth} with the fourth moment of the Riemann zeta function replaced by the second moment of $L$-functions for ${\rm GL}(2)$, and is there also a formula similar to~\eqref{shifted} for Hecke eigenvalues of automorphic forms on ${\rm GL}(2)$?
\end{quest}

In this degree of generality, both questions have remained open until now, despite various attempts over the past 60 years (e.g.\ \cite{Se, Ju2, Sa, Ha-thesis, Mo5, Mo4, BM}). The purpose of this paper is to give complete and affirmative answers to both of them.  Not surprisingly, these questions are intertwined. The first question is of a local nature and puts classical transformation and inversion formulae, such as the Mehler--Fock transform, in a general and structural representation theoretic context. It is therefore of independent interest, but it is also the crucial input for the second question, 
which is of a global nature. 

As a sample of our global results, we give explicit generalizations of Motohashi's formulae to the non-split case: for Hecke--Maass cusp forms $f, g$ for $\SL_2(\mathbb{Z})$ and $b \in \Bbb{N}$, we have
\begin{equation*}
  \int_{\mathbb{R}}
  L(\tfrac{1}{2} + i t, f)
  \overline{L(\tfrac{1}{2} + i t, g)} V(t) 
  \, d t
  = \text{ main term }
  + \sum_j
  \frac{L(\tfrac{1}{2}, \psi_j) c_j }{L(1, \psi_j, \Ad)}  \widecheck{V}(t_j)
  + \text{ similar}
\end{equation*}
and
\begin{equation}\label{shifted-new}
  \sum_{n \geq 1}
  \lambda_{f}(n) \lambda_{g}(n + b)
  V \left( \frac{n}{b} \right)
  =   \sum_j
  \frac{b^{1/2} \lambda_j(b) c_j}{L(1, \psi_j, \Ad)} \widetilde{V}(t_j) + \text{ similar},
\end{equation}
where $c_j$ satisfies 
\begin{equation*}
  \lvert c_j \rvert^2 = L(\tfrac{1}{2}, \pi_f \otimes \pi_g \otimes \pi_{\psi_j}).
\end{equation*}
Even in this most basic case of Maa{\ss} forms for the group ${\rm SL}_2(\Bbb{Z})$, such results have resisted all attempts so far, although there have been several different workaround solutions that worked reasonably well in practice, see e.g.\ \cite{Ha, Ju1, Ju2, Sa}.  These approaches have involved several intermediate steps, based on $\delta$-symbol methods or Poincar{\'e} series, each requiring their own asymptotic analysis. They often degenerate when applied over number fields, with higher ramification, or when studying questions of uniformity (in $f$, $g$, $V$ and $b$).  Our work yields an explicit formula that allows the practicing analytic number theorist to bypass all intermediate steps needed in prior approaches.

Note that these formulae contain the \emph{same} integral transform as in \eqref{fourth} and \eqref{shifted}, but different types of $L$-functions, namely ``canonical'' square-roots in the terminology of \cite{MV}. We establish these formulae in an adelic setting so that they become applicable over any number field and allow ramification and additional test functions at arbitrary finite places with corresponding explicit $p$-adic integral transforms.

The problem can be traced back to Selberg \cite{Se}, who established the analytic continuation of the inner product of two non-holomorphic Poincar\'e series in the context of \eqref{shifted-new} for Maa{\ss} forms, but was unable to perform some explicit analysis (``\emph{We cannot make much use of this function at present}'' \cite[p.\ 14]{Se}). The closest approximation to \eqref{shifted-new}  is \cite{BH} (see also \cite{Mo4}) where an exact spectral decomposition is obtained -- also bypassing all intermediate steps -- with useful bounds in terms of Sobolev norms of $V$, but the integral transform $\widetilde{V}$ remained abstract, based only on the surjectivity of the Kirillov model.  We mention also the recent work of Goldfeld--Hinkle--Hoffstein \cite[Theorem 1.10]{GHH}, who obtain an asymptotic formula for unsmoothed shifted convolution sums in the spirit of \eqref{shifted-new}.

In the following two subsections, we address the above Questions 1 and 2 and state the main results.

\subsection{Local integral transforms}\label{sec:d1be073b0aa9}
The integral transforms relevant for \S~\ref{sec:cngtt989nr} are given explicitly as follows: with
\begin{equation*}
  \mathcal{K}(t, y) := \frac{1}{2}\sum_{\pm}  \Big( 1 \pm \frac{i}{\sinh(\pi t)}\Big) y^{-\frac{1}{2} \mp it}\frac{{\Gamma(\tfrac{1}{2} \pm it)}^2}{\Gamma(1 \pm 2 i t)} F\big( \tfrac{1}{2} \pm it, \tfrac{1}{2} \pm it, 1 \pm 2 i t, -1/y\big),
\end{equation*}
we have
\begin{equation}\label{trafo}
  \begin{split}
    \widecheck{V}(t) =& 2\int_0^{\infty} \Big( \int_{-\infty}^{\infty} V(x) \cos\Big(x \log \frac{y+1}{y}\Big)  \, d x  \Big)  \mathcal{K}(t, y) \frac{dy}{\sqrt{y(1+y)}}
  \end{split}
\end{equation}
and
\begin{equation}\label{trafo1}
  \widetilde{V}(t) = \int_0^{\infty} \mathcal{K}(t, y) V(y) dy.
\end{equation}
(For the sake of comparison, we note that Motohashi uses the notation $\alpha_j = 2/L(1, \psi_j, \text{Ad})$.)  In Motohashi's derivation, these transforms arise in a global argument using, among other ingredients,  Voronoi summation, the Kuznetsov formula, and Ramanujan's expansion of divisor functions in terms of Ramanujan sums.  We explain here how these transformations may be understood, in greater generality, using local representation theory.

Let $F$ be a local field, archimedean or non-archimedean, and $G:=\PGL_2(F)$.  Let $\psi$ be a nontrivial additive unitary character of $F$.  We denote by $\widehat{G}$ the set of isomorphism classes of irreducible representations, $\widehat{G}_{\mathrm{gen}}$ the subset of generic representations and $\widehat{G}^{\mathrm{u}}_{\mathrm{gen}}$ the subset of unitary generic (with respect to $\psi$) representations.  We equip $\widehat{G}$ with the Plancherel measure $d\pi$. All relevant notation and normalization of measures will be explained in \S~\ref{sec21} and \ref{sec22}.  For $ \pi \in \widehat{G}_{\mathrm{gen}}$, we denote by
\begin{equation}\label{gammafactor}
  \gamma(s, \pi, \psi) = \eps(s, \pi, \psi) \frac{L(1 - s, \tilde{\pi})}{L(s, \pi)}
\end{equation}
the local $\gamma$-factor, with the standard $\epsilon$- and $L$-factor, cf.\ e.g.\ \cite[\S3]{Cog}.   We do not vary $\psi$, and so drop it from our notation and write $\gamma(s, \pi) := \gamma(s, \pi, \psi)$. For $i = 1, 2$ let $\pi_i \in \widehat{G}^{\mathrm{u}}_{\mathrm{gen}}$ such that
\begin{itemize}
\item $\pi_1$ is $\vartheta_1$-tempered (see \S~\ref{sec24}), and
\item $\pi_2$ is a $\vartheta_2$-tempered principal series representation (see Remark~\ref{remark:cq2zor9f44}), given by the unitarily normalized induction $\I(\chi_2)$ of some character $\chi_2$ of $F^\times$,
\end{itemize}
for some $0 \leq \vartheta_1, \vartheta_2 < 1/2$. 
We realize $\pi_i$ in its $\psi$-Whittaker model, and let $W_i \in \pi_i$ be a smooth vector.  We denote by $f_2\in\I(\chi_2)$ the vector in the induced model corresponding to $W_2$ via the standard intertwiner, defined for $\chi_2$ of large enough real part by the Jacquet integral
\begin{equation}\label{defn-intertwiner}
  W_2(g) = \int _{N} f_2 (w n g) \psi^{-1}(n) \, d n
\end{equation}
and in general by analytic continuation along holomorphic sections.

Given $W_1 \in \mathcal{W}(\pi_1, \psi)$ and $W_2 \in \mathcal{W}(\pi_2, \psi)$, we define a pair of local weights $h = h_{W_1 \otimes \overline{W}_2}$ and $h^\vee= h^\vee_{W_1 \otimes \overline{W}_2}$, fitting into a diagram (cf.\ \cite[\S2]{Ne})
\begin{equation*}
  \begin{tikzcd}[ampersand replacement=\&]
    \mathcal{W}(\pi_1, \psi) \otimes \mathcal{W}(\pi_2, \psi)
    \arrow[d]
    \arrow[rr,left]
    \&\& \mathcal{W}(\pi_1, \psi) \otimes \mathcal{I}(\chi_2) 
    \arrow[d]
    \\
    \{ h : F^\times \times F^\times \rightarrow \mathbb{C} \}
    \arrow[rr,leftrightarrow, dashed] \&\&
    \{ h^\vee : \widehat{G}_{\gen}^{\mathrm{u}} \times F^\times \rightarrow \mathbb{C} \}.
  \end{tikzcd}
\end{equation*}
In this diagram, 
\begin{equation}\label{def-h}
  h(y_1, y_2) := W_1(a(y_1)) \overline{W_2(a(y_2))},\quad a(y):=\text{diag}(y, 1), 
\end{equation}
is the local version of the weight function in the shifted convolution problem (e.g., as in \cite[(3)]{BH}).  The top horizontal arrow is given by the inverse of the Jacquet integral, say $W_2 \mapsto f_2 \in \mathcal{I}(\chi_2)$.  The local weight
\begin{equation}\label{def-h-vee}
  h^\vee(\pi,y) := \sum_{W\in\B(\pi)}W(a(y))\int_{N\backslash G}W_1(g)\overline{f_2(g)W(g)} \, d g
\end{equation}
arises naturally in the spectral theory approach to the shifted convolution problem \cite{Mo4, BH}, which consists of multiplying a pair of modular forms, spectrally expanding their product, and taking a Whittaker coefficient of that expansion (see \S~\ref{sec36}):
\begin{multline*}
  \sum \frac{\lambda_1(n) \lambda_2(n+b)}{\lvert n(n + b) \rvert^{1/2}} W_1(n) W_2(n + b) \\
  =
  \int e(-b x) \varphi_1 \varphi_2(n(x)) \, d x
  =
  \sum_{\Phi}
  \langle \varphi_1 \varphi_2, \Phi \rangle
  \int \Phi(n(x)) e(-b x) \, d x + \dotsb.
\end{multline*}
From this, expressions like \eqref{def-h-vee} arise as local factors in the setting we consider, thanks to the uniqueness of Whittaker functionals and invariant trilinear forms.
In formula \eqref{def-h-vee}, $\B(\pi)$ is an orthonormal basis of $\pi$ consisting of vectors isotypic with respect to the standard maximal compact subgroup $K$ of $G$. For $W_1(a(.)), W_2(a(.))\in C_c^\infty(F^\times)$, we will see in \S~\ref{sec25} below that the right-hand side of \eqref{def-h-vee} converges absolutely, and we will investigate its decay properties with respect to $\pi$ and $y$.  The second argument of $h^{\vee}$ plays the role of the shift.

\medskip

Our first result clarifies the bottom horizontal arrow in the diagram. It gives an explicit integral transform $h \mapsto h^{\vee}$, and inversion formula, in terms of $\gamma$-factors.

\begin{theorem}\label{thm1}
  Let $\pi_1, \pi_2 \in \widehat{G}^{\mathrm{u}}_{\mathrm{gen}}$ satisfy the above assumptions, 
  let $W_i(a(.))\in C_c^\infty(F^\times)$ and define $h$ by \eqref{def-h}. 
  \begin{enumerate}
  \item\label{GL1-to-GL2} Let $\pi \in \widehat{G}_{\mathrm{gen}}^{\mathrm{u}}$ be $\vartheta$-tempered, where $\vartheta_1+\vartheta_2+\vartheta<1/2$, and let $y\in F^\times$. Then  
    \begin{align*}
      h^\vee(\pi,y)&=\int_{\Re(\chi)=\sigma}{\gamma(\tfrac{1}{2},{\pi}\otimes\chi)}{\gamma(1,\pi_1\otimes\overline{\chi_2}^{-1}\otimes{\chi}^{-1})} \\ &\quad \times \int_{F^\times}h\big(y t,y(t-1)\big){\chi}^{-1}(t)\overline{\chi_2} \left(\frac{t-1}{t}\right) \left|\frac{t}{t-1}\right|^{1/2} \frac{dt}{|t|} \, d\chi
    \end{align*}
for $\vartheta_1+\vartheta_2<\sigma<1/2-\vartheta$.    The integrals converge absolutely, and the $\chi$-integrand is holomorphic for $\chi$ in the strip $\vartheta_1+\vartheta_2<\Re(\chi)<1/2-\vartheta$.
    
  \item\label{GL2-to-GL1} Let $y_1,y_2\in F^\times$, with $y_1 \neq y_2$. For any $\vartheta_1+\vartheta_2<\sigma<1/2$, we have
    \begin{align*}
      h(y_1, y_2) &= \frac{\overline{\chi_2(y_1/y_2)}\sqrt{|y_1y_2|}}{|y_1-y_2|} \int_{\Re(\chi)=\sigma}{\chi}\left(1 - \frac{y_2}{y_1}\right)\gamma(1, \pi_1 \otimes \overline{\chi_2} \otimes {\chi}^{-1})
      \\
                  &\quad \times\int_{\widehat{G}^{\text{\rm u}}_{\text{{\rm gen}}}} h^{\vee}(\pi, y_1- y_2)  \gamma(\tfrac{1}{2}, \pi \otimes {\chi})  d\pi  \, d\chi.
    \end{align*}
    Both the inner integral and the outer $\chi$-integral converge absolutely,\footnote{but the integral is not absolutely convergent as a double integral} and the $\chi$-integrand is holomorphic for $\chi$ in the strip $\vartheta_1<\Re(\chi)<1/2$.
  \end{enumerate}
\end{theorem}

\subsubsection{Moments of $L$-functions} We keep the notation from the previous result and turn now to the local integral transform underlying the formula \eqref{fourth}. As mentioned before, \eqref{fourth} follows essentially from summing \eqref{shifted} over the shift $b$.  With this in mind, we introduce a further correspondence
\begin{equation*}
  \begin{tikzcd}[ampersand replacement=\&]
    \{ H : F^\times \times \widehat{F^\times} \rightarrow \mathbb{C} \}
    \arrow[rr,leftrightarrow, dashed] \&\&
    \{ h^\sharp : \widehat{G}_{\gen}^{\mathrm{u}} \times \widehat{F^{\times}}_{\text{unit}} \rightarrow \mathbb{C} \}
  \end{tikzcd}
\end{equation*}
given in terms of the previous one by
\begin{equation}\label{capitalH}
  H(y, \chi) := \int_{F^{\times}} h(z, yz) \chi(z) \,d^{\times} z
\end{equation}
and
\begin{equation}\label{hsharp}
  h^{\sharp}(\pi, \chi) := \int_{F^{\times}} h^{\vee}(\pi,  z) \chi(z) \,d^{\times}z.
\end{equation}
The integral defining $h^\sharp$ is absolutely convergent if $\chi$ is unitary, or more generally if $\Re(\chi)>-1/2+\vartheta$ for $\vartheta$-tempered $\pi$, as follows from Lemma \ref{property-h-vee} below.  The functions $H$ and $h^{\sharp}$ are related by the following explicit local transforms.

\begin{theorem}\label{thm2}
  Let $\pi_1$ be $\vartheta_1$-tempered. Let $W_i(a(.))\in C_c^\infty(F^\times)$ and define $h$ be \eqref{def-h}. 
  \begin{enumerate}
  \item\label{fourth-to-cubic} Let $\pi$ be $\vartheta$-tempered such that $\vartheta_1+\vartheta<1/2$. We have
    \begin{equation}\label{hsh}
    \begin{split}
  h^{\sharp}(\pi, \chi_0)&=\int_{\Re(\chi)=\sigma}\gamma(\tfrac{1}{2},\pi\otimes\chi){\gamma(1,\pi_1\otimes\overline{\chi^{-1}_2}\otimes{\chi}^{-1})}\\
                         &\times\int_{F^\times}H(y,\chi_0)\chi_0(1-y)\chi(1-y)\overline{\chi_2(y)} \frac{\sqrt{|y|}}{|1 - y|}\,d^\times y\,d\chi,
                         \end{split}
\end{equation}
for $\vartheta_1<\sigma<1/2-\vartheta$.  The integrals converge absolutely, and the $\chi$-integrand is holomorphic for $\chi$ in the strip $\vartheta_1<\Re(\chi)<1/2-\vartheta$.

  \item\label{cubic-to-fourth} Let $y\in F^\times$ and $\chi_0\in\widehat{F^\times}$ with $\Re(\chi_0)>-1/2$. We have
    \begin{align*}
H(y, \chi_0)&=\frac{\overline{\chi_2^{-1}(y)}\sqrt{|y|}}{\chi_0(1-y)|1-y|}
    \int_{\Re(\chi)=\sigma}{\chi}(1-y) \gamma(1, \pi_1 \otimes \overline{\chi_2} \otimes {\chi}^{-1}) \\
            &\times \int_{\widehat{G}^{\text{{\rm u}}}_{\text{\rm gen}}} h^{\sharp}(\pi, \chi_0)\gamma(\tfrac{1}{2}, \pi \otimes {\chi})\, d\pi\,d\chi.
    \end{align*}
    for $\vartheta_1<\sigma<1/2$. The integrals converge absolutely, and the $\chi$-integrand is holomorphic for $\chi$ in the strip $\vartheta_1<\Re(\chi)<1/2$.
  \end{enumerate}
\end{theorem}
  If $F = \Bbb{R}$, then the local gamma factors are quotients of gamma functions, and the corresponding integrals become hypergeometric functions.  If $\pi_1, \pi_2$ are principal series corresponding to even Maa{\ss} forms, this is spelt out classically in \eqref{expl1}, \eqref{expl2} below.  In this case, $h^{\vee}$ (resp.\ $h^\sharp$) is essentially the transform \eqref{trafo1} (resp.\ \eqref{trafo}).  We verify this explicitly in the appendix, including all numerical constants, as an exercise in hypergeometric functions.

Theorems \ref{thm1} and \ref{thm2} give definitive answers to Question \ref{Q1}: they provide  explicit integral representations of the local transforms $h^{\vee}$ and $h^{\sharp}$ over any local field, along with inversion formulae.

\subsection{Global spectral decomposition}\label{sec:cnebsqy7eo}

For the second moment of $L$-functions of \emph{holomorphic} cusp forms, i.e., automorphic representations associated to the discrete series, a spectral decomposition was established by Good \cite{Go}, see also \cite{Mo3}. Similarly, the formula \eqref{shifted} admits a direct analogue if the divisor function is replaced with Hecke eigenvalues of \emph{holomorphic} cusp forms, see e.g.\ \cite[Lemma 4]{JM}. The source of success in both cases is the fact that holomorphic cusp forms are linear combinations of Poincar\'e series and Whittaker functions for the discrete series are exponentials that factorize with respect to additive shifts. This type of argument is not available for Maa{\ss} forms.  

The novelty of the present paper is to use local representation theory in a systematic way.  Locally, there is no difference between the Eisenstein case treated by Motohashi and the cuspidal case of Maa{\ss} forms. 
While the global approach of previous works is unable to establish such a result, we combine the local computations of the previous subsection with the global approach developed in \cite{Ne}.

We will employ the following notation. Let $F$ be a number field with integer ring $\mathcal{O}_F$. Let $S$ be a finite set of places, containing all places that are either archimedean or at which $F$ is ramified. We denote by $\widehat{[G]}$ the set of automorphic representations on $[G] := {\rm PGL}_2(F)\backslash {\rm PGL}_2(\Bbb{A})$, by $\widehat{[G]}_{\mathrm{gen}}$ the subset of generic representations, by $\widehat{[G]}^{\mathrm{u}}_{\mathrm{gen}}$ the further subset of unitary representations and by $\widehat{[G]}^{\mathrm{u}}_{\mathrm{gen}, S}$ the further subset of representations that are unramified outside $S$.  We consider a representation $\pi = \otimes \pi_{\mathfrak{p}} \in \widehat{[G]}^{\mathrm{u}}_{\mathrm{gen}, S}$ and a factorizable automorphic form $\varphi \in \pi$, both unramified outside $S$, and denote by $W_{\varphi} = \prod_{\mathfrak{p}} W_{\varphi, \mathfrak{p}}$ the corresponding Whittaker function.  We always normalize this factorization so that for $\mathfrak{p} \notin S$, the local factor $W_{\varphi, \mathfrak{p}}$ is normalized spherical, i.e.,  the unique ${\rm PGL}_2(\mathfrak{o}_{\mathfrak{p}})$-invariant vector taking the value 1 at the identity.  
We denote by
\begin{equation}\label{Wlambda}
   W_{\pi}^{(S)} = \prod_{\mathfrak{p} \not \in S} W_{\varphi, \mathfrak{p}}
\end{equation}
the ``outside $S$'' part of the Whittaker function, which in classical language is essentially the Hecke eigenvalue: for $b \in F^\times$, we have
\begin{equation*}
  W_{\pi}^{(S)} \big(\big( 
  \begin{smallmatrix}
    b &\\ & 1
  \end{smallmatrix}
  \big)\big) =
  \begin{cases}
    \lambda_{\pi}^{(S)}(b)|b|^{-1/2}_S 
    & \text{ if $b$ is integral outside $S$,} \\
    0
    & \text{ otherwise.}
  \end{cases}
\end{equation*}
Here and in other global situations, we often abbreviate $W(y) := W(a(y))$ to lighten the notation.  We also adopt the standard convention of using a subscripted (resp.\ superscripted) $S$ for a product of components inside (resp.\ outside) $S$.

We normalize in \eqref{spec} a global measure $d\pi$ on $\widehat{[G]}_{\mathrm{gen}}$. In particular, $d\pi$ is \emph{half} the counting measure on the cuspidal spectrum.  The local transform $h^{\vee}$ defined in \eqref{def-h-vee} and computed in Theorem \ref{thm1} generalizes immediately to an $S$-adic situation, and we set
\begin{equation}\label{Sh}
h^{\vee}(\pi_S, b_S) := \prod_{\mathfrak{p} \in S} h^{\vee}(\pi_\mathfrak{p}, b_\mathfrak{p}).
\end{equation}

\begin{theorem}\label{thm-shift}
  Let $\pi_1, \pi_2$ be cuspidal automorphic representations for $\PGL_2$ over a number field $F$. 
  Let $S$ be a finite set of places of $F$, containing all archimedean places and all places where $F$, $\pi_1$ or $\pi_2$ ramify.  Assume that $\pi_{2 , \mathfrak{p} } $ belongs to the principal series for each $\mathfrak{p} \in S$.  
  
  Then for all $0 \neq b \in \mathcal{O}_F[1/S]$ and all functions $h \in C_c^{\infty}(F_S^{\times} \times F_S^{\times})$, we have
  \begin{equation*}
    \sum_{n_1 - n_2 = b}
 W_{\pi_1}^{(S)}\big(\big( 
\begin{smallmatrix}
 n_1 &\\ & 1
\end{smallmatrix}
\big)\big) \overline{ W_{\pi_2}^{(S)}\big(\big( 
\begin{smallmatrix}
 n_2&\\ & 1
\end{smallmatrix}
\big)\big) }
    h(n_1,n_2)
    =
    \int_{ \widehat{[G]}^{\text{\rm u}}_{\text{{\rm gen}}, S}
    }
   W_{\pi}^{(S)}\big(\big( 
\begin{smallmatrix}
 b &\\ & 1
\end{smallmatrix}
\big)\big)
    c_\pi^{(S)}
    h^\vee (\pi_S, b_S) \, d \pi,
  \end{equation*}
  where $c_\pi^{(S)}$ satisfies
   $ |c_\pi^{(S)}|
  =   
  |L^{(S)} (\tfrac{1}{2}, \pi _1 \times \overline{\pi_2} \times \pi )|^{1/2}
/ L^{(S)} (1, \Ad, \pi )$
 if $\pi_{2, \mathfrak{p}}$ is tempered for all $\mathfrak{p}\in S$. 
  \end{theorem}

An exact formula for $|c_\pi^{(S)}|^2$ in all cases is given in Lemma \ref{lemma:cj7jjxnm7b} and its proof.  For $F = \Bbb{Q}$, $S = \{\infty\}$, $\pi_1, \pi_2$ two Maa{\ss} forms for ${\rm SL}_2(\Bbb{Z})$, $b \in \Bbb{Z} \setminus \{0\}$ this is just the classical shifted convolution problem:
\begin{equation}\label{overQ}
  \sum_{n_1 - n_2 = b}
  \frac{\lambda_1(n_1 ) \lambda_2(n_2 )}{ \lvert n_1 n_2  \rvert^{1/2} }
  h(n_1,n_2)
  =
  \int_{\widehat{[G]}^{\text{\rm u}}_{\text{{\rm gen}}, \{\infty\}} }    \frac{\lambda_\pi (b)}{ \lvert b \rvert ^{1/2} }
  c_\pi^{\{\infty\}} 
  h^\vee (\pi , b ) \, d \pi
\end{equation}
where $\pi$ runs over the discrete set of holomorphic cusp forms, the Maa{\ss} cusp forms and the continuous family of Eisenstein series for the group ${\rm SL}_2(\Bbb{Z})$, and $\lambda_1, \lambda_2$ are the usual Hecke eigenvalues of $\pi_1, \pi_2$. The key point here is that the right hand features the integral transform $h^{\vee}(\pi, b)$ given in Theorem \ref{thm1}(1). As mentioned before, despite a lot of work on the shifted convolution problem, an explicit spectral decomposition as in \eqref{overQ} is new even in this special case.

Theorem \ref{thm-shift} can handle much more general situations. To give a basic example, still for $F = \Bbb{Q}$, consider summation conditions of the form $\alpha n_1 + \beta n_2 = d$ with $\alpha, \beta, d \in \Bbb{Z} \setminus\{0\}$ and $(\alpha, \beta) = 1$, which come up often in typical applications (and which lead very naturally to a spectral expansion for the congruence subgroup $\Gamma_0(\alpha\beta)$). To this end, take
\begin{equation*}
  S = \{\infty\} \cup \{ p \mid \alpha\beta\}, \quad b = \frac{d}{\alpha \beta} \in \Bbb{Z}\Big[\frac{1}{\alpha\beta}\Big], \quad h = h_{\infty}\prod_{p \mid \alpha\beta} h_p, \quad h_p = 
  \begin{cases}
 \textbf{1}_{\Bbb{Z}_p^{\times} \times p^{-v_p(\alpha)} \Bbb{Z}_p^{\times}}, & p \mid \alpha,\\\textbf{1}_{p^{-v_p(\beta)}\Bbb{Z}_p^{\times} \times \Bbb{Z}_p^{\times}} & p \mid \beta,
\end{cases}
\end{equation*}
so that the left hand side of the identity in Theorem \ref{thm-shift} reads
\begin{equation*}
  \underset{    n_1 - n_2 = d/(\alpha\beta)  }
  {\underset{ (n_2\alpha n_1 \beta, \alpha\beta) = 1}
      { {\underset{n_1 \in \beta^{-1}\Bbb{Z} \, n_2 \in \alpha^{-1}\Bbb{Z} }
          {\sum \sum}}}}
  \frac{\lambda_1^{(S)} (n_1 ) \lambda_2^{(S)} (n_2 )}{ \lvert n_1 n_2  \rvert_S^{1/2} }
  h_{\infty}(n_1,n_2)  =  
  \underset{\alpha \nu_1 - \beta \nu_2 = d}
  {\underset{(\nu_1\nu_2 , \alpha\beta) = 1}
   {
        {\sum_{\nu_1 \in \Bbb{Z} } \sum_{\nu_2\in  \Bbb{Z}} }}}
  \frac{\lambda_1(\nu_1 )   \lambda_2(\nu_2 )}{ \lvert \nu_1 \nu_2   \rvert^{1/2} }
  h_{\infty}\Big(\frac{\nu_1}{\beta},\frac{\nu_2}{\alpha}\Big) 
\end{equation*}
since $\lambda^{(S)}(\nu/\beta) = \lambda^{(S)}(\nu) = \lambda(\nu)$ for $S = \{\infty\} \cup \{ p \mid \alpha\beta\}$ and $(\nu, \alpha\beta) = 1$. Of course, the functions $h_p$ can encode much more complicated conditions modulo powers of $p$.  \\

We now turn to a similar spectral decomposition for the second moment of automorphic $L$-functions. We denote by $\widehat{[A]}^{\mathrm{u}}_{S}$ the set of unitary characters of $\Bbb{A}^\times/F^{\times}$ unramified outside $S$. The definitions \eqref{capitalH} and \eqref{hsharp} extend to the $S$-adic setting, and also to characters in $\widehat{[A]}^{\mathrm{u}}_{S}$ via the inclusion $F_S^\times \hookrightarrow \mathbb{A}^\times$. 
We denote by
\begin{equation*}
  \mathcal{J} : C_c^\infty(F_S^\times \times F_S^\times) \rightarrow C^\infty(F_S^\times \times F_S^\times)
\end{equation*}
the continuous linear map given on pure tensors $h(y_1, y_2) = W_1(a(y_1)) \overline{W_2(a(y_2))}$, with $W_1, W_2 \in C_c^\infty(F_S^\times)$, corresponding to elements of the Kirillov models of $\pi_1$ and $\pi_2$, respectively,  by
\begin{equation*}
  \mathcal{J}(h)(y_1, y_2) :=   W_1(a(y_1)w)  \overline{ W_2(a(y_2)w)}
\end{equation*}
with  $w=\left(
  \begin{smallmatrix}
    &-1\\
    1& \\
  \end{smallmatrix}
\right)$.  The map $\mathcal{J}$ may be written as an explicit integral transform with kernel given by the tensor product of Bessel functions of $\pi_1$ and $\pi_2$ (see \cite[Section 4.5]{CPS}).

\begin{theorem}\label{thm-moment}
  Let $F$, $\pi_1$, $\pi_2$ and $S$ satisfy the same assumptions as in Theorem \ref{thm-shift}.  For $\chi \in \widehat{[A]}^{\rm u}_S$ and $h = h_{W_1 \otimes \overline{W}_2} \in C^{\infty}_c(F^{\times}_S \times F^{\times}_S)$, define $H$ as in \eqref{capitalH}, and for $\eta \in \widehat{[A]}^{\rm u}_S$, define
  \begin{equation}\label{wetachi}
    w(\eta, \chi) =  \int_{F^{\times}_S}  H(y, \chi) \eta(y)\,  d^{\times} y.
  \end{equation}

Then
\begin{multline*}
  \int_{
      \eta \in \widehat{[A]}^{{\rm u}}_S
    }
 \frac{    L^{(S)}(\tfrac{1}{2}, \pi_1 \otimes \chi \otimes \eta^{-1} )  L^{(S)}(\tfrac{1}{2}, \overline{\pi_2} \otimes \eta)  }{\underset{s=1}{\text{\rm res}}\zeta_F^{(S)}(s)}
    w(\eta, \chi) \, d \eta = \mathcal{M}  +
    \int_{  \widehat{[G]}_{\text{\rm gen}, S}^{\mathrm{u}}} L^{(S)}(\tfrac{1}{2}, \sigma \otimes \chi ) 
    c_\sigma^{(S)}
    h_S^\sharp (\sigma, \chi )
    \, d \sigma
    \end{multline*}
    where 
\begin{displaymath}
\begin{split}
  \mathcal{M}  & = \frac{L^{(S)}(1, \pi_1 \otimes \overline{\pi_2} \otimes \chi)}{L^{(S)}(2, \chi^2)}
                 \int_{F^{\times}_S} h(z, z) \chi(z) \, d^{\times} z \\
               & + \frac{L^{(S)}(1, \pi_1 \otimes \overline{\pi_2} \otimes \chi^{-1})}{L^{(S)}(2, \chi^{-2})}
                 \int_{F^{\times}_S}
                 \mathcal{J}(h)(z, z)
                 \chi^{-1}(z) \, d^{\times} z.
\end{split}
\end{displaymath}
Here $c^{(S)}_\sigma$ is as in Theorem \ref{thm-shift}, and the right hand side has a removable singularity at $\chi=\text{\rm triv}$ if $\pi_1 = \pi_2$.
\end{theorem}

Theorems \ref{thm-shift} and \ref{thm-moment} give definitive answers to Question \ref{Q2}: they provide explicit spectral decompositions of the shifted convolution problem and the second moment of $L$-functions for Maa{\ss} forms  with the ``correct'' integral transform that was studied in the previous subsection. 

\begin{remark}\label{remark:cngtue4fjm}
  The main term $\mathcal{M}$ in the formula of Theorem \ref{thm-moment} is well-known from a classical perspective. An approximate functional equation for the product $ L^{(S)}(\tfrac{1}{2}, \pi_1 \otimes \chi \otimes \eta^{-1} ) L^{(S)}(\tfrac{1}{2}, \overline{\pi_2} \otimes \eta)$ contains two terms, one of which contains the root number. After integrating over $\eta$, both of them contribute a diagonal term the sum of which is precisely $\mathcal{M}$.
\end{remark}

\begin{remark}\label{remark:cq2zor9f44}
  We briefly discuss our assumption that the components of $\pi_2$ belong to the principal series.  Consider, for instance, the case where $F = \mathbb{Q}$ and both $\pi_{1,\infty}$ and $\pi_{2,\infty}$ belong to the discrete series (the classical holomorphic setting).  This is the historically ``simpler'' case and can be handled directly via Poincaré series.  Our method could also be adapted by embedding $\pi_{2,\infty}$ into an appropriate principal series representation (using the archimedean subrepresentation theorem, as in the proof of \cite[Theorem 3]{BJN}).

  However, this approach is not robust: for the analogous situation at a non-archimedean place with both $\pi_1$ and $\pi_2$ supercuspidal, it is not clear how to establish a completely explicit spectral decomposition.  The obstacle is the absence of an obvious invariant trilinear form on the tensor product of three supercuspidal representations, which would otherwise play the role of the Rankin--Selberg integral.  For this reason, we retain the principal series assumption throughout the paper.
\end{remark}

\subsection{An application}

As a proof of concept, we consider in a classical setting the following application of shifted convolution sums in short windows, an instance where the archimedean test function depends on additional parameters.  This could probably be obtained other ways (cf.\ \cite[Theorem 1.12]{GHH}), but the explicit formula of Theorem \ref{thm-shift} allows a quick and direct proof, without the need to set up a circle method, Voronoi summation, the Kuznetsov formula or any other heavy machinery often employed in this context.

\begin{theorem}\label{appl}
  Let $1 \leq Y \leq X/10$ be two parameters, $1 \leq b \leq
  X/2$ an integer. Let $f, g$ be two fixed (possibly equal) Hecke--Maa{\ss} eigenforms of level 1. Let $V$ be a fixed smooth weight function with support in $[1, 2]$, and let
  \begin{equation*}
\mathcal{S}(X, Y, b) = \mathcal{S}_{f, g, V}(X, Y, b) := \sum_{n} \lambda_f(n+b)\lambda_g(n) V\Big( \frac{n -X}{Y}\Big).
\end{equation*}
  Then
  \begin{equation}\label{1}
    \mathcal{S}(X, Y, b) \ll_{V, f, g, \varepsilon} 
    \frac{X^{1+\varepsilon}}{Y} (Y^{1/2} + b^{1/2}) \min\Big(b^{\theta}, 1 + \frac{Yb^{1/4}}{X}\Big)
  \end{equation}
  and
  \begin{equation}\label{2}
    \int_{X}^{2X} | \mathcal{S}(x, Y, b) |^2 dx \ll_{V, f, g, \varepsilon} b^{2\theta} X^{2 + \varepsilon} \Big(1 + \frac{b}{Y}\Big)
  \end{equation}
  for every $\varepsilon > 0$, with $\theta \leq 7/64$ an approximation towards the Ramanujan--Petersson conjecture.
\end{theorem}

For $b = 1$, the bound \eqref{1} is non-trivial if $Y \geq X^{2/3 +\varepsilon}$, while the bound \eqref{2} is non-trivial if $Y \geq X^{1/2 + \varepsilon}$. The big advantage of Theorems \ref{thm-shift} and \ref{thm-moment} are their explicit shape in terms of all relevant weight functions and integral transforms.  This is particularly useful if $f_1, f_2$ come, for instance, with large spectral parameters, which should have applications elsewhere.  Moreover, since the estimate is based on an explicit formula, one could analyze any further averages (e.g., over $b$).  One could alternatively consider arithmetic progressions instead of short intervals, which should lead to analogous estimates involving the $p$-adic case of Theorem \ref{thm1}.

\section{Local results}\label{sec2}

This section is organized as follows. Some background material is provided in \S~\ref{sec21} -- \ref{sec24}. The key integral formulae are proved in \S~\ref{sec25} and used in the proofs of Theorems \ref{thm1} and \ref{thm2} in \S~\ref{sec26} and \ref{sec27}. Finally, \S~\ref{sec28} provides some explicit computations of the local transforms for $F = \Bbb{R}$ in classical language.

\subsection{Notation}\label{sec21} Let $F$ be a local field of characteristic zero, archimedean or non-archimedean.  We denote by $|.|$ the normalized absolute value on $F$. 
When $F$ is non-archimedean, we denote by $\mathfrak{o}$ its ring of integers and $q$ the cardinality of its residue class field. 

Set $G:=\PGL_2(F)$. 
Let $A$ (resp.\ $B$, $N$) denote\footnote{We will occasionally use $N$ also to denote an integral exponent, often sufficiently large; this notational overload should introduce no confusion.} the diagonal (resp.\ upper-triangular, unipotent upper-triangular) subgroup of $G$, and $K$ the standard maximal compact subgroup.  We identify $N$ with $F$ via $F\ni x\mapsto n(x):=(
\begin{smallmatrix}
  1&x\\&1
\end{smallmatrix}
)$ and $A$ with $F^\times$ via the map $F^\times \ni y\mapsto a(y):=(
\begin{smallmatrix}
  y&\\&1
\end{smallmatrix}
)$.
We write $n'(x) = (
\begin{smallmatrix}
  1&\\x&1
\end{smallmatrix}
)$.  
We denote by $w$ the long Weyl element $\left(
  \begin{smallmatrix}
    0  &-1\\
    1&0 \\
  \end{smallmatrix}
\right)$ in $G$.

We fix a non-degenerate additive character $\psi$ of $F$, which we also regard as a character of $N$.

In this paper, $\chi$ will usually denote a character of $F^\times$, and we denote the space of characters by $\widehat{F^{\times}}$.  Any $\chi \in\widehat{F^{\times}}$ is of the form $\chi_{\mathrm{unit}}|.|^\sigma$ for some $\sigma\in \mathbb{R}$, which we will denote by $\Re(\chi)$, and some unitary character $\chi_{\mathrm{unit}}$ of $F^\times$.
We regard $\{\chi : \Re(\chi) = \sigma\}$ as a principal homogeneous space for the Pontryagin dual of $F^\times$. 

We define $\widehat{G}$, $\widehat{G}_{\mathrm{gen}}$ and $\widehat{G}_{\mathrm{gen}}^{\mathrm{u}}$ as in \S  \ref{sec:d1be073b0aa9}.

\subsection{Measures}\label{sec22}  We use 
$\psi$ to normalize the Haar measure $d x$ on 
$F$, so that it is self-dual with respect to $\psi$.  We equip 
$F^\times$ with the Haar measure $d^\times y := |y|^{-1} \, d y$.  If $F$ is non-archimedean and $\psi$ is unramified then 
the volume of $\mathfrak{o}^\times$ is $\zeta_F(1)^{-1} = (1 - q^{-1})$. 
These measures transfer to $N$ and $A$. We obtain a Haar measure $dg$ on $G$ by pushforward under the multiplication map
\begin{equation*}
  N \times N \times A \rightarrow G, \quad (n_1, n_2, a) \mapsto n_1 w n_2 a.
\end{equation*}
Permuting the order of the last three factors introduces a Jacobian factor, giving the integral formula
\begin{equation*}
  \int_{G(F)} f = \int_{  F }\int_F
  \int_{  F^\times }
  f (n (x_1 )  a(y) w n (x_2 )) \, \frac{d^{\times} y}{\lvert y \rvert} \, d x _1 \, d x _2 .
\end{equation*}
which descends to an integral formula over $N\backslash G$. 

Suppose for the moment that $F$ is non-archimedean and $\psi$ is unramified.  Then, there are other natural choices of measures on $G$, namely:
\begin{itemize}
\item The Haar measure $d g'$, defined similarly, but using the Haar measure on $A$ given by $\zeta_F(1) d^\times y$, which is normalized to assign volume one to $\mathfrak{o}_{\mathfrak{p} }^\times $.  Clearly 
  $d g ' = \zeta_F(1) d g.$ 
\item The Haar measure $d g ''$, normalized by requiring that it assign volume one to the maximal compact subgroup $G(\mathfrak{o})$.
\end{itemize}
It is verified in [MV, \S3.1.6] that
\begin{equation*}
d g ' = \frac{\zeta_F(1)}{\zeta_F(2)} \, d g ''. 
\end{equation*}
It follows that
\begin{equation}\label{eq:cj8owhrlnp}
  d g =
  \frac{1}{\zeta_F  (1) } \, d g '
=  
\frac{1}{ \zeta_F  (2)} \, d g''.
\end{equation}
In particular, our chosen Haar measure $d g$ on $G$ satisfies (for unramified $\psi$)
\begin{equation*}
\vol ({\rm PGL}_2(\mathfrak{o}), d g) = \frac{1}{\zeta_F(2)}.
\end{equation*}

We equip the character space $\{\chi : \Re(\chi) = \sigma\}$ with the Haar measure $d\chi$ dual to $d^\times y$, so that the corresponding Parseval formula holds. For instance, if $F = \Bbb{R}$, $\psi$ is the standard additive character and $f$ is a function on $\widehat{F^{\times}}$, then
\begin{equation}\label{measurechi}
  \int_{\Re (\chi) = \sigma} f(\chi) \, d\chi =  \frac{1}{2} \sum_{\delta\in \{0, 1\}} \int_{(\sigma)} f(\text{sgn}^{\delta} |.|^{s}) \frac{ds}{2\pi i}.
\end{equation}

We equip $\widehat{G}$ with the local Plancherel measure $d\pi$, supported on the tempered part of $\widehat{G}_{\mathrm{gen}}^{\mathrm{u}}$, that is compatible with $dg$ in the sense that for each $f \in C_c^\infty(G) $, we have 
\begin{equation*}  
  f(1) = \int_{ \widehat{G}_{\mathrm{gen}}^{\mathrm{u}}} \tr(\pi(f)) \, d \pi .
\end{equation*}
The measure $d \pi$ may be described explicitly in terms of local $\gamma$-factors, and in particular, grows at most polynomially in all relevant parameters (e.g., the conductor of $\pi$).

\subsection{Schwartz spaces}

Let $\D$ denote the generalized Laplacian on $G$ as considered in \cite[\S2.3.3]{MV}.  If $F$ is archimedean, then $\D$ can be taken as the differential operator $\sum (1-X_i^2)$, where $\{X_i\}$ is an $\mathbb{R}$-basis of the Lie algebra $\mathfrak{pgl}_2(F)$. If $F$ is non-archimedean, then we denote $K[m]$ to be the standard principal congruence subgroup of $G$ of level $m$. In this case, $\D$ can be taken as $\sum_{m\ge 0} q^me[m]$ where $e[m]$ is the projector on the orthogonal complement of $K[m-1]$-invariant vectors (or functions) inside $K[m]$-invariant vectors (or functions).

As noted in \cite[\S3.1.9]{MV}, the operator $\D$ is self-adjoint (clear by construction for non-archimedean $F$, proved in \cite{NS} for archimedean $F$).  It is positive-definite on unitary representations.  It admits a bounded inverse $\mathcal{D}^{-1}$.  For $N$ sufficiently large in terms of $G$, the inverse power $\D^{-N}$ is of trace class (see \cite[\S2.4.1, (S1c)]{MV}, whose proof is implicit in \cite[\S2.6.3]{MV}; for a more detailed proof in the archimedean case, see \cite[Lemma A.3]{NV}).

If $\pi$ is a unitary representation of $G$, then we define the $d$-th Sobolev norm on $v\in \pi$ by $S_d(v):=\|\D^dv\|_\pi$.

We will speak of ``invariant differential operators'' $D$ on the group $F^\times$ (resp.\ $F$).  When $F$ is archimedean, this carries the usual meaning (e.g., on $\mathbb{R}$, we get the standard derivative map $\partial_x$ and its powers, while on $\mathbb{R}^\times $, we get $y \partial_y$ and its powers).  When $F$ is non-archimedean, we adopt the convention that $D$ is a power of the operator $\sum_{m\ge 0}q^m\tilde{e}[m]$, where $\tilde{e}[m]$ denotes the orthogonal projector on the orthogonal complement of $1+\mathfrak{p}^{m-1}$-invariant (resp.\ $\mathfrak{p}^{m-1}$-invariant) functions inside $1+\mathfrak{p}^{m}$-invariant (resp.\ $\mathfrak{p}^m$-invariant) functions.

We record a standard Paley--Wiener type result concerning the Mellin transform, which gives in particular conditions under which Mellin inversion holds.
\begin{lemma}
  Let $\delta > 0$. The Mellin transform
  \begin{equation*}
    f \mapsto \Big(\chi \mapsto \int_{F^\times} f(y) \chi^{-1}(y)\, d^\times y\Big)
  \end{equation*}
  provides an isomorphism between
  \begin{enumerate}
    [(i)]
  \item the space of all $f \in C^{\infty}(F^{\times})$ with the property that
    \begin{equation}\label{eq:cq2z3yiz5b}
      D f(y) \ll_{D, \sigma} \min(\lvert y \rvert^{\sigma}, \lvert y \rvert^{-\sigma})
    \end{equation}
    for each $0 < \sigma < \delta$ and each invariant differential operator $D$ on $F^\times$, and
  \item the space of all complex-valued functions
    \begin{equation}\label{eq:cq2z3yxxcu}
      \left\{
        \text{characters } \chi : F^\times \rightarrow \mathbb{C} 
      \right\}
      \rightarrow \mathbb{C} 
    \end{equation}
    that are holomorphic in $|\Re(\chi)| < \delta$ and rapidly decaying in $\chi$, where ``rapidly decaying'' means decaying faster than any inverse power of the analytic conductor of $\chi$.
  \end{enumerate}
  The usual Mellin inversion formula holds.
\end{lemma}
\begin{proof}
  Suppose \eqref{eq:cq2z3yiz5b} holds.  From the case $D = 1$, we see that the integral defining the Mellin transform converges absolutely, uniformly for $\Re(\chi)$ in each compact subset of the interval $(-\delta, \delta)$.  On the other hand, we can find $D$ that multiplies the Mellin coefficient at $\chi$ by a scalar comparable to some power of the analytic conductor of $\chi$.  (Take $D = 1 - (y \partial_y)^2$ or $\sum_{m \geq 0} q^m \tilde{e}[m]$ according as $F$ is archimedean or not.)  Applying \eqref{eq:cq2z3yiz5b} with a large power of that $D$ gives the required rapid decay.  In particular, both $f$ and its Mellin transform are integrable, so by the standard result concerning $L^1$-Fourier inversion, we deduce that the Mellin inversion formula holds.

  Conversely, given a function \eqref{eq:cq2z3yxxcu} satisfying the indicated conditions, we may define $f$ via the Mellin inversion formula using any contour $\Re(\chi) = \sigma$ with $\sigma \in (-\delta, \delta)$.  By arguments like in the previous paragraph, we see that \eqref{eq:cq2z3yiz5b} holds.
\end{proof}

For $\sigma > 0$, we define a space $\mathcal{V}_{\sigma}((N, \psi) \backslash G)$ of functions $f : G \rightarrow \mathbb{C}$, as follows.  In the non-archimedean case, this space consists of functions that are right-invariant under some open subgroup of $G$, satisfy the transformation rule $f(ng) = \psi(n) f(g)$ for all $n \in N, g \in G$, and satisfy the decay bound
\begin{equation}\label{sigmabound}
  f(a(y) k) \ll_A (1 + |y|)^{-A} |y|^{\sigma},\quad k\in K.
\end{equation}
In the archimedean case, it consists of all smooth functions $f$ that satisfy the same transformation rule and for which the image $X f$ of $f$ under each fixed invariant differential operator $X$ (i.e., element of the universal enveloping algebra of $G$, acting via right convolution) satisfies the decay bound \eqref{sigmabound}.

We remark that, in the archimedean case, the space $\mathcal{V}_{\sigma} ((N, \psi ) \backslash G)$ for $\sigma > 1/2$ is a subspace of the ``Harish--Chandra Schwartz space adapted to Whittaker models'' in the sense of \cite[\S4]{WallachArxiv}.

It is convenient to translate the bound \eqref{sigmabound} into Bruhat coordinates, where it reads
\begin{equation}\label{sigmaboundBruhat}
  f(a(y)wn(x) ) \ll_A \Big(1 + \frac{|y|}{1+|x|^2}\Big)^{-A} \Big(\frac{|y|}{1+|x|^2}\Big)^{\sigma}.
\end{equation}

\subsection{Representations and models}\label{sec24}

Fix $\vartheta \in [0,1/2)$.  We say that an irreducible representation $\pi$ of $G$ is \emph{$\vartheta$-tempered} if it lies in the discrete series or if, writing $\pi$ as a Langlands quotient of an isobaric sum $\sigma_1 \otimes |\det|^{s_1} \boxplus \sigma_2 \otimes |\det|^{s_2}$ with unitary characters $\sigma_1$ and $\sigma_2 $, we have that each $|\Re(s_i)| \leq \vartheta$.  Then $\pi$ is $0$-tempered in the above sense if and only if it is tempered in the usual sense, i.e., its matrix coefficients lie in $L^{2+\eps}(G)$ for each $\eps > 0$. In what follows, we work exclusively with smooth vectors in such representations.  Thus ``let $v \in \pi$'' is shorthand for ``let $v$ be a smooth vector in $\pi$.''

We recall that a representation $\pi$ of $G$ is \emph{generic} if there is a nonzero $G$-equivariant map $\pi \rightarrow \mathrm{Ind}_N^G\psi$; in that case, provided that $\pi$ is irreducible or belongs to the principal series, the space of such embeddings is one-dimensional.  We may fix a nonzero map, and refer to its image as the \emph{Whittaker model} of $\pi$.  When $\pi$ is irreducible, this map is of course an embedding, and we may identify $\pi$ with its image.  In all cases, we denote a typical element of the Whittaker model of $\pi$ by $W$.  It lies in the space $C^{\infty}((N, \psi)\backslash G)$ of smooth functions satisfying the transformation property
\begin{equation*}
  W(ng)=\psi(n)W(g),\quad n\in N, g\in G.
\end{equation*}
We will require growth estimates for Whittaker function at various boundaries.

\begin{lemma}\label{bound-whittaker}
  Fix $\eta > 0$, $\vartheta < 1/2$ and $N \geq 0$.  Fix invariant differential operators $D_y$ and $D_z$ on $F^\times$ and $F$, respectively.  Let $z \ll 1$.  Let $\pi \in \widehat{G}_{\text{\rm gen}}^{\text{\rm u}}$ be $\vartheta$-tempered.  Then
  \begin{equation*}
    D_y D_z W(a(y)wn(z)), \, D_y D_z W(a(y)n'(z))   \ll \min(|y|^{1/2-\vartheta - \eta}, |y|^{-N})S_d(W),
  \end{equation*}
  where
  \begin{itemize}
  \item the Sobolev norm index $d$ depends at most upon $D_y, D_z,N$ and, in the non-archimedean case, the degree $[\kappa : \kappa_0]$ of the residue field $\kappa$ of $F$ over its prime subfield $\kappa_0$, and
  \item the implied constant depends at most upon $D_y, D_z, N,\eta$.
  \end{itemize}
\end{lemma}

\begin{proof}
  \cite[\S2.4.1 and \S3.2.3]{MV}.
\end{proof}

In particular, Whittaker functions lie in $\mathcal{V}_{\sigma}((N, \psi)\backslash G)$ for 
$\sigma > 1/2 - \vartheta$.

For $\pi \in \widehat{G}_{\gen}^{\mathrm{u}}$, we can define an invariant inner product on $\pi$ by integrating in the Kirillov model, namely, by the formula
\begin{equation}\label{eq:cnebswcf2w}
  \langle W_1, W_2\rangle:=\int_{F^\times}W_1(a(y))\overline{W_2(a(y))}\,d^\times y.
\end{equation}
In particular, when we speak of orthonormal bases for $\pi$, we refer to those defined by this inner product.

Let $\chi$ be a character of $F^\times$, which we extend naturally to $NA$. We denote by $\I(\chi)$ the normalized parabolic induction of $\chi$ from $N A$ to $G$. If $f\in\I(\chi)$, then we have
\begin{equation}\label{transformation-f}
  f(n(x)a(y)g) = \chi(y)|y|^{1/2}f(g).
\end{equation}
An explicit intertwiner $f \mapsto W_f$ from $\I(\chi)$ to its Whittaker model is given as in \eqref{defn-intertwiner}, and can be recast as
\begin{equation}\label{defn-intertwiner-explicit}
  W_f(a(y)) = \lvert y \rvert^{1/2} \chi^{-1}(y) \int_{ F} f( w n(x)) \psi^{-1}(n(x y)) \, d x. 
\end{equation}

\begin{lemma}\label{principal-series-expansion}
  For $f\in\I(\chi)$, we have
  \begin{equation}\label{eqn:d1be069ddc66}
    f(wn(x))=\int_{F^\times}W_f(a(y))\chi(y)|y|^{-1/2}\psi(n(xy)) \, d y.
  \end{equation}
  The integral converges absolutely for $\Re(\chi)>-1/2$.  Moreover, if $F^\times \ni y \mapsto W_f(a(y)) \chi(y) |y|^{-1/2}$ extends to a smooth function on $F$, then $f(wn(.))$ is a Schwartz--Bruhat function on $F$.
\end{lemma}

\begin{proof}
  The absolute convergence of the integral follows from Lemma \ref{bound-whittaker}.  The formula \eqref{eqn:d1be069ddc66} is verified in \cite[eq.\ (7.4)]{Ne}.  For the final assertion, we use that the Fourier transform preserves the space of Schwartz--Bruhat functions.
\end{proof}

An important tool is the local functional equation for $\GL(2)\times\GL(1)$; we refer to \cite[Theorems 3.2 and 3.6]{Cog}.  In this paper, we use the following particular form of the local functional equation. We recall the local gamma factor in \eqref{gammafactor}. 

\begin{lemma}\label{local-functional-equation}
  For $x\in F$, $W\in \pi$ and $\chi$ a character of $F^\times$ we have  
  \begin{equation*}
    \int_{F^\times}W(a(y)wn(x))\chi^{-1}(y)\, d^\times y=\gamma(1/2,\pi\otimes\chi)
    \int_{F^\times}W(a(y))\psi(n(xy))\chi(y) \, d^\times y .
  \end{equation*}
  The above integrals converge absolutely for $-1/2 + \vartheta < \Re(\chi) < 1/2 - \vartheta$, in which case the gamma factor is polefree.
\end{lemma}

\begin{proof}
  The proof follows from what is stated in the cited reference using the formula $\widetilde{W}(g):=W(wg^{-t})$ and left $\psi$-equivariance of $W$.
\end{proof}

\subsection{Whittaker inversion and Bessel distribution}\label{sec25}
We record the Whittaker--Plancherel formula for $G$: 
for $\phi \in \mathcal{V}_{\sigma}((N, \psi)\backslash G)$ with $\sigma > 1/2$, we have
\begin{equation}\label{whittaker-plancherel}
  \phi(h)=\int_{\widehat{G}_{\mathrm{gen}}^{\mathrm{u}}}\sum_{W\in \mathcal{B}(\pi)}W(h) \left( \int_{N\backslash G}\phi(g)\overline{W(g)} \, d g \right)  \, d\pi
\end{equation}
where $\mathcal{B}(\pi)$ is an orthonormal basis of $\pi$ consisting of $K$-isotypic vectors. Moreover, $d\pi$ is supported on the tempered subset of $\widehat{G}^{\mathrm{u}}_{\mathrm{gen}}$.

We first justify that for $\sigma > 1/2$, the right hand side of \eqref{whittaker-plancherel} converges absolutely for all $\phi \in \mathcal{V}_\sigma ((N, \psi ) \backslash G)$, because we need this argument at several places.

Indeed, by Lemma \ref{bound-whittaker} with $\vartheta = 0$ and \eqref{sigmabound}, we see that the inner integral converges absolutely.  By repeated partial integration (using the operator $\mathcal{D}$), we see that the inner integral is moreover rapidly decaying as a function of (the conductor of) $\pi$ and (the Casimir eigenvalue or depth of) the $K$-type, hence the outer integral-sum is again absolutely convergent.

More intrinsically, we may write $\int_{N \backslash G} \phi \overline{W} = \langle V, W \rangle$ for some smooth vector $V \in \pi$; by the pointwise expansion $V(h) = \sum_{W \in \mathcal{B}(\pi)} \langle V, W \rangle W(h)$, we see that the sum equals $V(h)$. This is independent of the choice of basis, provided the expression converges absolutely. This holds in particular for $G$-translates of such bases, since these are isotypic with respect to a conjugate of $K$.  We will use this observation in Lemma \ref{bessel-transform} below.

Once we have justified the absolute convergence, we see that both sides of \eqref{whittaker-plancherel} define continuous linear functionals on that space. Since $K$-finite vectors are dense in $\mathcal{V}_\sigma ((N, \psi ) \backslash G)$, it suffices to verify the formula \eqref{whittaker-plancherel} for $K$-finite functions.  The required formula is then established for archimedean $F$ in \cite[Chapter 15]{Wal} (updated in \cite[Theorem 61]{WallachArxiv}) and for non-archimedean $F$ in \cite[p. 160]{De}.

\begin{remark}
  The cited references establish pointwise Whittaker--Plancherel theorems of the required shape (and in much greater generality), but it would be tedious to check that their normalizations match precisely with ours.  As a check that this is indeed the case, we appeal to the arguments of Sakellaridis--Venkatesh \cite[Theorem 6.3.4]{SV}, who give a short deduction of an $L^2$-Whittaker--Plancherel theorem (for non-archimedean $F$) from the Plancherel theorem for $L^2(G)$.  Their arguments validate \eqref{whittaker-plancherel} provided that the unitary structure on each $\pi$ is normalized so that we have the following evaluations of regularized integrals:
  \begin{equation*}
    \int_{N}^* \psi^{-1}(x) \langle n(x) W, W \rangle \, d x
    = \lvert W(1) \rvert^2.
  \end{equation*}
  That such identities hold in our context follows, via a Fourier inversion argument noted in the second paragraph of the proof of \cite[Lemma 3.4.2]{MV}, from the fact that we have normalized the unitary structure to be given in the Kirillov model with respect to the measure $\frac{d y}{\lvert y \rvert}$.
\end{remark}

By the previous discussion, we conclude the following. 
\begin{lemma}\label{bessel-transform}
  Let $\phi\in \mathcal{V}_{\sigma}((N, \psi)\backslash G)$, with $\sigma > 1/2$. 
  Let $h\in G$ and $\pi\in \widehat{G}_{\text{\rm gen}}^{\text{\rm u}}$. The sum
  \begin{equation*}
    \sum_{W\in\B(\pi)}W(h)\int_{N\backslash G}\phi(g)\overline{W(g)}dg
  \end{equation*}
  converges absolutely for any $K$-isotypic orthonormal basis and equals
  \begin{equation*}
    \sum_{W\in\B(\pi)}W(1)\int_{N\backslash G}\phi(gh)\overline{W(g)}dg.
  \end{equation*}
\end{lemma}

\begin{proof}
  This follows by changing the basis to $\{\pi(h)W\}_{W\in\B(\pi)}$ and a change of variable $g\mapsto gh$.
\end{proof}

Together with Lemma \ref{bound-whittaker}, this argument shows in particular the following useful upper bound. The precise notion of the conductor $C(\pi)$ in the following lemma is not relevant; we can, for instance, take the definition in \cite[\S3.1.8]{MV}. 

\begin{lemma}\label{property-h-vee}
  The quantity $h^{\vee}(\pi,z)$, defined in \eqref{def-h-vee} is rapidly decaying in $\pi$, uniformly in $z$, in the following sense: for any invariant differential operator $D_z$ on $F$ and any $L,N, \eta>0$, we have 
  \begin{equation*}
    D_z h^{\vee}(\pi,z)
    \ll_{L, D_z ,N, \eta, W_1,f_2}C(\pi)^{-L} \min(|z|^{1/2-\vartheta-\eta},|z|^{-N})
  \end{equation*}
  provided that $\pi$ is $\vartheta$-tempered.
\end{lemma}

From Lemma \ref{bessel-transform}, we see that the sum
\begin{equation*}
  \sum_{W\in\B(\pi)}W(g)\overline{W(1)}
\end{equation*}
defines a distribution, which we call the \emph{Bessel distribution} attached to $\pi$.

The following proposition, due to Cogdell and Piatetski-Shapiro \cite[Prop 8.3]{CPS}, describes the Mellin transform of the Bessel distribution.  Since their formulation is slightly different from ours, we record a proof to avoid translation issues, without claiming any novelty.

\begin{prop}\label{mellin-bessel}
  Let $0 \leq \vartheta, \vartheta_0 $, with $\vartheta + \vartheta_0 < 1/2$.  Let $\pi$ be $\vartheta$-tempered and $\phi\in \mathcal{V}_{1-\vartheta_0}((N, \psi)\backslash G)$. 
  Then 
  \begin{equation*}
    \begin{split}
      &\sum_{W\in\B(\pi)}W(1)\int_{N\backslash G}\phi(g)\overline{W(g)} dg\\
      &=\int_{\Re(\chi)=\sigma}{\gamma(1/2,{\pi}\otimes \chi)}
        \left(\int_{F}\int_{F^\times}\phi(a(y)wn(x))\chi(y)\psi(n(-x))\frac{d^\times y}{|y|}dx\right) \, d\chi.
    \end{split}
  \end{equation*}
  Here all the integrals and the sum converge absolutely in the indicated order for $\vartheta_0 < \sigma < 1/2 - \vartheta$.
\end{prop}

\begin{remark}
  The double integral over $x$ and $y$ may be understood as an integral over $N \backslash G$.  More precisely, it is $\int_{N\backslash G} \phi(g) J_{\psi, \chi}(g) $, where $J_{\psi,\chi}$ is the function defined on (an open subset of) $N\backslash G$ by $J_{\psi, \chi}(a(y) w n(x)) = \chi(y)\psi(-x)$.
\end{remark}

\begin{proof}
  The absolute convergence of the $F\times F^\times$-integral on the right hand side follows from \eqref{sigmaboundBruhat} and direct computation: for $\vartheta_0 < \sigma < 1/2$, we have
  \begin{displaymath}
    \begin{split}
      &\int_F \int_{F^{\times}} \min\Big(\Big( \frac{|y|}{1+|x|^2}\Big)^{-10}, \Big(\frac{|y|}{1+|x|^2}\Big)^{1-\vartheta_0}\Big) |y|^{\sigma} \frac{dy}{|y|^2} dx\\
      & = \int_F \Big( \int_{|y| \leq 1 + |x|^2}\frac{ |y|^{\sigma - \vartheta_0 - 1}}{(1 + |x|^2)^{1-\vartheta_0}} dy +   \int_{|y| >1 + |x|^2}\frac{ |y|^{\sigma - 12}}{(1 + |x|^2)^{10}} dy\Big) dx \ll  \int_F (1 + |x|^2)^{\sigma - 1} dx < \infty. 
    \end{split}
  \end{displaymath}

  Moreover, by integrating by parts in $y$ in the archimedean case and applying smoothness of $\phi$ in the non-archimedean case, we see that the $\chi$-integrand is rapidly decaying. Thus the $\chi$-integral is absolutely convergent, as the gamma factor is polynomially bounded in $\chi$ for $\Re(\chi) < 1/2 - \vartheta$ (i.e., away from poles). The left hand side converges absolutely by the argument in the previous subsection.

  We define $\phi_1\in C^\infty(G)$ by requiring that
  \begin{equation*}
    \int_{F}\phi_1(n(x)g)\psi(-x) dx =\phi(g).
  \end{equation*} 
  Let $\bar{\pi}$ denote the conjugate representation of $\pi$.  Since $\pi$ is unitary, its conjugate $\bar{\pi}$ is isomorphic to the contragradient of $\pi$.
  We denote by $\pi^\infty \leq \pi$ and $\bar{\pi}^\infty \leq \bar{\pi}$ the subspaces of smooth vectors and by $\bar{\pi}^{-\infty}$ the space of distributional vectors, i.e., continuous functionals $\pi^\infty \rightarrow \mathbb{C}$.  In the archimedean case, the spaces of smooth vectors are topologized as in, e.g., \cite[\S3.2]{NV}, while in the non-archimedean case, it is given the inductive limit topology with respect to fixed subspaces of compact open subgroups.  Let $\delta \in \bar{\pi}^{-\infty}$ denote the distributional vector such that $\delta \otimes W \mapsto W(1)$ under the canonical linear map $\bar{\pi}^{- \infty} \otimes \pi^{\infty} \rightarrow \mathbb{C}$.
  Then $\bar{\pi}(\varphi_1)\delta\in\bar{\pi}^\infty$. Spectrally expanding $\pi(\varphi_1)\delta$ in $\bar{\pi}^\infty$, we obtain \cite[Section 7]{JN}
  \begin{equation*}
    \bar{\pi}(\varphi_1)\delta(1) =\sum_{W\in\B(\pi)}{W(1)}\int_{N\backslash G}\varphi(g)\overline{W(g)}dg.
  \end{equation*}
  (To see this formally, start with the distributional expansion $\delta(g) = \sum_{W \in \mathcal{B}(\pi)} \overline{W(g)} W(1)$, then apply $\bar{\pi}(\varphi_1)$ and evaluate at $1$.)  We choose a sequence of vectors $W_j\in \pi$ such that the functions $W_j(a(.))$ are nonnegative, $L^1(F^\times)$-normalized, and have supports tending to the identity element as $j\to \infty$.  Then $\overline{W_j}\to\delta$ in $\bar{\pi}^{-\infty}$; indeed, for any $V \in \bar{\pi}^\infty$, the restriction $V(a(.))$ is a smooth function on $F^\times$, hence in particular continuous in a neighborhood of the identity element, from which it follows that
  \begin{equation*}
    \int_{F^\times}\overline{W_j(a(y))}V(a(y)) \, d^\times y\xrightarrow{j\to\infty} V(1).
  \end{equation*}
  Thus we obtain
  \begin{equation*}
    \bar{\pi}(\varphi_1)\delta(1) = \lim_{j\to\infty} \bar{\pi}(\varphi_1)\overline{W_j}(1)
  \end{equation*}
  which follows from the continuity of the map
  \begin{equation*}
    C_c^\infty(G)\times \bar{\pi}^{-\infty}\to \bar{\pi}^\infty,\quad (f,\ell)\mapsto\bar{\pi}(f)\ell,
  \end{equation*}
  see, e.g., \cite[Lemma 3.6]{NV} for the archimedean case, while in the non-archimedean case, the asserted continuity follows readily from the admissibility of $\pi$.

  Combining, we can write
  \begin{equation*}
    \lim_{j\to\infty}\int_{N\backslash G}\phi(g)\overline{W_j(g)}=\sum_{W\in\B(\pi)}{W(1)}\int_{N\backslash G}\varphi(g)\overline{W(g)}dg.
  \end{equation*}
  
  Now we compute the left hand side in a different way. We integrate using Bruhat coordinates and write
  \begin{equation*}
    \int_{N\backslash G}\phi(g)\overline{W_j(g)}=\int_{F}\int_{F^\times}\phi(a(y)wn(x))\overline{W_j(a(y)wn(x))}\frac{d^\times y}{|y|}dx.
  \end{equation*}
  We apply Plancherel on the $F^\times$-integral and then apply the $\GL(2)\times\GL(1)$ local functional equation (Lemma \ref{local-functional-equation}) to write the above as
  \begin{equation*}
    \int_F\int_{\Re(\chi)=\sigma}\left(\int_{F^\times}\phi(a(y)wn(x))\chi(y)\frac{d^\times y}{|y|}\right){\gamma(1/2,{\pi}\otimes\chi)}\left( \int_{F^\times}\overline{W_j(a(t))}\psi(n(-xt))\chi(t) \, d^\times t \right) \, d\chi \,dx.
  \end{equation*}
  We note that $y$-integral is rapidly decaying in $x$ and (the conductor of) $\chi$, and the double integral in $x$ and $\chi$ is absolutely convergent for 
  $ \Re(\chi)<1/2-\vartheta$ (i.e., away from poles).
  
  As $j \rightarrow \infty$, the $t$-integral converges to $\psi(n(-x))$, and more precisely we have
  \begin{equation*}
    \int_{F^\times}(...) \, d^\times t  - \psi(n(-x)) = (|x| + C(\chi)) o(1)
  \end{equation*}
  as $j \rightarrow \infty$, as follows easily from a Taylor expansion. By the rapid decay of the $y$-integral, we may now interchange the $j$-limit with the $x$- and $\chi$-integrals and conclude the proof.
\end{proof}

\subsection{Proof of Theorem \ref{thm1}}\label{sec26}

\begin{proof}
  [Proof of \eqref{GL1-to-GL2}]

  We recall the relevant notation and use Lemma \ref{bessel-transform} with
  \begin{equation*}
    \phi = W_1\bar{f}_2(. a(z)) \in \mathcal{V}_{1 - \vartheta_1-\vartheta_2}((N, \psi)\backslash G)
  \end{equation*}
  to write $h^\vee(\pi,z)$ as
  \begin{equation*}
    \sum_{W\in\B(\pi)}W(1)\int_{N\backslash G}W_1\overline{f_2(ga(z))}\overline{W(g)}dg.
  \end{equation*}
  For $\vartheta_1 +\vartheta_2 < \sigma<1/2-\vartheta$ (which we assume from now on), we apply Proposition \ref{mellin-bessel} to obtain the absolutely convergent integral formula
  \begin{equation*}
    h^\vee(\pi,z)=\int_{\Re(\chi)=\sigma}{\gamma(1/2,{\pi}\otimes\chi)}\left(\int_F\int_{F^\times}W_1\overline{f_2(a(y)wn(x)a(z))}\chi(y)\psi(n(-x))\frac{d^\times y}{|y|}dx\right) \, d\chi.
  \end{equation*}
  Conjugating $wn(x)$ with $a(z)$, and changing variables $x\mapsto xz$ and $y\mapsto yz$, we can write the inner $F\times F^\times$ integrals as
  \begin{equation*}
    \int_F\int_{F^\times}W_1\overline{f_2(a(y)wn(x))}\chi(yz)\psi(n(-xz))\frac{d^\times y}{|y|}dx.
  \end{equation*}
  Applying \eqref{transformation-f} for $f_2$, we can rewrite the above as
  \begin{equation*}
    \chi(z)\int_F\overline{f_2(w n(x))}\psi(n(-xz))\int_{F^\times}W_1(a(y)wn(x))\bar{\chi}_2\chi(y)\frac{d^\times y}{|y|^{1/2}}dx.
  \end{equation*}
  We apply Lemma \ref{local-functional-equation} to write the inner integral as an absolutely convergent integral
  \begin{equation*}
    \gamma(1,\pi_1\otimes\bar{\chi}_2^{-1}\otimes\chi^{-1})\int_{F^\times}W_1(a(y))\psi(n(xy))(\bar{\chi}_2\chi)^{-1}(y)|y|^{1/2}d^\times y.
  \end{equation*}
  As $f_2(w(n(.))$ is Schwartz by Lemma \ref{principal-series-expansion} and the $y$-integral is absolutely convergent (for $\sigma < 1 - \vartheta_1- \vartheta_2$), we can exchange the $F^\times$ and $F$-integrals and calculate the absolutely convergent integral
  \begin{equation*}
    \int_F \overline{f_2(wn(x))}\psi(n(x(y-z))) \, dx = \overline{W_2(a(y-z))}\overline{\chi_2(y-z)}|z-y|^{-1/2},
  \end{equation*}
  which follows from \eqref{defn-intertwiner-explicit}. Writing $y=zt$ and simplifying, we conclude the proof.
\end{proof}

\begin{proof}
  [Proof of \eqref{GL2-to-GL1}]
  Let $z\in F^\times$ and $x\in F$. We use \eqref{whittaker-plancherel} to write
  \begin{equation*}
    W_1\bar{f}_2(a(z)wn(x))=\int_{\widehat{G}_{\mathrm{gen}}^{\mathrm{u}}}\sum_{W\in\B(\pi)}W(a(z)wn(x))\int_{N\backslash G}W_1(h)\overline{f_2(h)W(h)} dh\, d\pi .
  \end{equation*}
  As described before Lemma \ref{property-h-vee}, the $\pi$-integral and $W$-sum are jointly absolutely convergent.  We use Lemma \ref{local-functional-equation} and Mellin inversion to replace $W(a(z)wn(x))$ above by
  \begin{equation*}
    \int_{\Re(\chi)=0}\gamma(1/2,\pi\otimes\chi)\int_{F^\times}W(a(t))\psi(n(xt))\chi(tz) \, d^\times t \, d\chi.
  \end{equation*}
  We write the resulting expression as
  \begin{equation*}
    \int_{\widehat{G}_{\mathrm{gen}}^{\mathrm{u}}}\sum_{W\in\B(\pi)} \int_{\Re(\chi)=0}\gamma(1/2,\pi\otimes\chi) I_x(W, \chi)  I(W) d\chi\, d\pi ,
  \end{equation*}
  where $I_x(W, \chi) $ is the $t$-integral in the previous display and $ I(W)$ is the $h$-integral in the penultimate display.  Note that $I_x(W, \chi)$ is rapidly decaying as a function of $\chi$ (in the strip specified in the statement of Theorem \ref{thm1}), which can be seen from Lemma \ref{local-functional-equation} or by partial integration in the $t$-integral, and of moderate growth with respect to $W$.  Also, note that $I(W)$ is rapidly decaying in the $K$-type of $W$, as seen in the discussion before Lemma \ref{property-h-vee}. Finally, the $\gamma$-factor is of moderate growth with respect to $\chi$.  It follows that the double sum/integral over $W$ and $\chi$ converges absolutely, and so we may exchange their order.  We now have
  \begin{equation*}
    \sum_{W\in\B(\pi)}I_x(W, \chi)  I(W) = \sum_{W\in\B(\pi)}   I(W)\int_{F^\times}W(a(t))\psi(n(xt))\chi(tz) \, d^\times t, 
  \end{equation*}
  which converges absolutely by Lemma \ref{bound-whittaker} and the decay properties of $I(W)$.  We can now further move the $W$-sum inside the $t$-integral.

  Recalling~\eqref{def-h-vee}, we can thus write
  \begin{equation*}
    W_1\bar{f}_2(a(z)wn(x))=\int_{\widehat{G}_{\mathrm{gen}}^{\mathrm{u}}}\int_{\Re(\chi)=0}\gamma(1/2,\pi\otimes\chi)\int_{F^\times}\psi(n(xt))\chi(tz)h^\vee(\pi,t) \, d^\times t \, d\chi \, d\pi.
  \end{equation*}
  Lemma~\ref{property-h-vee} guarantees the joint absolute convergence of $(\pi,\chi)$-integral, which we now exchange.  In this way, we realize the right hand side as the inverse Mellin transform of
  \begin{equation*}
  \chi \mapsto  \int_{\widehat{G}_{\mathrm{gen}}^{\mathrm{u}}}\gamma(1/2,\pi\otimes\chi)\int_{F^\times}\psi(n(xt))\chi(t)h^\vee(\pi,t) \, d^\times t \, d\pi
  \end{equation*}
  at $z^{-1}$. Writing the left hand side as $W_1(a(z)wn(x))\overline{\chi_2(z)}|z|^{1/2}\overline{f_2(wn(x))}$, we obtain that
  \begin{multline*}
    \overline{f_2(wn(x))}\int_{F^\times}W_1(a(t)wn(x))|t|^{1/2}\overline{\chi_2(t)}\chi^{-1}(t) d^\times t\\
    =\int_{\widehat{G}_{\mathrm{gen}}^{\mathrm{u}}}\gamma(1/2,\pi\otimes\chi)\int_{F^\times}\psi(n(xt))\chi(t)h^\vee(\pi,t) \, d^\times t \, d\pi.
  \end{multline*}
  Note that the left hand side is holomorphic for $\Re(\chi)<1-\vartheta_1-\vartheta_2$, while the right hand side is holomorphic for $-1/2 < \Re(\chi)<1/2 \,\, (<1-\vartheta_1- \vartheta_2)$. We choose $\Re(\chi)>0$ sufficiently small and apply Lemma \ref{local-functional-equation} to the integral on the left hand side above, which gives
  \begin{multline*}
    \overline{f_2(wn(x))}\int_{F^\times}W_1(a(t))\psi(n(xt))|t|^{-1/2}\overline{\chi_2}^{-1}(t)\chi(t) d^\times t \\
    = \gamma(1,\pi_1\otimes\overline{\chi_2}\otimes\chi^{-1})\int_{\widehat{G}_{\mathrm{gen}}^{\mathrm{u}}}\gamma(1/2,\pi\otimes\chi)\int_{F^\times}\psi(n(xt))\chi(t)h^\vee(\pi,t) \, d^\times t \, d\pi.
  \end{multline*}
  Applying Lemma \ref{property-h-vee}, we check that the joint $(\pi,t)$-integral on the right hand side above is absolutely convergent, which allows us to exchange the $t$-integral with the $\pi$-integral. Then we write $\int_{F^\times}(\cdots) \, d^\times t$ as $\int_F(\cdots)|t|^{-1}dt$ and realize the right hand side as a Fourier coefficient of the function
  \begin{equation*}
    t \mapsto \chi(t)|t|^{-1}\gamma(1,\pi_1\otimes\overline{\chi_2}\otimes\chi^{-1})\int_{\widehat{G}_{\mathrm{gen}}^{\mathrm{u}}}\gamma(1/2,\pi\otimes\chi)h^\vee(\pi,t) \, d \pi.
  \end{equation*}
  Recalling the choice of $W_2$ and Lemma \ref{principal-series-expansion}, we see that $f_2(wn(.))$ is a Schwartz--Bruhat function. Hence, applying the inverse Fourier transform, we obtain for any $z\in F$ that
  \begin{multline*}
    \int_{F^\times}W_1(a(t))\overline{W_2(a(t-z))\chi_2(t-z)}|z-t|^{-1/2}|t|^{-1/2}\overline{\chi_2}^{-1}(t)\chi(t) d^\times t\\
    =\chi(z)|z|^{-1}\gamma(1,\pi_1\otimes\chi_2\otimes\chi^{-1})\int_{\widehat{G}_{\mathrm{gen}}^{\mathrm{u}}}\gamma(1/2,\pi\otimes\chi)h^\vee(\pi,z) \, d\pi.
  \end{multline*}
  Note that due to the choice of the $W_i$ as compactly supported functions, the left hand side above is entire in $\chi$. This gives analytic continuation of the right hand side, which is originally defined for $\vartheta_1+ \vartheta_2<\Re(\chi)<1/2 $. Moreover, the smoothness of $W_i(a(.))$ implies that the left hand side above is rapidly decaying in $\chi$, which implies the same about the right hand side in the strip $\vartheta_1 + \vartheta_2<\Re(\chi)<1/2 $.  This is a crucial point -- the regularity properties on $\chi$ are easy to see on the left hand side of the equation, but highly non-trivial to retrieve from the right hand side.

  We may now take Mellin inversion of the both sides at any $y\in F^\times$ and obtain
  \begin{multline*}
    W_1(a(y))\overline{W_2(a(y-z))} = |y|^{1/2}|z-y|^{1/2}\overline{\chi_2(y)\chi_2^{-1}(y-z)}|z|^{-1}\\
    \times\int_{\Re(\chi)=\sigma}\chi^{-1}(y)\chi(z)\gamma(1,\pi_1\otimes\overline{\chi_2}\otimes\chi^{-1})\int_{\widehat{G}_{\mathrm{gen}}^{\mathrm{u}}}\gamma(1/2,\pi\otimes\chi)h^\vee(\pi,z) \, d\pi\, d\chi.
  \end{multline*}
  Writing $y=:y_1$ and $y-z=:y_2$ we conclude the proof.
\end{proof}

\subsection{Proof of Theorem \ref{thm2}}\label{sec27}

\begin{proof}
  [Proof of \eqref{fourth-to-cubic}]
  We integrate the expression of $h^\vee$ in Theorem \ref{thm1} \eqref{GL1-to-GL2} against $\chi_0(y)$ over $y\in F^\times$. The support condition for $h$ implies that the resulting integrand is compactly supported in $y$.  As the $\chi$-integral there is rapidly convergent, we can exchange the $y$-integral with the $\chi$-integral. We conclude the proof after substituting $y\mapsto y/t$ and $t\mapsto(1-t)^{-1}$.
\end{proof}

\begin{proof}
  [Proof of \eqref{cubic-to-fourth}]
  We plug $y_1=z$ and $y_2=yz$ in the expression for $h(y_1,y_2)$ in Theorem \ref{thm1} \eqref{GL2-to-GL1} to obtain
  \begin{align*}
    h(z,yz) &= \frac{\overline{\chi_2^{-1}(y)}\sqrt{|y|}}{|1-y|}\int_{\Re(\chi)=\sigma}\chi(1-y)\gamma(1,\pi_1\otimes\overline{\chi_2}\otimes\chi^{-1})\\
            &\times\int_{\widehat{G}^u_\gen}h^\vee(\pi,z(1-y))\gamma(1/2,\pi\otimes\chi)\,d\pi\, d\chi.
  \end{align*}
  We integrate the above against $\chi_0(z)$ over $z\in F^\times$.  Lemma \ref{property-h-vee} and the rapid convergence of the $\chi$-integrand and $\pi$-integrand allows us to exchange the $z$-integral with the $\chi$- and $\pi$-integrals.  The required identity follows after a change of variable $z\mapsto z(1-y)^{-1}$.
\end{proof}

\subsection{Some explicit formulae over $\Bbb{R}$}\label{sec28} For convenience, we spell this out for $F = \Bbb{R}$, $\psi(x) = e(x)$ the standard additive character, and two principal series representations $\pi_1$, $\pi_2$ corresponding to two even Maa{\ss} forms of spectral parameters $r_1, r_2$ respectively.

For a character $\chi = \chi_{\tau, \delta} = |.|^{\tau} \text{sgn}^{\delta}$ of $\Bbb{R}^{\times}$ with $\delta \in \{0, 1\}$, $\tau \in i\Bbb{R}$, set
\begin{equation}\label{gamma}
  G(\pi; \chi ) :=  
  \begin{cases}
    \prod_{\pm} \frac{\Gamma(\frac{1}{2} (1/2\pm ir - \tau + \rho))}{\Gamma(\frac{1}{2}(  1/2\pm ir + \tau + \rho ))}
    , & \pi = \pi_{(r, \eta)}  \text{ principal series of parity $\eta$}\\ \textbf{1}_{\delta = 0} i^k\frac{\Gamma(k/2 - i\tau)}{\Gamma(k/2 + i\tau)} 
    , & \pi = \pi_k \text{ discrete series of weight $k$}
  \end{cases}
\end{equation}
where in the first case, $\rho := 0$ if $\eta = \delta$ and $\rho := 1$ otherwise. 

Then
\begin{equation}\label{expl1}
  \begin{split}
    h^{\vee}(\pi, y) = \frac{1}{2} \sum_{\delta \in \{0, 1\}}  \int_{(\frac{1}{4})} &\int_{\Bbb{R}^{\times}} h(yt, y(t-1)) \text{sgn}(t)^{\delta}|t|^{-\tau} \Big| \frac{1-t}{t}\Big|^{-ir_2 - \frac{1}{2}} \frac{dt}{|t|}  \\
                                                                                    & \times     G(\pi, \chi_{\tau, \delta})
    \pi^{2ir_2}\prod_{\pm} \frac{\Gamma(\frac{1}{2}(\pm ir_1 -ir_2+ \tau + \delta))}{\Gamma(\frac{1}{2}(1\pm ir_1 +ir_2- \tau + \delta))} \frac{d\tau}{2\pi i}
  \end{split}
\end{equation}
and 
\begin{equation}\label{expl2}
  \begin{split}
    h^{\sharp}(\pi, \text{triv}) =  \frac{1}{2} \sum_{\delta \in \{0, 1\}}  \int_{(\frac{1}{4})} &\int_{\Bbb{R}^{\times}}\int_{\Bbb{R}^{\times}} h(z,  yz ) \text{sgn}(1-y)^{\delta}|y-1|^{\tau-1}|y|^{-ir_2 + \frac{1}{2}} \frac{dz}{|z|}  \frac{dy}{|y|} \\
                                                                                                 & \times G(\pi, \chi_{\tau, \delta})
    \pi^{2ir_2}\prod_{\pm} \frac{\Gamma(\frac{1}{2}(\pm ir_1 -ir_2+ \tau + \delta))}{\Gamma(\frac{1}{2}(1\pm ir_1 +ir_2- \tau + \delta))} \frac{d\tau}{2\pi i}. 
  \end{split}
\end{equation}
Analogous identities hold if $\pi_1$ and/or $\pi_2$ have a different parity. For $F= \Bbb{R}$, the inversion formulae in Theorems \ref{thm1} and \ref{thm2} are variations of the classical Mehler--Fock transform (see \cite{Mo5}).

\section{Global results}

In this section, we prove Theorems \ref{thm-shift} and \ref{thm-moment} in \S~\ref{sec36} and \S~\ref{sec37}. This section takes a global point of view except for \S~\ref{sec34} and \S~\ref{sec37a} where we return temporarily to the local set-up of the previous section.

\subsection{Notation and measures}\label{sec:cj8owg7kvw}
Let $F$ be a (global) number field with adele ring $\mathbb{A}$.  
We denote by $\mathfrak{p}$ a (possibly archimedean) place of $F$, by $F_\mathfrak{p}$ the corresponding completion, and, when $\mathfrak{p}$ is finite, by $\mathfrak{o}_\mathfrak{p}$ the maximal order.  Let $\psi$ be a nontrivial unitary character of $\mathbb{A}/F$. Without loss of generality, we may assume that $\psi$ is unramified at places where $F$ is unramified. 
We use the local components $\psi_\mathfrak{p}$ to normalize Haar measures $d x$ on each of the local completions $F_\mathfrak{p}$ of $F$ and define local measures as in \S~\ref{sec22}.

We equip $\mathbb{A}$ with the product of the given Haar measures on $F_\mathfrak{p}$.  The resulting quotient measure on $\mathbb{A} / F$ then has volume one. We equip $\mathbb{A}^\times$ with the Haar measure given on factorizable functions $f(y) = \prod_{\mathfrak{p}} f_\mathfrak{p}(y_\mathfrak{p})$, with $f_\mathfrak{p}$ the characteristic function of $\mathfrak{o}_\mathfrak{p}^\times$ for almost all $\mathfrak{p}$, by
\begin{equation}\label{haar}
  \int_{\mathbb{A}^\times } f =
  \frac{1}{\xi_F^*(1)}
  \prod_{\mathfrak{p} }
  \Big(\zeta_\mathfrak{p}(1) \int_{F_\mathfrak{p}^\times } {f_{\mathfrak{p} }}\Big).
\end{equation}
Here and in the following $\xi_F$ denotes the (completed) Dedekind zeta function of $F$ and an asterisk, here and elsewhere, indicates that the residue is taken.

We define the algebraic group $G = {\rm PGL}_2$ with subgroups  $B, A, N$ as usual. We equip $A(\mathbb{A})$ and $N(\mathbb{A})$ with the Haar measures transferred from $\mathbb{A}^\times$ and $\mathbb{A}$, and $B(\mathbb{A})$ with the left Haar transferred via the multiplication map $A \times N \rightarrow B$. We equip $G(\mathbb{A})$ with the usual Tamagawa measure so that the adelic quotient $[G] =G(F) \backslash G(\mathbb{A})$ has volume $2$. For a factorizable integral function $f$, we have
\begin{equation}\label{tamagawa}
  \int_{G(\mathbb{A})} f(g) dg
  =
  \prod_{\mathfrak{p} }
  \int_{ A(F_\mathfrak{p})}
  \int_{  N(F_\mathfrak{p})} \int_{  N(F_\mathfrak{p})}
  f_\mathfrak{p} (n_1 a w n_2) \, d n_1 \, d n_2 \, \frac{d^{\times} a}{\lvert a \rvert}.
\end{equation}
This identity follows, after unfolding the definition of the Tamagawa measure (see for instance \cite[\S X.3]{CasFr}), from the fact that the differential form on $N w N A$ given in the coordinates $n(x_1) w n(x_2) a(y)$ by $d x_1 \wedge  \frac{d y}{y} \wedge d x_2$ is the restriction of a top-degree invariant differential form on $G$.

We write $[A] := \Bbb{A}^{\times}/F^{\times}$.  We denote by $\widehat{[A]}$ the dual of $[A]$ and $\widehat{[A]}^{\mathrm{u}}$ the unitary dual (i.e.\ Hecke characters $\chi$ with $\Re(\chi) = 0$), equipped with the measure dual to the measure on $[A]$. We write $\widehat{[A]}_S^{\mathrm{u}}$ for those unitary characters that are unramified outside $S$. We write $[G] = G(F) \backslash G(\Bbb{A})$ and $\widehat{[G]}, \widehat{[G]}_{\mathrm{gen}}, \widehat{[G]}_{\mathrm{gen}}^{\mathrm{u}}, \widehat{[G]}_{\mathrm{gen}, S}^{\mathrm{u}}$ 
as in \S~\ref{sec:cnebsqy7eo}.  The spectral decomposition and the measure $d\pi$ on $\widehat{[G]}$ are given in \eqref{spec} below. In particular, $d\pi$ is half the counting measure on the cuspidal spectrum.

\subsection{Global Petersson and Whittaker norms}\label{sec:cj8owh97fg}
Let $\pi \in\widehat{[G]}_{\mathrm{gen}}^{\mathrm{u}} $ 
and let $\varphi \in \pi$ be a factorizable vector.  Its Whittaker function $W$ may then be factored as $\otimes _{\mathfrak{p}} W_\mathfrak{p}$, where for some large enough finite set of places $S$ (containing all archimedean places, and all finite places at which $\psi$ ramifies), we have for $\mathfrak{p} \notin S$ that $W_\mathfrak{p}$ is normalized spherical, i.e., $W_\mathfrak{p}(1) = 1$ and $W_\mathfrak{p}$ is $G(\mathfrak{o}_\mathfrak{p})$-invariant.   
We then have the following formula, relating the Petersson norm of $\varphi$ to the norm of its Whittaker function as defined in the Kirillov model:
\begin{equation}\label{eq:cj7ok440fl}
  \int_{[G]} \lvert \varphi  \rvert^2
  =
  2
  \left.\prod_{\mathfrak{p}}\right.^* \int_{  F_\mathfrak{p}^\times }
  \lvert W_\mathfrak{p}(a(y)) \rvert^2 \, d^\times y.
\end{equation}
Here $\prod_{\mathfrak{p}}^*$ denotes a ``regularized Euler product'' as in \cite[\S4.1.5]{MV}, given  as follows:   
\begin{equation}\label{eq:cj7ok542dg}
   \left.\prod_{\mathfrak{p}}\right.^* \int_{ F_\mathfrak{p}^\times }
   \lvert W_\mathfrak{p}(a(y)) \rvert^2 \, d^\times y
   =
   \frac{
    L^{(S)}(1, \pi, \Ad)
  }{
    \zeta_F^{(S)}(2)
  }
  \prod_{\mathfrak{p} \in S}
  \int_{ F_\mathfrak{p}^\times }
  \lvert W_\mathfrak{p}(a(y)) \rvert^2 \, d^\times y.
\end{equation}
As usual,  a superscripted ${(S)}$ denotes partial $L$-functions, defined by Euler products over places $\mathfrak{p} \notin S$. 
The formula \eqref{eq:cj7ok440fl} is standard and may be derived from, e.g., \cite[Lem.\ 2.2.3]{MV}, cf.\ also \cite[Section 8]{Ne}. 
\begin{remark}
  While \eqref{eq:cj7ok440fl} is a standard consequence of the Rankin--Selberg method, the precise constants depend upon measure normalizations.  For the convenience of the reader, we record a brief heuristic indicating why the above normalization is correct.  We consider the integral, defined initially for $\Re(s) > 0$,
  \begin{equation*}
    I(s) := \int_{  \mathbb{A}^\times } \left\lvert W_\varphi (a(y))  \right\rvert^2 \lvert y \rvert^s
    \, \frac{d y}{\lvert y \rvert}.
  \end{equation*}
  By \eqref{haar}, this factors as
  \begin{equation*}
    I(s) = \xi_F^*(1)^{-1} \prod_{\mathfrak{p} } \zeta_{\mathfrak{p}}(1) I_\mathfrak{p}(s), \quad 
    I_{\mathfrak{p} } (s ) := \int_{  F_\mathfrak{p}^\times } \left\lvert W _\mathfrak{p} (a(y)) \right\rvert^2 \lvert y \rvert^s \, \frac{d y}{\lvert y \rvert}.
  \end{equation*}
  For any finite place $\mathfrak{p}$ at which $\psi_\mathfrak{p}$ and $W_\mathfrak{p}$ are unramified, we can compute (cf.\ \cite[Lem.\ 1.6.1, Prop.\ 3.8.1]{Bump})
  \begin{equation*}
    I_\mathfrak{p}(s) = \frac{L_\mathfrak{p}(1+s, \pi, \Ad)}{\zeta_\mathfrak{p}(2 + 2 s)}
    \frac{\zeta_\mathfrak{p}(1 + s)}{\zeta_\mathfrak{p} (1)}.
  \end{equation*}
  We see in particular that $I(s)$ extends to a meromorphic function of $s$,  the right hand side of \eqref{eq:cj7ok542dg} is independent of the choice of the large enough finite set of places $S$, and we have
\begin{displaymath}
  \begin{split}
   \underset{s =0}{\text{res}} I(s) & = \xi_F^*(1)^{-1}  \Big( \frac{L^{(S)}_\mathfrak{p}(1 , \pi, \Ad)}{\zeta^{(S)}_F(2  )}
  \underset{s =0}{\text{res}}   \zeta^{(S)}_F(1 + s) \Big) \prod_{\mathfrak{p} \in S} \zeta_{\mathfrak{p}}(1)
  \int_{ F_\mathfrak{p}^\times }
  \lvert W_\mathfrak{p}(a(y)) \rvert^2 \, d^\times y\\
  & =      \frac{L^{(S)}_\mathfrak{p}(1 , \pi, \Ad)}{\zeta^{(S)}_F(2  )}
  \prod_{\mathfrak{p} \in S}  
  \int_{  F_\mathfrak{p}^\times }
  \lvert W_\mathfrak{p}(a(y)) \rvert^2 \, d^\times y.
  \end{split}
  \end{displaymath}
   On the other hand, by unfolding, we have
   \begin{equation*}
     I(s) =  \int_{  \mathbb{A}^\times / F^\times} \int_{  \mathbb{A}/F} |\varphi(n(x) a(y))|^2 |y|^s \,
     \frac{d y}{|y|}\, dx.
   \end{equation*}
  The Haar measure on $\mathbb{A}^\times / F^\times$ has been normalized so that its pushforward under the idelic absolute value is the standard Haar measure $\frac{d t}{t}$ on the positive reals.  It follows that
  \begin{equation*}
    \underset{s =0}{\text{res}} I(s) = \lim_{t \rightarrow 0}
    \mathbb{E}_{y : |y| = t} \int_{  \mathbb{A}/F}
    |\varphi(n(x) a(y))|^2 \, d x.
  \end{equation*}
  Since we have normalized so that $\vol(\mathbb{A}/F) = 1$, we deduce from the equidistribution of the horocycle flow that the limit in question is the average value of $|\varphi|^2$, defined using the probability Haar measure on $[G]$.  Since  $\text{vol}([G]) = 2$, we obtain $ \underset{s =0}{\text{res}} I(s) = \frac{1}{2} \int_{[G]} |\phi|^2$, and the claimed formula follows. 
\end{remark}

A similar formula holds in the Eisenstein case.  We quote \cite[\S8.2]{Ne}.  For a character $\chi$ of $[A]$, we denote by $\mathcal{I}(\chi)$   the normalized induced representation of $G(\mathbb{A})$.  The factorizable vectors $f \in \mathcal{I}(\chi)$ may be written $f = \otimes_{\mathfrak{p}} f_\mathfrak{p}$, where $f_\mathfrak{p}$ belongs to the local induced representation $\mathcal{I}(\chi_\mathfrak{p})$ and is normalized spherical (i.e., $f_\mathfrak{p}(1) = 1$ and $f_\mathfrak{p}$ is $G(\mathfrak{o}_\mathfrak{p})$-invariant) for almost all $\mathfrak{p}$.  There is a natural $G(\mathbb{A})$-invariant duality between $\mathcal{I}(\chi)$ and $\mathcal{I}(\chi^{-1})$, given on factorizable vectors $f, f'$ by the following regularized Eulerian integral:
\begin{equation}\label{regEuler}
  (f,f')_{\mathcal{I}(\chi)} := \int_{B(\mathbb{A}) \backslash G(\mathbb{A})} f f' := \frac{\xi_F^*(1)}{\xi_F(2) } \prod_{\mathfrak{p}}\Big( \frac{\zeta_{F_\mathfrak{p}}(2)}{\zeta_{F_\mathfrak{p}}(1)} \int_{B(F_\mathfrak{p} ) \backslash G(F_\mathfrak{p} )} {f}_\mathfrak{p} f_\mathfrak{p}' \Big).
\end{equation}
Up to a null set, we can parametrize $B(\Bbb{A}) \backslash G(\Bbb{A})$ as $w N(\Bbb{A})$.  In the unitary case $\chi = \overline{\chi }^{-1}$, this defines the invariant inner product on $\mathcal{I}(\chi)$.

To unify the notation, for $\pi \in \widehat{[G]}^{\mathrm{u}}_{\mathrm{gen}}$  
we define a pairing $\langle ,  \rangle_{\pi}$ on $\pi$, as follows:
\begin{itemize}
\item If $\pi$ is cuspidal, then 
\begin{equation}\label{cusp-pairing}
   \langle \varphi_1, \varphi_2 \rangle_{\pi} := \frac{1}{2}\int_{[G]} \varphi_1 \overline{\varphi_2}.
   \end{equation}
  We will see in a moment that the inclusion of the factor $1/2$ leads to esthetically pleasing formulae. 
\item If $\pi = \mathcal{I}(\chi)$ is a unitary Eisenstein series and $\varphi_j = \Eis(f_j)$, where as usual
  \begin{equation*}
    \varphi(g) = \text{Eis}(f)(g) = \sum_{\gamma \in B(F) \backslash G(F)} f(\gamma g),
  \end{equation*}
  then we set 
  \begin{equation}\label{eq:cj8owoe87h}
    \langle \varphi_1, \varphi_2 \rangle_{\pi} := (\varphi_1, \overline{\varphi_2} )_{\pi} = \int_{B(\mathbb{A}) \backslash G(\mathbb{A})} f_1 \overline{f_2}.
  \end{equation}
  Note here that, since $\chi$ is unitary, we have $\overline{f_2} \in \mathcal{I}(\bar{\chi}) = \mathcal{I}(\chi^{-1})$, so we are in position to apply~\eqref{regEuler}.
\end{itemize}

For future reference, we note that if $\pi = \mathcal{I}(\chi)$ is Eisenstein, then the completed $L$-functions satisfy
\begin{equation*}
  \Lambda(s, \pi, \Ad)
  =
  \Lambda(s, \chi^2)
  \Lambda(s, \chi^{-2})
  \xi_F(s).
\end{equation*}
For $\chi^2 \neq 1$, we define
\begin{equation}\label{eq:cj7mbffu4s}
  \Lambda^{\ast}(1, \pi, \Ad)
  :=
  \Lambda(1, \chi^2)
  \Lambda(1, \chi^{-2})
  \xi^{\ast}_F(1).
\end{equation}
When $\chi^2 = 1$, we define $1/\Lambda^*(1,\pi,\Ad) := 0$, so that $\chi \mapsto 1 / L^*(1,
\mathcal{I}(\chi),\Ad)$ is continuous.

\subsection{Spectral decomposition}
We briefly recall the spectral decomposition of $L^2 ([G])$.  It involves the one-dimensional representations, the cuspidal automorphic representations and the Eisenstein series.  The latter arise from the normalized induced representations $\mathcal{I}(\chi)$ attached to characters $\chi$ of $[A] = A(F) \backslash A(\mathbb{A})$. The following can be found in  \cite[Section 8.6]{Ne}. 

\begin{lemma}\label{lem31}
 Let $\Psi \in C_c^{\infty}([G])$. Then
  \begin{equation*}
    \begin{split}
      \Psi(g) &=  \sum_{\substack{\omega \in \widehat{[A]} \\\omega^2 = 1}}  \omega(\det g)  \frac{\int_{[G] } \Psi \omega^{-1}(\det)}{\vol ([G])}+ \frac{1}{2}\sum_{\sigma \text{ {\rm  cusp.}}} \sum_{\varphi \in \mathcal{B}(\sigma)} \varphi(g) \frac{\int_{[G] } \Psi \bar{\varphi}}{\langle \varphi, \varphi \rangle_{\sigma}}\\
           & + \frac{1}{2} \int_{\chi   \in \widehat{[A]}^{\rm u}} \sum_{f \in \mathcal{B}(\mathcal{I}(\chi))} \text{{\rm Eis}}(f)(g) \frac{ \int_{[G]} \Psi \text{ {\rm Eis}}(\bar{f})}{ \langle \text{{\rm Eis}}(f),\text{{\rm Eis}}(f)  \rangle_{\mathcal{I}(\chi)} }
    \end{split}
  \end{equation*}
where 
$\sigma$ runs over the cuspidal automorphic representations, 
$\varphi$ (resp.\ $f$) runs over any orthogonal basis consisting of $K$-isotypic vectors for $\sigma$ (resp.\ $\mathcal{I}(\chi)$). 
\end{lemma}

\begin{remark}
  \begin{enumerate}[a)]
  \item The factor $1/2$ in front of the cuspidal term compensates the factor $1/2$ in the definition of the pairing \eqref{cusp-pairing}. 
  \item To simplify the notation, we write the above decomposition as
    \begin{equation}\label{spec}
      \begin{split}
        \Psi(g) &=
                  \int_{\widehat{[G]}^{\mathrm{u}}}
                  \sum_{\varphi \in \mathcal{B}(\sigma)}
                  \varphi(g)
                  \frac{
                  \int_{[G] } \Psi \bar{\varphi}
                  }{
                  \langle \varphi, \varphi \rangle_{\sigma}
                  } \, d \sigma   \\
                &=
                  \sum_{\omega^2 = 1}  \omega(\det g)  \frac{\int_{[G] } \Psi \omega^{-1}(\det)}{\vol ([G])} + \int_{\widehat{[G]}^{\mathrm{u}}_{\mathrm{gen}}} \sum_{\varphi \in \mathcal{B}(\sigma)} \varphi(g) \frac{\int_{[G] } \Psi \bar{\varphi}}{\langle \varphi, \varphi \rangle_{\sigma}} \, d \sigma.
      \end{split}
    \end{equation}
    Here $d \sigma$ is given on, e.g., the cuspidal spectrum, by \emph{half} the counting measure.
  \item If necessary, the Eisenstein term can be extended holomorphically to all characters $\chi$ of $[A]$.  To do so, we replace this term with
    \begin{equation*}
      \sum_{f \in \mathcal{B}(\mathcal{I}(\chi))} \text{{\rm Eis}}(f)(g) \frac{ \int_{[G]} \Psi \text{ {\rm Eis}}(\tilde{f})}{ ( \text{{\rm Eis}}(f),\text{{\rm Eis}}(\tilde{f})  )_{\mathcal{I}(\chi)} } 
    \end{equation*}
    \begin{itemize}
    \item $f$ runs over a basis of $K$-isotypic vectors for $\mathcal{I}(\chi)$, say $f_1, f_2, \dotsc$, and
    \item $\tilde{f}$ simultaneously runs over the corresponding element of a dual (up to normalization) basis for $\mathcal{I}(\tilde{\chi})$, where $\tilde{\chi} := \chi^{-1}$, say $\tilde{f}_1, \tilde{f}_2, \dotsc$.
    \end{itemize}
    In the special case that $\chi$ is unitary, we can simply take $f$ to traverse an orthogonal basis and $\tilde{f} := \bar{f}$.
  \end{enumerate}
\end{remark}

\subsection{Local intertwining operators}\label{sec34}
We recall some identities from \cite[\S3]{BJN} and pass for a moment to the local setting, recalling the conventions in \S~\ref{sec21} and \ref{sec22}. Let $F$ be a local field and $G = {\rm PGL}_2(F)$.  Let $\psi$ be a nontrivial unitary character of $F$, which normalizes the Haar measure $d x$ on $F$. Let $\chi$ be a character of $F^\times$. Recall that $\mathcal{I}(\chi)$ denotes the associated principal series representation of $G$. We write $M(\chi)$ for the standard intertwining  operator from $\mathcal{I}(\chi)$ to $\mathcal{I}(\chi^{-1})$, given by meromorphic continuation along flat sections of
\begin{equation*}
  M(\chi) f(g) = \int_{ F} f (w n(x) g) \, d x.
\end{equation*}
We define the normalized variant
\begin{equation*}
  M^*(\chi) := \gamma (0, \psi, \chi^2 ) M (\chi),
\end{equation*}
with $\gamma$ the Tate local $\gamma$-factor, cf.\ \cite[(2.2)]{BJN}. The significance of this normalization is that it preserves Jacquet integrals, i.e., for $f \in \mathcal{I}(\chi)$ we have 
\begin{equation}\label{eq:cj7piyddx4}
  W_{M^*(\chi) f} = W_f;
\end{equation}
cf.\ \cite[Prop 4.5.9]{Bump}.

Set $\pi_2 := \mathcal{I}(\chi)$, and let $\pi_1, \pi_3 \in \widehat{G}_{\mathrm{gen}}$. 
For $W_1 \in \mathcal{W} (\pi_1, \psi )$, $W_3 \in \mathcal{W} (\pi_3, \psi)$ and $f_2 \in \pi_2 = \mathcal{I}(\chi)$, we define the local Rankin--Selberg integral
\begin{equation*}
  \Psi (W_1, f_2, W_3 ) :=
  \int_{N \backslash G}
  W_1 f_2 \overline{W_3}.
\end{equation*}
The following lemma records how these behave with respect to the normalized intertwining operators. 
\begin{lemma}\label{lemma:cngq4van7c}
  With the notation and assumptions as above,   we have
  \begin{equation*}
    \Psi (W_1, f_2, {W}_3 ) \gamma (\tfrac{1}{2}, \pi_1 \otimes \pi_3 \otimes \chi )
    = \Psi (W_1, M^\ast (\chi) f_2, {W}_3 ).
\end{equation*}
\end{lemma}
\begin{proof}
  This is \cite[Lemma 5]{BJN}.  Strictly speaking, the cited lemma treats the case $F = \mathbb{R}$, but the proof applies verbatim over any local field.  We note also that the cited reference works with $\mathcal{W} (\pi_3, \psi^{-1} )$, while we work here with $\overline{\mathcal{W}(\pi_3, \psi)}$, but the two spaces are identical.
\end{proof}

We now assume that $\pi_2$ is unitarizable, so that $\chi$ is equal to either $\bar{\chi}$ (the ``non-tempered axis'') or $\bar{\chi}^{-1}$ (the ``tempered axis'').  We may then normalize an invariant inner product on $\pi_2$ as follows: for $f_1, f_2 \in \pi_2$ let
\begin{equation*}
  \left\langle f_1, f_2 \right\rangle :=
  \begin{cases}
\int_{F} f_1 (n ' (x)) \overline{f_2 (n ' (x))} \, d x &  \text{ if } \chi = \bar{\chi}^{-1}, \\
\int_{ F} f_1 (n ' (x)) \overline{M^*(\chi) f_2 (n ' (x))} \, d x &  \text{ if } \chi = \bar{\chi},
\end{cases}
\end{equation*}
where as before $n'(x) = \left(
\begin{smallmatrix}
 1 & \\ x & 1 
\end{smallmatrix}
\right)$. 
\begin{remark}
  \begin{enumerate}[a)]
  \item As a check, we note that the above formula ``varies continuously'' with $\chi$ at the intersection of the tempered and non-tempered axes (i.e., $\chi^2 = 1$).  Indeed, recall from \eqref{eq:cj7piyddx4} that the normalized intertwining operator preserves Jacquet integrals, i.e.,
    \begin{equation*}
      \int_{F}
      f \left( w n(x) \right) \psi^{-1}(x) \, d x
      =
      \int_{F}
      M^*(\chi) f \left( w n(x) \right) \psi^{-1}(x) \, d x.
    \end{equation*}
    When $\chi^2 = 1$, the induced representation $\pi_2$ is irreducible, so $M^*(\chi)$ acts on it by some scalar $c$.  Since the Jacquet integral is not identically zero, it follows that $c = 1$.
  \item The definition is also consistent with the familiar formula
    \begin{equation}\label{isom}
      \langle f, f \rangle =  \int_{ F^\times }
      \left\lvert W_f(a(y)) \right\rvert^2
      \, \frac{d y}{ \lvert y \rvert},
    \end{equation}
    i.e., the fact that the induced model and the Whittaker model are isometric. Indeed,
    \begin{equation*}
      W_f(a(y)) = \int_{ F} f( w n(x)a(y)) \psi^{-1}(x) \, d x
    \end{equation*}
    rearranges to
    \begin{equation*}
      W_f(a(y)) = \lvert y \rvert^{1/2} \chi^{-1}(y) \int_{ F} f( w n(x)) \psi^{-1}( x y) \, d x, 
    \end{equation*}
    and by the same argument,
    \begin{equation*}
      W_{M^*(\chi) f}(a(y)) = \lvert y \rvert^{1/2} \chi(y) \int_{ F} M^*(\chi) f( w n(x)) \psi^{-1}( x y) \, d x.
    \end{equation*}
    By the fact \eqref{eq:cj7piyddx4} that the normalized intertwining operator preserves the Jacquet integral, we conclude \eqref{isom} also in the non-tempered case.
  \end{enumerate}
\end{remark}

With the above normalization, we have the following identity relating integrals of matrix coefficient to products of Rankin--Selberg integrals:
\begin{lemma}\label{lemma:cngq4varpv}
  Let $\pi_1, \pi_2, \pi_3 \in \widehat{G}_{\gen}^{\rm u}$, with inner products $\langle , \rangle$, and with $\pi_2 = \mathcal{I}(\chi)$ belonging to the (unitary) principal series, as above.  Assume each $\pi_j$ is $\vartheta_j$-tempered, where $\vartheta_1 + \vartheta_2 + \vartheta_3 < 1/2$.

  We identify $\pi_1$ with $\mathcal{W}(\pi_1, \psi)$ and $\pi_3$ with $\mathcal{W}(\pi_3, \psi)$, as above, normalized with respect to the inner product \eqref{eq:cnebswcf2w}.

  Let $W_1 \in \mathcal{W}(\pi_1, \psi), f_2 \in \mathcal{I}(\chi), W_3 \in \mathcal{W}(\pi_3, \psi)$.  Write simply $v_1, v_2, v_3$ for the same vectors regarded as elements of the ``abstract'' unitary representations $\pi_1, \pi_2, \pi_3$.  Then, with the notation
  \begin{equation*}
    f_2^{\natural} :=
    \begin{cases}
      f_2 & \text{ if } \pi_2 \text{ is tempered}, \\
      M^\ast (\chi) f_2  &  \text{ otherwise,}
    \end{cases}
  \end{equation*}
  we have
  \begin{equation}\label{eq:cngq4vi3q1}
    \int_{G}
    \prod_{i = 1}^3 \left\langle \pi_i (g) v_i, v_i  \right\rangle
    \, d g
    =
    \Psi (W_1, f_2, {W}_3 )
    \overline{\Psi (W_1, f_2^{\natural}, {W}_3 )}.
  \end{equation}
\end{lemma}
\begin{proof}
  This is \cite[Lemma 6]{BJN}, with two caveats.  The first is that the cited reference states this result for the real case, but the arguments apply over any local field.  The second is that the cited reference imposes stronger $\vartheta$-temperedness conditions on the representations than we do here.  These conditions are used to justify the convergence of the integrals in \eqref{eq:cngq4vi3q1}.  We have recorded here the weaker conditions required in that justification.
\end{proof}

\begin{cor}\label{corollary:cj7mbfn7cs}
  With the same notation and assumptions as in the previous lemma, we have
  \begin{multline*}
    \int_{G}
    \prod_{i = 1}^3 \left\langle \pi_i (g) v_i, v_i  \right\rangle
    \, d g  
    =
    \begin{cases}
     | \Psi (W_1, f_2, {W}_3 ) |^2,  &
                                                                 \text{ $\pi_2$ tempered principal series,} 
      \\
      \gamma (\tfrac{1}{2}, \pi_1 \otimes \pi_3 \otimes \chi )
      | \Psi (W_1, f_2, {W}_3 ) |^2,  &
                                                                 \text{ $\pi_2$ complementary series.} 
    \end{cases}
  \end{multline*}
\end{cor}
\begin{proof}
  We see first from Lemma \ref{lemma:cngq4van7c} and Lemma \ref{lemma:cngq4varpv} that the left hand side is
  \begin{equation*}
    \Psi (W_1, f_2, {W}_3 )
    \overline{\Psi (W_1, {f}_2, {W}_3 )\gamma (\tfrac{1}{2}, \pi_1 \otimes \pi_3 \otimes \chi )}.
  \end{equation*}
  The $\gamma$-factor in question is known to be positive (see \cite[Prop.\ 2.1]{PR}), so we obtain the stated formula.
\end{proof}

\begin{remark}
  \begin{enumerate}[a)]
  \item   Once again, if $\chi^2 = 1$, then $\gamma (\tfrac{1}{2}, \pi_1 \otimes \pi_3 \otimes \chi ) = 1$.  Thus the right hand side of the above formula ``varies continuously with $\chi$'' at the intersection of the tempered and non-tempered unitary axes.
  \item If $\pi_1, \pi_2, \pi_3$ are $\vartheta$-tempered for some fixed $\vartheta < 1/6$, then $ \gamma (\tfrac{1}{2}, \pi_1 \otimes \pi_3 \otimes \chi ) \asymp_{\vartheta} 1.$
  \end{enumerate}
\end{remark}

\subsection{Triple product formulae}

We now return to the global setting and let $F$ denote a global field and $S$ a finite set of places containing all archimedean places and all places where $F$ is ramified. 

\begin{lemma}\label{lemma:cj3nf2fkak}
  Let $\pi_1, \pi_2, \pi_3 \in \widehat{[G]}^{\text{{\rm u}}}_{\text{{\rm gen}}, S}$, at least one of which is cuspidal.  Assume that $\pi_2$ belongs to the principal series at each place (in $S$).  There exists a unique $c_{\pi_3}^{(S)} \in \mathbb{C}$ (depending also upon $\pi_1$ and $\pi_2$, but we suppress this dependence for notational brevity) with the following property.  Let $\varphi_j \in \pi_j$ ($j=1,2,3$) be unramified outside $S$, with $\| \varphi_3 \| \not= 0$.  We denote by
  \begin{equation*}
    W_j(g) = \int_{\mathbb{A} / F} \varphi_j(n(x) g) \psi(- x) \, d x,
  \end{equation*}
  the associated $\psi$-Whittaker functions.  We factor $W_2 = W_{2,S} W_2^S$ into its local components inside and outside $S$, with the normalization $W_2^S(1) = 1$.  We choose a character $\chi_{2,S}$ that induces the local component $\pi_{2,S}$, and write $f_2 \in \mathcal{I}(\chi_{2,S})$ for the element of the induced model whose Jacquet integral is the Whittaker function $W_{2,S}$.  Then
\begin{equation}\label{eq:cj4umojgpv}
    \frac{\int_{[G]} \varphi_1 \overline{\varphi_2 \varphi_3}}{ \langle \varphi_3, \varphi_3  \rangle_{\pi_3} }
    =
    c_{\pi_3}^{(S)}
    \frac{    \int_{N(F_S) \backslash G(F_S)}
      W_{1} \overline{W_3 f_{2}}
    }{
      \int_{A(F_S)}
      \lvert W_{3} \rvert^2
    }.
  \end{equation}
\end{lemma}

We will compute the magnitude of the constant $c_{\pi}^{(S)}$ in the next lemma.  We emphasize that it depends only upon the representations, and not the choice of vectors $\varphi_j$ (unramified outside $S$).  In particular, it is independent of the normalizations of those vectors.

\begin{proof}
  The denominators both define invariant inner products on the irreducible representation $\pi_{3,S}$ and thus differ by a multiplicative scalar depending only on $\pi_{3,S}$. The numerators both define invariant trilinear functionals on the tensor product $\pi_{1,S} \otimes \overline{\pi_{2,S}} \otimes \overline{\pi_{3,S}}$, with the functional appearing on the right hand side not vanishing identically. The result now follows  from the fact that the space of such functionals is one-dimensional.
\end{proof}

\begin{lemma}\label{lemma:cj7jjxnm7b}
  Let notation and assumption be as in Lemma \ref{lemma:cj3nf2fkak}.
  Assume that $\sigma$ is cuspidal or $\sigma = \sigma(\chi)$ is Eisenstein, induced by a unitary idele class character $\chi$. Then 
  \begin{equation*}
    \lvert c_{\sigma}^{(S)} \rvert^2 =
    \frac{L^{(S)} (\tfrac{1}{2}, \pi_1 \otimes \overline{\pi_2} \otimes \sigma)}{L^{(S),\ast}(1, \sigma, \Ad)^2} G^{(S)}(\pi_1, \pi_2, \sigma)
  \end{equation*}
  where $G^{(S)}$ depends only on the places $\mathfrak{p} \in S$ where $\pi_2$ is non-tempered and is of size $\exp (O (\lvert S \rvert))$.
\end{lemma}
Before giving the proof, we record the following:
\begin{remark}
  We have assumed that $\pi_2$ is cuspidal, but one can also make sense of $c_{\pi_3}^{(S)}$ when $\pi_2$ is Eisenstein.  In that case, the local condition on $\pi_2$ is automatically satisfied, and we may use the precise formula for $G^{(S)}$ given below in \eqref{thirdfactor} and the functional equation to deduce that
  \begin{equation*}
    \lvert c_\sigma^{(S)} \rvert^2 =
    \left\lvert
      \frac{L^{(S)}(\tfrac{1}{2}, \pi_1 \otimes \sigma \otimes \overline{\chi_2})}{
        L^{(S), \ast}(1, \sigma, \Ad)
      }
    \right\rvert^2.
  \end{equation*}
  Indeed, the same identity holds without squares, as one may see by direct Rankin--Selberg unfolding. Together with this identity, one can generalize the computations in the appendix of this paper to show that our results are completely consistent with the case of general, possibly non-unitary Eisenstein series treated by Motohashi, which are required if one wants to compute the fourth moment of the Riemann zeta function with shifted arguments; cf.\ \cite[(4.3.1)]{Mo1}.
\end{remark}
\begin{proof}
  (1) Consider first the case that $\sigma = \sigma(\chi)$ is Eisenstein.

  We assume throughout the proof $\chi^2 \not= 1$.  For the degenerate case where $\chi^2 = 1$, one may deduce that $c_{\sigma}^{(S)} = 0$ by continuity, or using that the unnormalized Eisenstein intertwiner vanishes in such cases.

  The proof is then an exercise in unfolding.  Write $\varphi_3 = \Eis(f_3)$.  Then, with the notation as above, and working formally for the moment, we have by \eqref{tamagawa} the standard calculation 
  \begin{equation*}
    \begin{split}
      \int_{[G]} \varphi_1\overline{ \varphi_2  \varphi_3}
      &= \int_{B(F)\backslash G(\Bbb{A})} \varphi_1 \overline{\varphi_2 f_3} =  \int_{N(\Bbb{A})\backslash G(\Bbb{A})}W_1  \overline{W_2 f_3} = \prod_{\mathfrak{p}}  \int_{N(F_{\mathfrak{p}})\backslash G (F_{\mathfrak{p}})}W_1  \overline{W_2 f_3} \, dg. 
    \end{split}
  \end{equation*}
  To compute this for an unramified place $\mathfrak{p}$, we use \eqref{eq:cj8owhrlnp} to write
  \begin{multline*}
    \int_{N(F_{\mathfrak{p}})\backslash G (F_{\mathfrak{p}})}W_1  \overline{W_2 f_3} \, dg =  \frac{1}{\zeta_{\mathfrak{p}}(2)}\int_{N(F_{\mathfrak{p}})\backslash G (F_{\mathfrak{p}})}W_1  \overline{W_2 f_3} \, dg'' 
    =\frac{1}{\zeta_{\mathfrak{p}}(2)} \int_{F^{\times}} \int_K W_1  \overline{W_2 f_3}  (a(y)k) \frac{d^{\natural}y}{|y|} \, dk
  \end{multline*}
  where here $dk$ is the probability Haar measure and $d^{\natural}y/|y|$ assigns volume one to $\mathfrak{o}_{\mathfrak{p}}^\times$. The $K$-integral can be dropped, and we obtain finally 
  \begin{equation}\label{unfolding}
    \begin{split}    
      \int_{[G]} \varphi_1\overline{ \varphi_2  \varphi_3}    & = \frac{L^{(S)}(\tfrac{1}{2}, \pi_1 \otimes \overline{\pi_2} \otimes \chi^{-1})}{\zeta_F^{(S)}(2) L^{(S)}(1, \chi^{-2})}
                                           \int_{N(F_S)\backslash G(F_S)}W_1 \overline{W_2  f_3}.
    \end{split}
  \end{equation}
  To make this calculation rigorous, one deforms $f_3$ into a flat family $f_3(s)$ with $f_3(0) = f_3$ and $f_3(s) \in \mathcal{I}(\chi |.|^s)$, performs the calculation for $\Re(s)$ sufficiently large, and then concludes via analytic continuation; indeed, this is how one establishes the analytic continuation of the $L$-function in question.

  Dividing by the definition \eqref{eq:cj4umojgpv} of $c_{\sigma }^{(S)}$ gives
  \begin{equation*}
    c_{\sigma }^{(S)} = 
    \frac{L^{(S)}(\tfrac{1}{2}, \pi_1 \otimes \overline{\pi_2} \otimes \chi^{-1})}{\zeta_F^{(S)}(2) L^{(S)}(1, \chi^{-2})}
    \frac{\int_{A(F_S)} \lvert W_3  \rvert^2}{\langle \varphi_3, \varphi_3  \rangle_{\sigma } }
    \frac{ \int_{N(F_S)\backslash G(F_S)}W_1\overline{W_2 f_3}}{   \int_{N(F_S) \backslash G(F_S)}
      W_{1} \overline{W_3 f_2} }.
  \end{equation*}

  We pause to compare $W_3$, restricted to $G(F_S)$, with the product $\prod_{\mathfrak{p} \in S} W_{f_\mathfrak{p}}$.  Let us define
    $W_{3,S} : G(F_S) \rightarrow \mathbb{C}$ by 
  \begin{equation*}
    W_{3,S}(g) = \prod_{\mathfrak{p} \in S} W_{f_\mathfrak{p}}(g_\mathfrak{p}).
  \end{equation*}
  For $g \in G(F_S)$, we have
  \begin{equation}\label{eq:cnfnkrem0h}
    W_3(g) = W_{\mathrm{Eis}(f)}(g) = \frac{1}{L^{(S)}(1,\chi^2)} W_{3,S}(g).
  \end{equation}
  This last identity is the Casselman--Shalika formula at places not in $S$.  The above formula then becomes
  \begin{equation*}
    c_{\sigma }^{(S)} = 
    \frac{L^{(S)}(\tfrac{1}{2}, \pi_1 \otimes \overline{\pi_2} \otimes \chi^{-1}) }{\zeta_F^{(S)}(2) \lvert L^{(S)}(1, \chi^{2}) \rvert^2}
    \frac{\int_{A(F_S)} \lvert W_{3,S}  \rvert^2}{\langle \varphi_3, \varphi_3  \rangle_{\sigma } }
    \frac{ \int_{N(F_S)\backslash G(F_S)}W_1\overline{W_2 f_3}}{   \int_{N(F_S) \backslash G(F_S)}
      W_{1} \overline{W_{3,S} f_2} }.
  \end{equation*}

  We now compute the squared magnitudes of each of the three factors on the right hand side.

  We start with the first factor.  Since $\chi$ is unitary and recalling \eqref{eq:cj7mbffu4s}, we have
  \begin{equation}\label{firstfactor}
    \frac{ | L^{(S)}(\tfrac{1}{2}, \pi_1 \otimes \overline{\pi_2} \otimes \chi^{-1}) |^2}{ |\zeta_F^{(S)}(2)  L^{(S)}(1, \chi^{-2})^2 |^2   } =
    L^{(S)} (\tfrac{1}{2}, \pi_1 \otimes \overline{\pi_2} \otimes \sigma)   \frac{\zeta_F^{(S),\ast}(1)^2}{\zeta_F^{(S)}(2)^2  L^{(S),\ast}(1, \sigma, \Ad)^2}.
  \end{equation}

  We now compute the second factor.  Let us write $\varphi_3 = \Eis(f)$, where $f = \otimes f_{\mathfrak{p}}$ is factorizable and, for all $\mathfrak{p} \notin S$, the local factor $f_\mathfrak{p}$ is normalized spherical (i.e., $G(\mathfrak{o}_\mathfrak{p})$-invariant and satisfying $f_\mathfrak{p}(1) = 1$).  By the definitions \eqref{eq:cj8owoe87h} and \eqref{regEuler}, we then have
  \begin{equation*}
    \langle \varphi_3, \varphi_3  \rangle_{\sigma} 
    =  \frac{\zeta_F^{(S), \ast}(1)}{\zeta_F^{(S)}(2)  }\prod_{\mathfrak{p} \in S}   \int_{B(F_{\mathfrak{p}} ) \backslash G(F_{\mathfrak{p}} )}  \lvert f_{\mathfrak{p}} \rvert^2 .
  \end{equation*}
    
  By the isometry property \eqref{isom}, it follows that
  \begin{equation*}
    \prod_{\mathfrak{p} \in S}   \int_{B(F_{\mathfrak{p}} ) \backslash G(F_{\mathfrak{p}} )} \lvert f_{\mathfrak{p}} \rvert^2
    = \prod_{\mathfrak{p} \in S} \int_{A(F_\mathfrak{p})} \left\lvert W_{f_\mathfrak{p}} \right\rvert^2.
  \end{equation*}
  Thus, in total, we obtain
  \begin{equation}\label{secondfactor}
    \frac{\int_{A(F_S)} \lvert W_{3,S} \rvert^2}{\langle \varphi_3, \varphi_3 \rangle_{\sigma } } =
    \frac{ \prod_{\mathfrak{p} \in S} \int _{A(F_\mathfrak{p})} \lvert W_{f_\mathfrak{p}} \rvert^2}{
      \frac{\zeta_F^{(S), \ast}(1)}{\zeta_F^{(S)}(2)  } \prod_{\mathfrak{p} \in S} \int_{B(F_\mathfrak{p}) \backslash G(F_\mathfrak{p})} \lvert f_\mathfrak{p}  \rvert^2
    }=
    \frac{\zeta_F^{(S)}(2) }{\zeta_F^{(S),*}(1)}.
  \end{equation}

  We now compute the squared magnitude of the third factor.  For this, we assume that $W_1, W_2, f_3$ are factorizable and work locally at a place $\mathfrak{p} \in S$, writing simply $W_1, W_2, f_3$ for their components at that place. 
  Since $f_3$ belongs to the tempered principal series at $\mathfrak{p}$, we can apply Corollary \ref{corollary:cj7mbfn7cs} to obtain
  \begin{equation*}
    \Big| { \int_{N(F_{\mathfrak{p}}) \backslash G(F_{\mathfrak{p}})}W_1\overline{W_2 f_3}}\Big|^2
    =
    \int_{  G(F_{\mathfrak{p}})}
    \langle g W_1, W_1 \rangle
    \langle g \overline{W_2}, \overline{W_2} \rangle
    \langle g \overline{f_3}, \overline{f_3} \rangle
    \, d g.
  \end{equation*}
  By~\eqref{isom}, we have  $\langle g \overline{f_2}, \overline{f_2} \rangle = \langle g \overline{W_2}, \overline{W_2} \rangle$. 
  Applying now Corollary \ref{corollary:cj7mbfn7cs} once again, this equals
  \begin{equation*}
    \begin{cases}
\Big| { \int_{N(F_{\mathfrak{p}}) \backslash G(F_{\mathfrak{p}})}W_1}\overline{f_2 W_{f_3}} \Big|^2, & \pi_2 \text{ tempered at $\mathfrak{p}$}, \\  \gamma (\tfrac{1}{2}, \pi_1 \otimes \sigma \otimes \bar{\chi} ) \Big| { \int_{N(F_{\mathfrak{p}}) \backslash G(F_{\mathfrak{p}})}W_1\overline{f_2W_{f_3}} }\Big|^2,  &\pi_2 \text{ non-tempered at $\mathfrak{p}$}, 
\end{cases}
  \end{equation*}
  and we conclude
  \begin{equation}\label{thirdfactor}
    \Big| \frac{ \int_{N(F_S)\backslash G(F_S)}W_1\overline{W_2 f_3}}{   \int_{N(F_S) \backslash G(F_S)}
      W_{1} \overline{W_{3,S} f_2} } \Big|^2 = \prod_{\substack{\mathfrak{p} \in S\\ \pi_{2, \mathfrak{p}} \text{ non-tempered}}}  \gamma (\tfrac{1}{2}, \pi_{1, \mathfrak{p}} \otimes \sigma_{ \mathfrak{p}} \otimes \bar{\chi}_{\mathfrak{p}} ).
  \end{equation}
  Let us denote the right hand side of the above by
  $G^{(S)}(\pi_1, \pi_2, \sigma)$.  Then, combining \eqref{firstfactor}, \eqref{secondfactor} and \eqref{thirdfactor} and recalling that each gamma factor is $\asymp 1$, we obtain the claimed formula. \\

(2) Let us now assume that $\sigma$ cuspidal. Then the statement follows from a careful investigation of Ichino's   triple product formula \cite{Ic}. 

Before we state the formula, we recall that Ichino uses the Tamagawa measure on $G$ which he factorizes into local measures $\prod_{\mathfrak{p}} dg_{\mathfrak{p}}^{\text{Ic}}$, say, 
so that for   all places $\mathfrak{p} \not\in S$, the volume of a maximal compact subgroup is one. From \eqref{eq:cj8owhrlnp}, we have 
\begin{equation*}
\prod_{\mathfrak{p} \notin S} d g_\mathfrak{p}^{\mathrm{Ic}} = \zeta^{(S)}_F(2) \prod_{\mathfrak{p} \notin S} d g_{\mathfrak{p} }
\end{equation*}
so that 
\begin{equation}\label{Ichinomeasure}
 \prod_{\mathfrak{p} \in S} d g_\mathfrak{p}^{\mathrm{Ic}} = \zeta^{(S)}_F(2)^{-1} \prod_{\mathfrak{p} \in S} d g_{\mathfrak{p} }. 
\end{equation}

We apply Ichino's formula to triples of cuspidal automorphic representations $\pi_1, \pi_2, \sigma\in \widehat{[G]}_{\mathrm{gen}}^{\mathrm{u}}$. For $\varphi_1 \in \pi_1$, $\varphi_2 \in \pi_2$, $\varphi_3 \in\sigma$, and any large enough finite set of places $S$, one has
\begin{equation*}
  \Big\lvert \int_{[G]} \varphi_1 \varphi_2 \varphi_3 \Big\rvert^2
  =
  \frac{1}{8}
  \frac{\zeta_F^{(S)}(2)^2 L^{(S)}(\tfrac{1}{2}, \pi_1 \otimes \pi_2 \otimes \sigma)}{
    L^{(S)}(1, \pi_1, \Ad)L^{(S)}(1, \pi_2, \Ad)L^{(S)}(1, \sigma, \Ad)
  }
    \int_{  G(F_S)} \langle g \varphi_1, \varphi_1 \rangle
  \langle g \varphi_2, \varphi_2 \rangle
  \langle g \varphi_3, \varphi_3 \rangle \, d g^{\mathrm{Ic}}.
\end{equation*}

Using that all representations equal to its complex conjugate together with \eqref{Ichinomeasure}, we obtain the variant formula
\begin{align*}
  \Big\lvert \int_{[G]} \varphi_1 \overline{\varphi_2 \varphi_3} \Big\rvert^2
  &=
  \frac{1}{8}
  \frac{\zeta_F^{(S)}(2) L^{(S)}(\tfrac{1}{2}, \pi_1 \otimes \overline{\pi_2} \otimes \sigma)}{
  L^{(S)}(1, \pi_1, \Ad) L^{(S)}(1, \pi_2, \Ad) L^{(S)}(1, \sigma, \Ad)
  }
  \int_{G(F_S)} \langle g \varphi_1, \varphi_1 \rangle
  \overline{\langle g \varphi_2, \varphi_2 \rangle
  \langle g \varphi_3, \varphi_3 \rangle} \, d g.
\end{align*}

By definition, we have 
\begin{equation}\label{eq:cj8owit0x0}
  \lvert c_{\sigma}^{(S)} \rvert^2
  =
     \frac{\lvert   \int_{[G]} \varphi_1 \overline{\varphi_2 \varphi_3}\rvert^2 }{
       \lvert \int_{N(F_S) \backslash G(F_S)}
       W_{1} \overline{W_{3} f_{2}} \rvert^2
    } 
     \frac{
      \lvert \int_{A(F_S)}
      \lvert W_{3} \rvert^2 \rvert^2
    }{ \lvert \langle \varphi_3, \varphi_3  \rangle_{\sigma} \rvert^2 }.
  \end{equation}
  By \eqref{eq:cj7ok440fl}, \eqref{eq:cj7ok542dg} and \eqref{cusp-pairing}, the final fraction on the right hand side of \eqref{eq:cj8owit0x0} evaluates to
  \begin{equation*}
 4 \Big( \frac{
    \zeta_F^{(S)}(2)
  }
  {
  2 L^{(S)}(1, \sigma, \Ad)
  } \Big)^2.
\end{equation*}
On the other hand, by what we saw above, the first fraction on the right hand side of \eqref{eq:cj8owit0x0} evaluates to
\begin{equation*}
\frac{1}{8}
  \frac{\zeta_F^{(S)}(2) L^{(S)}(\tfrac{1}{2}, \pi_1 \otimes \overline{\pi_2} \otimes \sigma)}{
   L^{(S)}(1, \pi_1, \Ad)L^{(S)}(1, \pi_2, \Ad)L^{(S)}(1, \sigma, \Ad)
  }
  \frac{\int_{  G(F_S)} \langle g \varphi_1, \varphi_1 \rangle
  \overline{\langle g \varphi_2, \varphi_2 \rangle
    \langle g \varphi_3, \varphi_3 \rangle} \, d g}{
  \lvert \int_{N(F_S) \backslash G(F_S)}
  W_{1} \overline{W_{3} f_{2}} \rvert^2
}.
\end{equation*}

We now rewrite the inner products involving $\varphi_{1}$ and $\varphi_3$ in terms of Whittaker functions.  Let us factor the Whittaker functions $W_j = W_{j,S} W_j^S$ for $j=1,3$ like we already have for $j = 2$, so that $W_j^S$ is spherical and takes the value one at the identity, and then further factor $W_{j,S} = \prod_{\mathfrak{p} \in S} W_{j, \mathfrak{p}}$.  By another application of \eqref{eq:cj7ok440fl} and \eqref{eq:cj7ok542dg}, we have
\begin{equation*}
  \langle g \varphi_j, \varphi_j  \rangle
  = 2
  \frac{
    L^{(S)}(1, \pi_j, \Ad)
  }{
    \zeta_F^{(S)}(2)
  }
  \prod_{\mathfrak{p} \in S}
  \langle g_{\mathfrak{p}} W_{j,\mathfrak{p}}, W_{j,\mathfrak{p}} \rangle
\end{equation*}
for $j = 1, 2$ and analogously for $j = 3$ with $\sigma$ in  place of $\pi_j$.  Recall that $\pi_{2,\mathfrak{p}}$ belongs to the principal series for each $\mathfrak{p} \in S$, and that we have normalized inner products in such a way that $f_{2,\mathfrak{p}} \mapsto W_{2,\mathfrak{p}}$ is an isometry (cf.\ \eqref{isom}). We obtain in this way
\begin{align*}
  \lvert c_{\sigma }^{(S) } \rvert^2 
    = & \frac{ \zeta_F^{(S)}(2)L^{(S)}(\tfrac{1}{2}, \pi_1 \otimes \overline{\pi_2} \otimes \sigma)}{8}   \frac{8}{( \zeta_F^{(S)}(2))^3} 
    \\
  & 
 \times    \frac{\int_{ G(F_S)} \langle g W_{1,S}, W_{1,S} \rangle
  \overline{\langle g f_{2,S}, f_{2,S} \rangle
    \langle g W_{3,S}, W_{3,S} \rangle} \, d g}{
  \lvert \int_{N(F_S) \backslash G(F_S)}
    W_{1} \overline{W_3 f_{2}} \rvert^2
    }   
    \Big( \frac{
    \zeta_F^{(S)}(2)
  }
  {
   L^{(S)}(1, \pi_3, \Ad)
  } \Big)^2.
\end{align*}
We now apply Corollary \ref{corollary:cj7mbfn7cs} to see that the first factor in the second line  evaluates to some quantity $G_S(\dotsb)$ as in the statement.  This completes the proof.  
\end{proof}

\subsection{Shifted convolution sums: Proof of Theorem \ref{thm-shift}}\label{sec36}

We start with a preliminary lemma and recall   the notation \eqref{def-h-vee}, \eqref{Wlambda} and \eqref{Sh}.

\begin{lemma}\label{lem37}
  Let $\pi_1, \pi_2, \sigma \in \widehat{[G]}^{\text{\rm u}}_{\text{{\rm gen}}, S}$, 
  with at least one of $\pi_1$ or $\pi_2$ cuspidal.  Assume also that $\pi_2 $ belongs to the principal series at each place.  Let $\varphi_1 \in \pi_1$ and $\varphi_2 \in \pi_2$ be unramified outside $S$.  Then, for $g \in G(F_S)$ and $0 \not= b \in \mathcal{O}_F[1/S]$, we have 
\begin{equation}\label{eq:cj3shwic31}
  \sum_{\phi \in \mathcal{B}(\sigma)}
  \frac{
   \int_{[G]} \varphi_1\overline{\varphi_2 \phi}
}{
  \langle \phi, \phi  \rangle_\sigma 
  } W_\phi(a(b) g)
  =
  c_\sigma^{(S)}
W_{\sigma}^{(S)}(a(b))
  h^{\vee}(\sigma_S, a(b) g).
\end{equation}

\end{lemma}
\begin{proof}
  Let $\sigma_S$ denote the product of local components of $\sigma$ at places inside $S$.  By tensoring with normalized unramified vectors at places outside $S$, we identify $\sigma_S$ with the unramified-outside-$S$ subspace of $\sigma$.  Since $\varphi_1$ and $\varphi_2$ are unramified outside $S$, we may assume that the sum over $\phi$ is restricted to a factorizable orthogonal basis $\mathcal{B}(\sigma_S)$ for $\sigma_S$. 
  We then evaluate the left hand side of \eqref{eq:cj3shwic31} using the definition \eqref{eq:cj4umojgpv} of $c_\sigma^{(S)}$ and the factorization
  \begin{equation*}
    W_{\phi } (a (b) g) =
    W_{\sigma}^{(S)}(a(b))
    W_{\phi, S} (a (b) g)
  \end{equation*}
  of the Whittaker function of $\phi$ into its components outside and inside $S$ (since $b$ is $S$-integral), giving
  \begin{equation*}
    c_\sigma^{(S)}
 W_{\sigma}^{(S)}(a(b))
    \sum_{W \in \mathcal{B}(\sigma_S)}
    W(a(b) g)
    \frac{    \int_{N(F_S) \backslash G(F_S)}
      W_{1,S} \overline{W f_{2,S}}
    }{
      \int_{A(F_S)}
      \lvert W \rvert^2
    }.
  \end{equation*}
  We recognize this last sum as the $S$-adic extension \eqref{Sh}
 \begin{equation*}
h^\vee (\sigma_S, a (b) g)
\end{equation*}
   of the transform  \eqref{def-h-vee}.  
 \end{proof}

 We are now ready to complete the \textbf{proof of Theorem \ref{thm-shift}}. 
 We assume that $h$ is factorizable, i.e.,
 \begin{equation*}
    h(x,y) = \prod_{\mathfrak{p} \in S} h_\mathfrak{p}(x_\mathfrak{p},y_\mathfrak{p}),
  \end{equation*}
  where each $h_\mathfrak{p}$ is a smooth function of compact support on $F_\mathfrak{p}^\times \times F_\mathfrak{p}^\times$, which moreover factors as a product
  \begin{equation}\label{eqn:cj3nfxrs67}
    h_\mathfrak{p}(x_\mathfrak{p},y_\mathfrak{p})
=W_{1,\mathfrak{p}}(a(x_\mathfrak{p}))\overline{W_{2,\mathfrak{p}}(a(y_\mathfrak{p}))}.
  \end{equation}
  
  We choose factorizable automorphic forms $\varphi_1 \in \pi_1, \varphi_2 \in \pi_2$, with Whittaker functions
  \begin{equation*}
    W_{\varphi_1}(g) = \prod_{\mathfrak{p}} W_{1,\mathfrak{p}}(g_\mathfrak{p}),
    \qquad 
    W_{\varphi_2}(g) = \prod_{\mathfrak{p}} W_{2,\mathfrak{p}}(g_\mathfrak{p})
  \end{equation*}
  such that
  \begin{itemize}
  \item for $\mathfrak{p} \in S$, each $W_{j,\mathfrak{p}}$ is as in \eqref{eqn:cj3nfxrs67}, while
  \item for $\mathfrak{p} \notin S$, each $W_{j,\mathfrak{p}}$ is normalized spherical (i.e., $\PGL_2(\mathfrak{o}_\mathfrak{p})$-invariant and taking the value $1$ at the identity).
  \end{itemize}
  This is possible by the basic theory of the Kirillov model.  We then evaluate in two ways the integral
  \begin{equation*}
    \int_{  \mathbb{A} /  F}
    \varphi_1 (n (x) )
    \overline{\varphi_2 (n (x) )}
    \psi (- b x) \, d x.
  \end{equation*}
  On the one hand, using the Whittaker--Fourier expansions of the $\varphi_j$ and Parseval's identity, we obtain the left hand side of the desired identity.  On the other hand, we can first expand the product $\varphi_1 \varphi _2 $ spectrally, so that using the notation \eqref{spec} we obtain 
  \begin{equation*}
   \int_{\widehat{[G]}_{\mathrm{gen}, S}^{\mathrm{u}}}
    \sum_{\phi \in \mathcal{B}(\pi)}
    \frac{
      \int_{[G]} \varphi_1 \overline{\varphi_2 \phi} 
      \int_{x \in \mathbb{A} / F }
      \phi(n(x)) \psi(-b x)
      \,d x 
    }{
      \langle \phi, \phi  \rangle_\pi 
    }
    \,d \pi  .
  \end{equation*}
Note that since $b \neq 0$, the second parenthetical integral vanishes when $\pi$ is one-dimensional.  We thereby reduce to considering the contribution from those $\pi$ that are generic, i.e., cuspidal or Eisenstein.

  In that case, the integral in question evaluates to $W_\phi (a (b))$.  We may assume that our (orthogonal) basis $\mathcal{B}(\pi)$ is chosen to be factorizable and that each element is either unramified outside $S$ (i.e., $\PGL_2(\mathfrak{o}_\mathfrak{p})$-invariant for each $\mathfrak{p} \notin S$) or orthogonal to anything unramified outside $S$.  Then, since $\varphi_1$ and $\varphi_2$ are unramified outside $S$, we obtain a nonzero contribution only from those $\phi$ that are unramified outside $S$.
  
  An appeal to Lemma \ref{lem37} completes the proof.

\subsection{Local preliminaries for the proof of Theorem \ref{thm-moment}}\label{sec37a}
Let $F$ be a local field, write $G := \PGL_2(F)$ as usual, and let $\pi_1, \pi_2 \in \widehat{G}_{\mathrm{gen}}$, realized as usual in their $\psi$-Whittaker models with respect to a nontrivial unitary character $\psi$ of $F$.

We will on several occasions appeal to the standard theory of local Hecke/Jacquet--Langlands integrals, so let us recall the main conclusions of that theory.  Let $\pi \in \widehat{G}^{\mathrm{u}}_{\mathrm{gen}}$ and $\chi \in \widehat{A}$.  Let us realize $\pi$ in its $\psi$-Whittaker model, as a space of functions $W$ on $G$.  Then for $W \in \pi$ and $\chi \in \widehat{A}$ with $\Re(\chi)$ sufficiently large, the integral $\int_A W \chi := \int_{A} W(a) \chi(a) \, d a$ converges absolutely.  It moreover admits a meromorphic continuation to the complex plane, with the further property that the ratio
\begin{equation*}
  \frac{\int_A W \chi }{ L (\tfrac{1}{2} , \pi \otimes \chi)}
\end{equation*}
defines an entire function of $\chi$.  Finally, if $\psi$, $\pi$ and $\chi$ are unramified, and if $W$ is the unramified vector with $W(1) = 1$, then the above ratio evaluates (in view of our measure normalization on $A$) to $1 / \zeta_F(1)$.

Recall that, to $W_1 \in \pi_1$ and $W_2 \in \pi_2$, we have in \eqref{def-h} attached the function $h$ on $F^\times \times F^\times$ by the formula $h(y_1,y_2) = W_1(a(y_1)) \overline{W_2(a(y_2))}$.  We also defined a dual function $h^\vee$ by \eqref{def-h-vee}, as well as $h^{\sharp}$ in \eqref{hsharp}. Since $h^\vee (\sigma, \bullet)$ is, by definition, an element of the Whittaker model of $\sigma$, we can appeal to the theory of local Hecke/Jacquet--Langlands integrals, which implies that the ratio
$h^\sharp (\sigma, \chi )/L (\tfrac{1}{2}, \sigma \otimes \chi)$ 
defines an entire function of $\chi$.

For notational convenience, we define some further transform. 
First, we let 
\begin{equation}\label{hflat}
  h^{\flat} : \widehat{F^{\times}} 
  \times F^\times \rightarrow \mathbb{C}, \quad (\chi, y) \mapsto \int_{F^\times } h (z, y z ) \chi(z) \,d^\times z = \int_{F^\times} W_1(a(z)) \overline{W_2(a(y z))} \chi(z) \, d^\times z.
\end{equation}
We next define
\begin{equation*}
  \tilde{h} :   \widehat{F^{\times}}  \times  \widehat{F^{\times}}  \rightarrow \mathbb{C}, \quad (\chi, \eta) \mapsto \int_{F^\times } h^\flat (\chi, y ) \eta(y) \,d^\times y = \Big( \int_{F^\times} W_1 \chi \eta^{-1}  \Big) \Big( \int_{F^\times} W_2 \eta  \Big).
\end{equation*}
Both of these definitions make sense for all arguments provided that $W_1$ and $W_2$ are compactly supported. To lighten the notation in the last formula, we made the convention to identity $W_j(a(.))$ with $W_j$ for $j = 1, 2$ which we will continue to use in similar situations.

Finally, for a character $\chi \in \widehat{F^{\times}}$ and a unitary character $\eta$, we define 
\begin{equation*}
  h^{\heartsuit}(\eta, \chi)
  :=
  \int_{F^\times }
  \sum_{f_3 \in \mathcal{B}(\mathcal{I}(\eta))}
  \Big( \int_{N \backslash G} W_1 \overline{W_2 f_3 }                 \Big)
  W_{f_3}(a(b)) \chi (b) \,d^\times b
\end{equation*}
which can be extended to non-unitary characters $\eta$ in the usual way. 

All local transforms in the subsection can be extended to $S$-adic versions in the obvious way. 

In a special configuration, the transform $h^{\heartsuit}$ simplifies considerably.
\begin{lemma}\label{lemma:cnfnkqp8fc}
  We have
  \begin{equation*}
  h^{\heartsuit}(|.|^{-1/2}\chi, \chi) =  \int_{  F^\times}
  ( w W_1 \cdot \overline{w W_2})(a(y))
  \chi^{-1}(y)
  \, d^\times y.
\end{equation*}
\end{lemma}
\begin{proof}
  It suffices to prove the formula in the case $\Re(\chi) = 1/2$, as the general statement follows then by meromorphic continuation. This assumption has the notational advantage that the first argument of $h^{\heartsuit}$ is unitary.

  Recall by \eqref{defn-intertwiner-explicit} that
  \begin{equation*}
  W_{f_3}(a(y)) = |y| \chi^{-1}(y) \int_{ F}
  f_3(w n(x)) \psi(- y x) \, d x
\end{equation*}
for $f_3 \in \mathcal{B}( \mathcal{I}(|.|^{-1/2} \chi))$, so that by Fourier inversion
\begin{equation*}
 f_3(w) =  \int_{F  }  \frac{W_{f_3}(a(b))}{|b|} \chi (b)\, db =   \int_{F^\times }
  W_{f_3}(a(b)) \chi (b) \,d^\times b 
\end{equation*}
and hence
\begin{align*}
  h^{\heartsuit}(\lvert . \rvert^{-1/2} \chi, \chi)
  &=\hspace{-0.25cm}
    \sum_{f_3 \in \mathcal{B}(\mathcal{I}(\lvert . \rvert^{-1/2} \chi))}
    \Big( \int_{N \backslash G} W_1 \overline{W_2 f_3 }                 \Big)
    f_3(w) = \hspace{-0.25cm}
                \sum_{f_3 \in \mathcal{B}(\mathcal{I}(\lvert . \rvert^{-1/2} \chi))}
    \Big( \int_{N \backslash G} w W_1\cdot \overline{w W_2 \cdot  f_3 }                 \Big)
    f_3(1)
\end{align*}
since $w^2 = 1$.  We have (recall that $\Re(\chi) = 1/2$)
\begin{align*}
  \int_{N \backslash G} w W_1 \overline{w W_2 \cdot  f_3 }
  &=
  \int_{ F}
  \int_{  F^\times}
  (w W_1 \cdot \overline{w W_2})(a(y) w n(x))
  \overline{f_3(a(y) w n(x))}
  \, d x
  \, \frac{d^\times y}{\lvert y \rvert} \\
  &=
  \int_{ F}
  \int_{ F^\times}
  (w W_1 \cdot \overline{w W_2})(a(y) w n(x))
    |y|^{1/2} \overline{|y|^{-1/2} \chi(y) f_3(w n(x))}
    \, d x
    \, \frac{d^\times y}{\lvert y \rvert} \\
  &=
    \int_{ F}
    \int_{  F^\times}
    (w W_1 \cdot \overline{w W_2})(a(y) w n(x))
    \chi^{-1}(y) \overline{ f_3(w n(x))}
    \, d x
    \, d^\times y.
\end{align*}
This last formula may be understood as the inner product of $f_3$ against the element of $\mathcal{I}(|.|^{-1/2} \chi)$ given by
\begin{equation*}
  g \mapsto \int_{F^\times}
  w W_1(a(t) g) \cdot \overline{w W_2(a(t) g)}
  \chi^{-1}(t)\, d^\times t.
\end{equation*}
It follows that
\begin{equation*}
  \sum_{f_3 \in \mathcal{B}(\mathcal{I}(\lvert . \rvert^{-1/2} \chi))}
  \Big( \int_{N \backslash G} w W_1\cdot \overline{w W_2 \cdot  f_3 }                 \Big)
  f_3(1) =  \int_{  F^\times}
  (w W_1 \cdot \overline{w W_2})(a(y))
  \chi^{-1}(y)
  \, d^\times y.
\end{equation*}
as claimed.
\end{proof}

\subsection{Moments of \texorpdfstring{$L$}{L}-functions}\label{sec37}
In this subsection we prove Theorem \ref{thm-moment}. Let $F$ be a number field.  Let $S$ be a finite set of places of $F$, containing all places that are archimedean or at which $F$ is ramified.  Let $\pi_1$ and $\pi_2$ be cuspidal automorphic representations for $G(F)$, unramified outside $S$.  Let $\varphi_1 \in \pi_1$ and $\varphi_2 \in \pi_2$ be factorizable vectors, unramified outside $S$, with factorizable Whittaker functions
\begin{equation*}
  W_j (g) = \prod_{\mathfrak{p}} W_{j, \mathfrak{p} } (g_{\mathfrak{p} }),
\end{equation*}
normalized as in the discussion before \eqref{Wlambda}.  Let $\chi : [A] \rightarrow \mathbb{C}^\times$ be a Hecke character, also unramified outside $S$.  By our cuspidality assumption, the integral
\begin{equation}\label{eq:cj401v58in}
  \mathcal{P}(\chi) :=   \int_{[A]} \varphi_1 \overline{\varphi_2} \chi
\end{equation}
converges absolutely.  The proof of Theorem \ref{thm-moment} consists of expanding this integral spectrally in two different ways.

First, we apply the Mellin--Parseval relation on $[A]$ to obtain
\begin{equation*}
\mathcal{P}(\chi) =   \int_{ \widehat{ [A]}^{\mathrm{u}} }
  \Big( \int_{[A]} \varphi_1 \chi \eta^{-1}  \Big)
  \Big( \int_{[A]} \overline{\varphi_2} \eta \Big)
  \, d \eta.
\end{equation*}
Here, because the representations are cuspidal and the vectors are smooth, the integrand defines a holomorphic function of $\eta$, of rapid decay, and there are no convergence issues.  Note that the integrand is non-zero only for $\eta$ unramified outside $S$. We may unfold each of the two integrals over $[A]$, giving
\begin{equation}\label{parseval}
  \mathcal{P}(\chi) = \int_{ \widehat{ [A]}^{\mathrm{u}} }
  \frac{L^{(S)}(\tfrac{1}{2}, \pi_1 \otimes \chi \otimes \eta^{-1} )  L^{(S)}(\tfrac{1}{2}, \overline{\pi_2} \otimes \eta)  }{\zeta_F^{(S), \ast}(1)^2}
  \int _{F_S^\times} W_1 \chi \eta^{-1}
  \int _{F_S^\times} \overline{W_2}\eta.
\end{equation}

On the other hand, we can spectrally expand over $[G]$.  As a preliminary maneuver, we consider the contribution of the constant term $(\varphi_1 \overline{\varphi_2})_N$ of $\varphi_1 \overline{\varphi_2}$ to the integral \eqref{eq:cj401v58in}, where we adopt, for $\Psi \in C^\infty(N(F) \backslash G(\mathbb{A}))$, the customary notation
\begin{equation}\label{eq:cnehsj3zhr}
  \Psi_N(g) := \int_{ N(F)\backslash N(\Bbb{A})} \Psi(n g) \, d n.
\end{equation}

\begin{lemma}\label{lemma:cnehsf3ch3}
 For $\Re(\chi) > 0$ we have
 \begin{equation*}
   \int_{ [A]} (\varphi _1 \overline{\varphi_2})_N \chi =
   h_S^\flat (\chi, 1)
   \frac{L^{(S) } (1, \pi_1 \otimes \overline{\pi_2} \otimes \chi)}{\zeta^{(S), \ast}_F(1) L^{(S)} (2, \chi^2)},
 \end{equation*}
 which can be extended as an equality of meromorphic functions. 
\end{lemma}
\begin{proof}
In view of the Whittaker expansions
\begin{equation*}
\varphi_j(n(x) a(y)) = \sum_{b \in F^\times } \psi(b x) W_j(a(b y)),
\end{equation*}
the constant term in question evaluates to 
\begin{equation*}
(\varphi_1 \overline{\varphi_2})_N(a(y)) = \sum_{b \in F^\times } W_1(b y) \overline{W_2}(b y).
\end{equation*}
We consider the double sum/integral
\begin{equation}\label{eq:cnee0vfega}
  \int_{  [A] } \chi(y) \sum_{b \in F^\times } W_1(a(b y)) \overline{W_2}(a(b y)) \,d^\times y.
\end{equation}
By a Rankin--Selberg estimate, we obtain 
\begin{equation*}
  \sum_{b \in F^\times } |W_1(a(b y)) \overline{W_2}(a(b y))| \ll_{\eps,N, W_1, W_2}
  \min(\lvert y \rvert^{-\eps}, y^{-N})
\end{equation*}
for any $N, \varepsilon > 0$, which shows that \eqref{eq:cnee0vfega} converges absolutely for $\Re(\chi) > 0$.

Integrating against $\chi(y)$ (for $\Re(\chi) > 0$) over $y \in [A]$ gives 
\begin{equation}\label{eq:cj401wpnl7}
  \int_{[A]} (\varphi _1 \overline{\varphi_2})_N \chi
  =
  \int_{ \mathbb{A}^{\times}} W_1(a(y)) \overline{W_2}(a(y)) \chi(y) \,d^\times y.
\end{equation}

The right hand side of \eqref{eq:cj401wpnl7} factors as a convergent product over the places and can be  computed using \eqref{haar}.  The contribution from places inside $S$ is $ h_S^\flat (\chi, 1)$, while that from places not in $S$ is $ L^{(S) } (1, \pi_1 \otimes \overline{\pi_2} \otimes \chi)/L^{(S)} (2, \chi^2)$ (see for instance \cite[Prop 3.8.1]{Bump}).  This completes the proof.
\end{proof}

It remains to treat the contribution of the difference $\varphi_1 \overline{\varphi_2} - (\varphi_1 \overline{\varphi_2} )_N $.
Recall the spectral expansion \eqref{spec}, which we may rewrite as follows: for $\Psi \in C_c^\infty ([G])$,
we have
\begin{equation}\label{eq:cnehsletef}
  \Psi(g) =
  \int_{\widehat{[G]}^{\mathrm{u}}}
  \Psi_{\sigma}(g) \, d \sigma,
  \qquad
  \Psi_{\sigma}(g) := \sum_{\phi \in \mathcal{B}(\sigma)}
  \phi(g)
  \frac{
    \int_{[G] } \Psi \bar{\phi}
  }{
    \langle \phi, \phi \rangle_{\sigma}
  }.
\end{equation}
The compact integral \eqref{eq:cnehsj3zhr} that defines the constant term commutes with this spectral expansion, so we may write
\begin{equation*}
  \Psi_N (g) =
  \int_{\widehat{[G]}^{\mathrm{u}}}
  \Psi_{\sigma,N}(g) \, d \sigma,
\end{equation*}
where $\Psi_{\sigma,N}$ is attached to $\Psi_\sigma$ as in \eqref{eq:cnehsj3zhr}.

With this notation, the difference in question expands for $\Re(\chi) > 1/2$ as the following absolutely convergent double integral:
\begin{equation}\label{eq:cnehsrdbfj}
  I(\chi)
  :=
  \int_{ [A]}
  \int_{ \widehat{[G]}^{\mathrm{u}}_{\mathrm{gen}}}
  \left( ( \varphi_1 \overline{\varphi_2})_\sigma(y)  - (\varphi_1 \overline{\varphi_2})_{\sigma, N}(y)  \right) \chi(y) \, d \sigma \, d ^\times y =:   \int_{\widehat{[G]}^{\mathrm{u}}_{\mathrm{gen}}} I_\sigma(\chi) \, d \sigma, 
\end{equation}
say, where 
\begin{equation}\label{isigma}
  I_\sigma(\chi) :=
\int_{ [A]}
  \left( ( \varphi_1 \overline{\varphi_2})_\sigma(y)  - (\varphi_1 \overline{\varphi_2})_{\sigma, N}(y)  \right) \chi(y) \, d ^\times y.
\end{equation}
For details concerning the convergence, we refer to the final paragraph of \cite[\S8.8]{Ne}; the main point is that for $\Psi = \varphi_1\overline{\varphi_2}$, $d \geq 0$ and $\eps >0$, we have
\begin{equation*}
  \Psi_\sigma(y)  - \Psi_{\sigma, N}(y)
  \ll_{\Psi,d, \eps} \frac{\min( \lvert y \rvert^{- 1/2 - \eps }, \lvert y \rvert^{- d} )}{C(\sigma)^{d}}
\end{equation*}
where $C(\sigma)$ denotes the analytic conductor of $\sigma$. In verifying this, the only subtle point is to understand the $y \rightarrow 0$ asymptotics in the case that $\sigma$ is Eisenstein; the factor $\lvert y \rvert^{- 1/2 - \eps}$ comes from estimating the constant term of $\Psi_\sigma(w g w) = \Psi_\sigma(g w)$.  In particular, in the special case that $\sigma$ is cuspidal, the above estimate may be refined to
\begin{equation*}
  \Psi_\sigma(y) = \Psi_\sigma(y)  - \Psi_{\sigma, N}(y)
  \ll_{\Psi,d, \eps} \frac{\min \left( \lvert y \rvert^{d}, \lvert y \rvert^{- d} \right)}{C(\sigma)^{d}}.
\end{equation*}

\begin{lemma}\label{lemma:cnehsf3kvv}
The function $I_\sigma(\chi)$, defined initially for $\Re(\chi) > 1/2$, admits a meromorphic continuation to all $\chi$. More precisely, the ratio $I_\sigma(\chi) / \Lambda (\tfrac{1}{2},\sigma \otimes \chi)$ is holomorphic in $\chi$. We have
\begin{equation}\label{claim1}
  I_\sigma(\chi) = \frac{L^{(S)}( \tfrac{1}{2}, \sigma \otimes \chi)}{ \zeta_F^{(S), *}(1)}
  c_\sigma^{(S)} h_S^\sharp (\sigma, \chi ).
\end{equation}
If $\sigma = \sigma(\eta)$ is Eisenstein, then
\begin{equation}\label{claim2}
  I_{\sigma(\eta)}(\chi) =
  \frac{L ^{(S)}(\tfrac{1}{2}, \pi_1 \otimes \overline{\pi_2} \otimes \eta^{-1}) L^{(S)}(\tfrac{1}{2},\sigma(\eta) \otimes \chi)  }{ \zeta^{(S), \ast}_F(1)^2 L^{(S)}(1, \eta^{2}) L^{(S)}(1, \eta^{-2})}
  h_S^{\heartsuit}(\eta, \chi).
  \end{equation}
\end{lemma}
\begin{proof}
  Abbreviate $\Psi := \varphi_1 \overline{\varphi_2}$, and, for ${\sigma \in \widehat{[G]}^{\mathrm{u}}_{\mathrm{gen}}}$, let us write simply $W_\sigma := W_{\Psi_\sigma}$ for the associated Whittaker function, which may be given explicitly by integrating the definition \eqref{eq:cnehsletef}, as follows:
  \begin{equation}\label{eq:cnehslktkn}
    W_\sigma (g) = \sum_{\phi \in \mathcal{B}(\sigma)}
    \frac{
      \int_{[G] } \varphi_1 \overline{\varphi_2 \phi}
    }{
      \langle \phi, \phi \rangle_{\sigma}
    } W_\phi(g).
  \end{equation}
  The spectral component $\Psi_\sigma$ then admits a Whittaker expansion
  \begin{equation*}
    \Psi_\sigma (g) = \Psi_{\sigma, N} (g) + \sum_{
      \alpha \in A(F)
    }
    W_\sigma\big(\big(
\begin{smallmatrix}
\alpha & \\ & 1
\end{smallmatrix}
\big) g\big).
  \end{equation*}
  By unfolding, it follows that
  \begin{equation*}
    I_\sigma (\chi) = \int_{ \mathbb{A}^{\times}} W_{\sigma}(a(y)) \chi(y) \, d^{\times} y.
  \end{equation*}
  The local components $\Psi_\sigma$, being defined via projection from the factorizable forms $\varphi_1$ and $\varphi_2$, are likewise factorizable, and may be written $W_\sigma = W_{\sigma,S} W_\sigma^{S}$, where
  \begin{itemize}
  \item $W_{\sigma}^{S}$ is unramified and takes the value $1$ at $1$, and
  \item $W_{\sigma,S}$ is simply the restriction of $W_\sigma$ to $G(F_S)$.
  \end{itemize}
  The integral $I_\sigma(\chi)$ correspondingly factorizes.  Recall that $\chi$ is also assumed unramified outside $S$.  By the standard evaluation of unramified local Hecke/Jacquet--Langlands integrals (see~\cite[Prop 3.5.3]{Bump}), we then have
  \begin{equation*}
    I_\sigma (\chi) =
    \frac{L^{(S)}(\tfrac{1}{2}, \sigma \otimes \chi)}{ \zeta_F^{(S), *}(1)}
    \int_{ F_S^{\times}} W_{\sigma,S}(a(y)) \chi_S(y) \, d^{\times} y.
  \end{equation*}
  From here, we see already, by the theory of Hecke/Jacquet--Langlands integrals, that $I_\sigma(\chi)$ admits a meromorphic continuation to all $\chi$, with the property that $I_\sigma(\chi) / \Lambda(\tfrac{1}{2},\sigma \otimes \chi)$ is holomorphic. 
  
  It remains to evaluate this last integral.  We do so by appeal to Lemma \ref{lem37}, whose left hand side, specialized to $b = 1$, is $W_{\sigma,S}(g)$ (cf.\ \eqref{eq:cnehslktkn}), giving $W_{\sigma,S}(a(y)) = c_\sigma^{(S)} h^\vee (\sigma_S, y)$.  Substituting into the above and recalling the definition \eqref{hsharp} of $h^\sharp$, we obtain \eqref{claim1}.
  
  To prove \eqref{claim2} we start from the definition \eqref{isigma}. 
  We have
\begin{equation*}
  (\varphi_1 \overline{\varphi_2})_\sigma(y)
  =
  \sum_{f_3 \in \mathcal{B}(\mathcal{I}(\eta))}
  \Big( \int_{[G]} \varphi_1 \overline{\varphi_2 \Eis(f_3)} \Big)
  \Eis(f_3)(y).
\end{equation*}
Since $\varphi_1$ and $\varphi_2$ are unramified outside $S$, we can restrict the sum to $f_3$ in the unramified-outside-$S$ subspace of $\mathcal{I}(\eta)$, which we may identify further with $\mathcal{I}(\eta_S)$ via the restriction map from $G(\mathbb{A})$ to $G(F_S)$.  For $f_3$ in this subspace, the inner products may be compared using \eqref{regEuler}, which gives
\begin{equation*}
  \frac{ \langle f_{3}, f_{3} \rangle_{\mathcal{I}(\eta_S)}}{\langle f_3, f_3 \rangle_{\mathcal{I}(\eta)}}
  =
  \frac{\zeta_F^{(S)}(2)}{\zeta_F^{(S), \ast}(1)}.
\end{equation*}
Next, we subtract off the constant term and integrate against $\chi(y)$, which unfolds to
\begin{equation*}
  \int_{[A]} \left( \Eis(f_3) - \Eis(f_3)_N \right)(y) \chi(y) \, d^\times y =   \frac{L^{(S)}(\tfrac{1}{2},\sigma(\eta) \otimes \chi) }{\zeta^{(S), \ast}_F(1)}\int_{F_S^\times} W_{\Eis(f_3)} \chi.
\end{equation*}
Passing from $W_{\Eis(f_3)}$ to $W_{f_3}$, as in \eqref{eq:cnfnkrem0h}, introduces an additional factor $L^{(S)}(1, \eta^2)^{-1}$.  Next, the triple product integral unfolds, as in \eqref{unfolding}, to  
\begin{equation*}
 \int_{[G]} \varphi_1 \overline{\varphi_2 \Eis(f_3)}  =  \frac{L ^{(S)}(\tfrac{1}{2}, \pi_1 \otimes \overline{\pi_2} \otimes \eta^{-1})  }{\zeta_F^{(S)}(2) L^{(S)}(1, \eta^{-2})}
  \int_{N(F_S) \backslash G(F_S)} W_1 \overline{W_2 f_3}.
\end{equation*}
Combining these formulae completes the proof of \eqref{claim2}. 
\end{proof}

We note that $I_\sigma(\chi)$ is holomorphic on the line $\Re(\chi) = 0$.  In particular, for $\chi$ in a neighborhood of $\Re(\chi) = 0$, we may define the absolutely convergent integral
\begin{equation}\label{eq:cnehsrvptq}
  I_0(\chi) := \int_{  \widehat{[G]}_{\mathrm{gen}}^{\mathrm{u}}} I_\sigma(\chi) \, d \sigma.
\end{equation}
This integral converges absolutely for $\Re(\chi) > - (\tfrac{1}{2} - \frac{7}{64})$ due to Lemma \ref{lemma:cnehsf3kvv},   Lemma \ref{property-h-vee}, the Kim-Sarnak bound and the convexity bound for $L$-functions.

We note that $I_0$ is \emph{not} obviously the same as $I$, despite the fact that they are defined by identical integrals: recall that $I$ is defined initially for $\Re(\chi) > 1/2$ and then extended by meromorphic continuation, while $I_0$ is defined initially for $\Re(\chi)$ close to zero.  In general, meromorphic continuation does not commute with integration.  By carefully commuting these two processes here, one obtains additional terms, described by the following lemma.

\begin{lemma}\label{lemma:cnehspj54y}
  The function $I(\chi)$, defined initially for $\Re(\chi) > 1/2$ as above, admits a meromorphic continuation to a neighborhood of the region $\Re(\chi) \geq 0$.  For $\chi$ in a neighborhood of the region $\Re(\chi) = 0$, we have
  \begin{equation*}
I(\chi) = \mathcal{M}(\chi) + I_0(\chi),
\end{equation*}
where for $\chi^2 \not= 1$ we have 
\begin{equation*}
  \mathcal{M}(\chi) := 
   \frac{L ^{(S)}(1, \pi_1 \otimes \overline{\pi_2} \otimes \chi^{-1})  }{
    \zeta_F^{(S), \ast}(1) L^{(S)}(2, \chi^{-2})}
  h_S^{\heartsuit}(|.|^{-1/2} \chi, \chi)
\end{equation*}
\end{lemma}

As preparation for the proof, we record the following general argument:
\begin{lemma}\label{lemma:cnehsw1yl1}
  Let $f(s,w)$ be a meromorphic function on $S\times W \subseteq \Bbb{C}\times \Bbb{C}$ where $W$ is an open right half plane containing the line $\Re w = 1/2$ and $S$ is a neighborhood of the line $\Re s = 0$. Suppose that $ f(s, w)(-\tfrac{1}{2} + s + w)(-\tfrac{1}{2}- s + w)$ is holomorphic on $S\times W$ and $f(., w)$ is rapidly decaying on $\Re s = 0$. For $\Re w > 1/2$, let
  \begin{equation*}
    J(w) =  \int_{(0)} f(s, w) \frac{ds}{2\pi i}.
  \end{equation*}
  Then $J$ has holomorphic continuation to $W \setminus\{1/2\}$, and for $\Re w < 1/2$, $w \in W$ we have 
  \begin{equation}\label{cont}
    J(w) =  \int_{(0)} f(s, w) \frac{ds}{2\pi i} + \underset{s = 1/2 - w}{\text{\rm res}}f(s, w) -\underset{s =  w-1/2 }{\text{\rm res}} f(s, w). 
  \end{equation}
\end{lemma}

\begin{proof}
  Let $C> 0$ be arbitrary. It suffices to show meromorphic continuation to $w \in W$, $|\Im w| < C$. Fix $\varepsilon  = \varepsilon (C) > 0$ so that  $\{ |\Im s| \leq C + 1, |\Re s | \leq 3\varepsilon\} \subseteq S$. Let initially be $\tfrac{1}{2} + \varepsilon < \Re w < \tfrac{1}{2} + 2\varepsilon$, $|\Im w | \leq C$.  Let $\mathcal{C}$ be the contour 
\begin{displaymath}
\begin{split}
    [-i\infty, -i(C+1)] \cup [-i(C+1), -i(C+1) - 3\varepsilon]& \cup [-i(C+1) - 3\varepsilon, i(C+1) - 3\varepsilon] \\
    &   \cup [i(C+1) - 3\varepsilon, i(C+1) ] \cup [i(C+1) , i \infty].
    \end{split}
    \end{displaymath}
Then
\begin{equation*}
J(w)  = \int_{\mathcal{C}}  f(s, w) \frac{ds}{2\pi i}  + \underset{s = 1/2 - w}{\text{\rm res}}f(s, w). 
\end{equation*}
The right hand side gives a holomorphic continuation to the larger set
\begin{equation*}
\tfrac{1}{2} - 2\varepsilon < \Re w < \tfrac{1}{2} + 2\varepsilon, \quad |\Im w | \leq C, \quad w \not= \tfrac{1}{2}.
\end{equation*}
 Let us restrict this to $\tfrac{1}{2} - 2\varepsilon < \Re w < \tfrac{1}{2} -\varepsilon$. For such $w$ we shift the contour back getting the formula \eqref{cont}, 
which provides analytic continuation to $\Re w < 1/2 + \varepsilon$, $w \in W$, $|\Im w | < C$. 
\end{proof}

\begin{proof}
[Proof of Lemma \ref{lemma:cnehspj54y}]
  Write $I = I^{\mathrm{cusp}} + I^{\mathrm{Eis}}$  and $I_0 = I_0^{\mathrm{cusp}} + I_0^{\mathrm{Eis}}$ for the respective cuspidal and Eisenstein contributions to the definitions \eqref{eq:cnehsrdbfj} and \eqref{eq:cnehsrvptq} of $I$ and $I_0$.  The same arguments as above give the meromorphic continuations of these contributions.  The double integral defining $I_0^{\mathrm{cusp}}$ converges absolutely for all $\chi$, hence defines an entire function of $\chi$.  It follows that $I_0^{\mathrm{cusp}} = I^{\mathrm{cusp} }$.  It remains to verify that $I^{\mathrm{Eis}} - I_0^{\mathrm{Eis}} = \mathcal{M}$ on a neighborhood of $\{\chi : \Re(\chi) = 0\}$.

  We can parametrize the relevant Eisenstein representations $\sigma$ as $\sigma(\eta)$, where $\eta \in [\widehat{A}]^{\rm u}_S$ where the subscript $S$ denotes that $\eta$ is unramified outside $S$. Let us initially take $\Re(\chi) > 1/2$. 
  In view of the definition \eqref{eq:cnehsrdbfj}, 
  we have  
  \begin{equation*}
    I^{\Eis}(\chi) = \frac{1}{2} 
    \int_{  \widehat{[A]}^{\rm u}_S} I_{\sigma(\eta)}(\chi) \, d \eta
\end{equation*}
We recall Lemma \ref{lem31} for the factor $1/2$ and note that the measure $d\eta$ is of the form $dt/(2\pi i)$ on the continuous part $|.|^{it}$. 

Here we can insert both \eqref{claim1} with $\sigma = \sigma(\eta)$ and \eqref{claim2} for the integrand. For the latter, we 
observe that, for $\eta$ in some small enough neighbourhood of any given compact subset of $\widehat{[A]}^{\mathrm{u}}$, the denominator is holomorphic and nonzero. 
This is a reformulation of the prime number theorem for Hecke characters.  

We now apply Lemma \ref{lemma:cnehsw1yl1} to obtain an explicit formula for $I^{\Eis}(\chi)$ when $\Re(\chi) = 0$. As can be seen from the numerator in \eqref{claim2}, the only two relevant poles are at $\eta = \chi^{-1} |.|^{1/2}$ and $\eta = \chi | . |^{- 1/2}$. Since $\sigma(\eta) = \sigma(\eta^{-1})$, we see from \eqref{claim1} that the two residues are negatives of each other. Thus
\begin{equation*}
\mathcal{M}(\chi) = \underset{\eta = \chi^{-1} |.|^{1/2}}{\text{res}} \frac{I_{\sigma(\eta)}(\chi)}{2} - \underset{\eta = \chi |.|^{-1/2}}{\text{res}} \frac{I_{\sigma(\eta)}(\chi)}{2} =  - \underset{\eta = \chi |.|^{-1/2}}{\text{res}}  I_{\sigma(\eta)}(\chi),
\end{equation*}  
and the desired formula frollows from \eqref{claim2}. 
\end{proof}

Combining \eqref{eq:cj401v58in}, \eqref{parseval}, Lemma \ref{lemma:cnehsf3ch3}, \eqref{eq:cnehsrdbfj}, Lemma \ref{lemma:cnehspj54y},  \eqref{eq:cnehsrvptq} and Lemma \ref{lemma:cnehsf3kvv}, we have shown the identity
\begin{displaymath}
\begin{split}
 \int_{ \widehat{ [A]}^{\mathrm{u}} } &
  \frac{L^{(S)}(\tfrac{1}{2}, \pi_1 \otimes \chi \otimes \eta^{-1} )  L^{(S)}(\tfrac{1}{2}, \overline{\pi_2} \otimes \eta)  }{\zeta_F^{(S), \ast}(1)^2}
  \int _{F_S^\times} W_1 \chi \eta^{-1}
  \int _{F_S^\times} \overline{W_2}\eta \\
  &= h_S^\flat (\chi, 1)
\frac{L^{(S) } (1, \pi_1 \otimes \overline{\pi_2} \otimes \chi)}{\zeta_F^{(S), \ast}(1)L^{(S)} (2, \chi^2)}+  \frac{L ^{(S)}(1, \pi_1 \otimes \overline{\pi_2} \otimes \chi^{-1})  }{
    \zeta_F^{(S), \ast}(1) L^{(S)}(2, \chi^{-2})}
  h_S^{\heartsuit}(|.|^{-1/2} \chi, \chi)\\
&  + \int_{  \widehat{[G]}_{\mathrm{gen}}^{\mathrm{u}}} \frac{L^{(S)}(\tfrac{1}{2}, \sigma \otimes \chi )}{ \zeta_F^{(S), *}(1)}
 c_\sigma^{(S)} h_S^\sharp (\sigma, \chi ) \, d \sigma
 \end{split}
\end{displaymath}
for $\chi$ in a neighborhood of $\Re(\chi) = 0$, $\chi^2 \not= 1$. We now cancel a factor $\zeta_F^{(S), \ast}(1)$, recall \eqref{hflat} and Lemma \ref{lemma:cnfnkqp8fc} and observe that with \eqref{capitalH} and \eqref{wetachi} the integral over $F_S^{\times} \times F_S^{\times}$ in the first line is simply $w(\eta, \chi)$.  This completes the proof of Theorem \ref{thm-moment}.

\section{An application: proof of Theorem \ref{appl}}
For simplicity, we use in the following the Selberg eigenvalue conjecture, known for ${\rm SL}_2(\Bbb{Z})$.  It is straightforward to generalize the following arguments to congruence subgroups, where approximations towards the Ramanujan--Petersson conjecture have to be used.

Let $t_f, t_g$ denote the spectral parameters of $f$ and $g$. 
Take any decomposition $V = V_1V_2$.  Then we can write 
\begin{equation*}
\mathcal{S}_{f, g, V}(X, Y, b)  = 
\sum_{n_1 - n_2 = b} \frac{\lambda_f(n_1)\lambda_g(n_2)}{\sqrt{|n_1n_2|}} h(n_1, n_2),
\end{equation*}
where
\begin{equation*}
h(n_1, n_2) = \sqrt{|n_1n_2|} V_1\Big( \frac{n_1 -X}{Y}\Big)V_2\Big( \frac{n_2+b -X}{Y}\Big).
\end{equation*}
By Theorem \ref{thm-shift}, this is majorized by 
\begin{equation}\label{3}
  \int_{ [\widehat{G}]^{\mathrm{u}}_{\mathrm{gen}, \{\infty\}}} \frac{|\lambda_{\pi}(b)|}{b^{1/2}}  \frac{\sqrt{|L(\tfrac{1}{2}, \pi \otimes \pi_f \otimes \pi_g)|}}{L(1, \pi, \text{Ad})}  |h^{\vee}(\pi, b)| \, d \pi.
\end{equation}
We estimate $h^{\vee}(\pi, b)$, recalling that $b$ is a natural number between $1$ and $X/2$. In the notation of Theorem \ref{thm1}, let us write $\chi = \text{sgn}^{\delta}|.|^s$ with $\delta \in \{0, 1\}$.  We first study the inner integral over $t \in \Bbb{R}^{\times}$, call it $\tilde{h}(s) $, which equals 
\begin{equation}\label{3aa}
  \begin{split}
    \tilde{h}(s)& = \int_{\Bbb{R}} b \sqrt{|t(1-t)|}V_1\Big( \frac{bt -X}{Y}\Big)V_2\Big(\frac{b(t-1)+b-X}{Y}\Big)\text{sgn}(t)^{\delta} |t|^{-s} \Big|\frac{1-t}{t}\Big|^{-it_g -1/2} \frac{dt}{|t|}\\
    & =b \int_{\Bbb{R}} V\Big( \frac{bt -X}{Y}\Big) \text{sgn}(t)^{\delta} |t|^{-s } \Big|\frac{t-1}{t}\Big|^{-it_g } dt.
  \end{split}
\end{equation}
Note that the condition $1\leq b \leq X/2$ implies $t \geq 2$.  The function $\tilde{h}$ is an entire function in $s$, and by partial integration, we have
\begin{equation}\label{4}
  \tilde{h}(s) \ll_{t_g, \Re s, A} Y \Big(\frac{X}{b}\Big)^{-\Re s} \Big(1 + \frac{|\Im s|}{X/Y}\Big)^{-A}
\end{equation}
for $A > 0$.

Next, we study the gamma quotient
\begin{equation*}
G(\pi; \chi) = \gamma(1/2, \pi \otimes \chi) \gamma(1, \pi_1\otimes \chi_2^{-1}\otimes {\chi}^{-1})
\end{equation*}
for fixed $\pi_1$, $\chi_2$. 
 
If $\pi = \pi_{r, \eta}$ with $r \in \Bbb{R}$, $\eta \in \{0, 1\}$ belongs to the principal series, then for $\chi = |.|^s \text{sgn}^{\delta}$ with $s = \sigma + it$, the function $G(\pi; \chi)$ has a pole at $s = 1/2 \pm ir$ if $\delta \equiv \eta $ (mod 2), whose residue is
\begin{equation}\label{3c}
  \underset{s = 1/2 \pm ir}{\text{res}} G(\pi; \chi) \asymp (1 + |r|)^{-1/2}.
\end{equation}
Away from this pole, we have by Stirling's formula the upper bound
\begin{equation}\label{3a}
  \begin{split}
    &G(\pi; \chi) \ll_{ \pi_1, \chi}(1 + |t|)^{2\sigma - 1}
      ((1 + |t - r|)(1 + |t + r|))^{-\sigma}
  \end{split}
\end{equation}
for $1/3 \leq \sigma \leq 4/3$, say.

If $\pi = \pi_k$ is discrete series, then $G(\pi; \chi)$ is holomorphic for $1/3 \leq \sigma \leq k/2 - 1/3$, and bounded by
\begin{equation}\label{3b}
  \begin{split}
    &G(\pi; \chi) \ll_{\sigma, \pi_1, \chi}(1 + |t|)^{2\sigma - 1}(t^2 + k^2)^{-\sigma}.
\end{split}
\end{equation}

If $\pi$ belongs to the principal series, we recall \eqref{4} and shift the contour to $\Re s = 1 + \varepsilon$ getting a contribution
\begin{equation}\label{bound-h}
  \begin{split}
    h^{\vee}(\pi, b)& \ll_{A, t_f, t_g}  Y \Big(\frac{X}{b}\Big)^{-1/2} (1+|r|)^{-1/2} \Big(1 + \frac{|r|}{X/Y}\Big)^{-A} \\
    &\quad\quad\quad+ \int_{\Bbb{R}} Y \Big(\frac{X}{b}\Big)^{-1-\varepsilon} \Big(1 + \frac{|t|}{X/Y}\Big)^{-A} \frac{(1 + |t|)^{1+2\varepsilon}}{ ((1 + |t - r|)(1 + |t + r|))^{1+\varepsilon}}dt\\
    & \ll Y \Big(\frac{X}{b}\Big)^{-1/2} (1+|r|)^{-1/2} \Big(1 + \frac{|r|}{X/Y}\Big)^{-A} + X^{\varepsilon} \frac{Yb}{X} \min\Big(\frac{X^2/Y^2}{r^{2+\varepsilon}}, 1\Big).
  \end{split}
\end{equation}
The key point here is that the exponent of $r$ in the denominator of the last term is strictly larger than $2$.  With this, we return to \eqref{3}. By Weyl's law and individual bounds for Hecke eigenvalues on the one hand, and \cite[Lemma 12]{BM} on the other hand, we have
\begin{equation}\label{hecke-av}
\sum_{C(\pi) \ll T} |\lambda_{\pi}(b)|^2 \ll T^{2+\varepsilon} \min\Big( b^{2\theta} , 1 + \frac{b^{1/2}}{T^2}\Big)
\end{equation}
Moreover, by \cite[Theorem 1]{BJN} (simpler technology, e.g.\ \cite{BR}, would in principle suffice, too) we have
\begin{equation}\label{BR}
  \sum_{C(\pi) \ll T} L(\tfrac{1}{2}, \pi \otimes \pi_f \otimes \pi_g) \ll_{f, g} T^{2+\varepsilon}.
\end{equation}
Combining \eqref{bound-h} -- \eqref{BR} with Cauchy--Schwarz, dyadic partition, suitable choices of $\varepsilon$, and the standard lower bound for $L(1, \pi, \text{Ad})\gg C(\pi)^{-\varepsilon}$, we can bound the contribution of the Maa{\ss} spectrum in \eqref{3} by
\begin{displaymath}
  \begin{split}
    &  \ll_{A, \varepsilon, t_f, t_g} \sum_{T = 2^{\nu}} \frac{T^{2+\varepsilon/10}}{b^{1/2}} \min\Big( b^{\theta} , 1 + \frac{b^{1/4}}{T}\Big)\Big[Y \Big(\frac{X}{b}\Big)^{-1/2} \frac{1}{T^{1/2}} \Big(1 + \frac{T}{X/Y}\Big)^{-A} + X^{\varepsilon} \frac{Yb}{X} \min\Big(\frac{X^2/Y^2}{T^{2+\varepsilon}}, 1\Big)\Big]\\
    &\ll_{\varepsilon} X^{5\varepsilon} \Big(\frac{X}{Y}\Big)^{2}\frac{1}{b^{1/2}} \min\Big( b^{\theta} , 1 + \frac{Yb^{1/4}}{X}\Big)\Big[ \frac{b^{1/2} Y^{3/2}}{X} +  \frac{Yb}{X} \Big],
  \end{split}
\end{displaymath}
and  \eqref{1} follows. The estimation of the contribution of the discrete series and the Eisenstein series (with \cite[Theorem 2]{BJN} instead of \cite[Theorem 1]{BJN}) is similar, but easier.\\

To prove \eqref{2}, we argue similarly. We first observe that $ \int_{X}^{2X} | \mathcal{S}(x, Y, b)|^2 dx$ is bounded by
\begin{equation*}
  \begin{split}
    \int_{ [\widehat{G}]^{\mathrm{u}}_{\mathrm{gen}, \{\infty\}}} \int_{ [\widehat{G}]^{\mathrm{u}}_{\mathrm{gen}, \{\infty\}}} & \frac{|\lambda_{\pi}(b)\lambda_{\pi'}(b)|}{b}  \frac{\sqrt{|L(\tfrac{1}{2}, \pi \otimes \pi_f \otimes \pi_g)L(\tfrac{1}{2}, \pi' \otimes \pi_f \otimes \pi_g)|}}{L(1, \pi, \text{Ad})L(1, \pi', \text{Ad})} \\& \Big| \int W\Big(\frac{x}{X}\Big) h^{\vee}(\pi, b)h^{\vee}(\pi', b) dx\Big| d \pi \,d\pi'
  \end{split}
\end{equation*}
where $W$ is a suitable fixed smooth function with support in $[3/4, 5/2]$ and $h^{\vee}$ has the same meaning as before except that $X$ is replaced by the integration variable $x$. We estimate this by
\begin{equation}\label{2a}
  \begin{split}
    b^{2\theta - 1} \int_{ [\widehat{G}]^{\mathrm{u}}_{\mathrm{gen}, \{\infty\}}} \int_{ [\widehat{G}]^{\mathrm{u}}_{\mathrm{gen}, \{\infty\}}} &   \big( L(\tfrac{1}{2}, \pi \otimes \pi_f \otimes \pi_g) + L(\tfrac{1}{2}, \pi' \otimes \pi_f \otimes \pi_g) \big)(C(\pi)C(\pi'))^{\varepsilon}\\
                                                                                                               & \Big| \int W\Big(\frac{x}{X}\Big) h^{\vee}(\pi, b)h^{\vee}(\pi', b) dx\Big| d \pi \,d \pi'.
  \end{split}
\end{equation}
Note that by the inequality $\lvert A B \rvert \leq (A^2 + B^2)/2$, we can assume that $\pi$ and $\pi'$ are of the same type: both principal series, both discrete series or both Eisenstein.

We study the $x$-integral more carefully: recalling \eqref{3aa}, it equals
\begin{equation}\label{5}
  \begin{split}
    \sum_{\delta_1, \delta_2\in \{0, 1\}} \int_{(3/4)} \int_{(3/4)} b^2 \Phi(s_1, s_2)
    G(\pi, \text{sgn}^{\delta_1}|.|^{s_1})G(\pi', \text{sgn}^{\delta_2}|.|^{s_2}) \frac{ds_1\, ds_2}{(2\pi i)^2}
  \end{split}
\end{equation}
where
\begin{equation*}
\Phi(s_1, s_2) =  \int_{\Bbb{R}}\int_{\Bbb{R}} \Xi(t_1, t_2)  |t_1|^{-s_1} |t_2|^{-s_2} \phi(t_1, t_2) dt_1\, dt_2 
\end{equation*}
with
\begin{equation*}
\Xi(t_1, t_2) = \Xi_{X, Y, b}(t_1, t_2) =  \int W\Big(\frac{x}{X}\Big)V\Big( \frac{bt_1 -x}{Y}\Big)   V\Big( \frac{bt_2 -x}{Y}\Big) dx
\end{equation*}
and
\begin{equation*}
\phi(t_1, t_2) = \text{sgn}(t_1)^{\delta_1}\text{sgn}(t_2)^{\delta_2}\Big|\frac{( t_1-1)( t_2-1)}{t_1t_2}\Big|^{-it_g }.
\end{equation*}
We change variables $t_1 = t$, $t_2 = t_1(1 + \tau)$ 
and observe that the map $(t, \tau) \mapsto \Xi(t, t(1+\tau))$ is supported on $t \asymp X/b$ and $\tau \ll Y/X \leq 1$ and satisfies
\begin{equation*}
 \frac{d^{j_1}}{dt^{j_1}} \frac{d^{j_2}}{d \tau^{j_2}}\big( \Xi(t, t(1+\tau))\phi(t, t(1+\tau)) \big)\ll_{j_1, j_2, g}Y \Big(\frac{X}{b}\Big)^{-j_1}   \Big(\frac{Y}{X}\Big)^{-j_2}, 
\end{equation*}
in particular it is ``flat'' in $t$ and $\tau$. To see this, we couple each differentiation with respect to $t$ with an integration by parts in $x$ using the identity
\begin{displaymath}
\begin{split}
&\frac{d}{dt}  \int W\Big(\frac{x}{X}\Big)V_1\Big( \frac{bt -x}{Y}\Big)   V_2\Big( \frac{bt(1 + \tau) -x}{Y}\Big) dx \\
& = \frac{b\tau}{Y} \int W\Big(\frac{x}{X}\Big)V_1\Big( \frac{bt -x}{Y}\Big)   V'_2\Big( \frac{bt(1 + \tau) -x}{Y}\Big)dx + \frac{b}{X}\int  W'\Big(\frac{x}{X}\Big)  V_1\Big( \frac{bt -x}{Y}\Big)   V_2\Big( \frac{bt(1 + \tau) -x}{Y}\Big) dx
\end{split}
\end{displaymath}
for any smooth and compactly supported function $V_1, V_2$ (e.g.\ derivatives of $V$), and note that  $b|\tau|/Y  \ll b/X$.

Now integrating by parts, we see that $\Phi$ is entire in $s_1, s_2$ and bounded by
\begin{equation*}
\Phi(s_1, s_2) \ll_{g, \Re s_1, \Re s_2, A} Y \Big(\frac{X}{b}\Big)^{1-\Re(s_1 + s_2)} \frac{X}{b} \frac{Y}{X} (1 + |s_1+s_2|)^{-A}\Big(1 + \frac{|s_2|}{X/Y}\Big)^{-A}.
\end{equation*}
We substitute this bound back into \eqref{5} and estimate the $s_1, s_2$-integral by contour shifts. Let us assume that $\pi, \pi'$ are principal series representations with spectral parameters $r, r'$ and  recall the bounds \eqref{3a}, \eqref{3b} as well as \eqref{3c}. Again we shift both contours to $\Re s_j = 1+\varepsilon$, each of which contributes a (possible) residue and a remaining integral. The four terms contribute at most
\begin{displaymath}
  \begin{split}
    &bY^2 (1 + |r|)^{-1/2}(1 + |r'|)^{-1/2} (1 + |r- r'|)^{-A}\Big(1 + \frac{|r|}{X/Y}\Big)^{-A} \Big(1 + \frac{|r'|}{X/Y}\Big)^{-A} \\
    & + \frac{b^{3/2} Y^2}{X^{1/2}} (1 + |r|)^{-1/2} \int_{\Bbb{R}} \frac{ (1 + |t|)^{1+2\varepsilon}}{  ((1 + |t - r'|)(1 + |t + r'|))^{1+\varepsilon}}(1 + |r+t|)^{-A}\Big(1 + \frac{|t|}{X/Y}\Big)^{-A} dt\\
    & +  \frac{b^{3/2} Y^2}{X^{1/2}} (1 + |r'|)^{-1/2} \Big(1 + \frac{|r'|}{X/Y}\Big)^{-A}  \int_{\Bbb{R}} \frac{ (1 + |t|)^{1+2\varepsilon}}{  ((1 + |t - r|)(1 + |t + r|))^{1+\varepsilon}}(1 + |r'+t|)^{-A}dt\\
    & + \frac{b^2 Y^2}{X}\int_{\Bbb{R}} \int_{\Bbb{R}} \frac{( (1 + |t_1|)(1 + |t_1|))^{1+2\varepsilon} (1 + |t_1+t_2|)^{-A}(1 + Y| t_2|/X)^{-A}}{ ((1 + |t_1 - r|)(1 + |t_1 + r|)(1 + |t_2 - r'|)(1 + |t_2 + r'|))^{1+\varepsilon}} dt_1\, dt_2.
  \end{split}
\end{displaymath}
We call the four lines in the previous display $M_i(r, r')$, $1 \leq i \leq 4$. 
It is straightforward to estimate 
\begin{displaymath}
  \begin{split}
& M_2(r, r')  \ll \frac{b^{3/2} Y^2}{X^{1/2}}  \Big(1 + \frac{|r|Y}{X}\Big)^{-A} \frac{(1 + |r|)^{1/2+2\varepsilon}}{(1 + |r - r'|)(1 + |r + r'|)^{1+\varepsilon}}\\
& M_3(r, r')  \ll \frac{b^{3/2} Y^2}{X^{1/2}}  \Big(1 + \frac{|r'|Y}{X}\Big)^{-A} \frac{(1 + |r'|)^{1/2+2\varepsilon}}{(1 + |r - r'|)(1 + |r + r'|)^{1+\varepsilon}}\\
& M_4(r, r')  \ll  \frac{b^2 Y^2}{X} \Big(\frac{X}{Y}\Big)^{3+4\varepsilon} \frac{1}{r^{2+\varepsilon} r'^{2+\varepsilon}}. 
  \end{split}
\end{displaymath}
We substitute these bounds for \eqref{5} back into \eqref{2a} and use Weyl's law together with \eqref{BR}. We consider the four terms in turn. 

For the first term we observe that
\begin{displaymath}
\begin{split}
\int_{\Bbb{R}} |r'| M_1(r, r')\, dr'\ll   bY^2 \Big(1 + \frac{|r|}{X/Y}\Big)^{-A}, \quad \ \int_{\Bbb{R}} |r| M_1(r, r')\, dr\ll   bY^2 \Big(1 + \frac{|r'|}{X/Y}\Big)^{-A} 
\end{split}
\end{displaymath}
Summing the first against $L(\tfrac{1}{2}, \pi\otimes \pi_f\otimes \pi_g)$ and the second against  $L(\tfrac{1}{2}, \pi'\otimes \pi_f\otimes \pi_g)$ by \eqref{BR} we obtain a total contribution of $b^{2\theta}X^{2+O(\varepsilon)}$ to \eqref{2a}. 

 Next we have
 \begin{displaymath}
   \begin{split}
     \int_{\Bbb{R}} |r'| M_2(r, r')\, dr'\ll    \frac{b^{3/2} Y^2(1 + |r|)^{1/2}}{X^{1/2}} \Big(1 + \frac{|r|Y}{X}\Big)^{-A} , \quad \int_{\Bbb{R}} |r| M_2(r, r')\, dr\ll  b^{3/2} Y^{3/2}   \min\Big(1, \frac{(X/Y)^2}{(r')^{2+\varepsilon}}\Big).
   \end{split}
 \end{displaymath}
 Summing the first against $L(\tfrac{1}{2}, \pi\otimes \pi_f\otimes \pi_g)$ and the second against  $L(\tfrac{1}{2}, \pi'\otimes \pi_f\otimes \pi_g)$ by \eqref{BR}, we obtain a  total contribution of $b^{2\theta + 1/2}X^{3/2+O(\varepsilon)} Y^{-1/2}$. 

 The third term is similar. 

 Finally, the last term is easily seen to contribute $b^{2\theta + 1}X^{1+O(\varepsilon)} Y^{-1}$. Hence the total contribution of the Maa{\ss} spectrum is
 \begin{equation*}
   b^{2\theta} X^{2+\varepsilon} \Big(1 + \frac{b^{1/2}}{Y^{1/2}} + \frac{b}{Y}\Big)
 \end{equation*}
 as desired. 

The contribution of the discrete series and the Eisenstein series is similar, but easier.

\section{Appendix}

In this appendix we compare the integral transform \eqref{trafo} in Motohashi's formula with the formula in Theorem \ref{thm-moment} for $\chi = \text{triv}$, $\pi_1, \pi_2$ principal series with spectral parameter 0, and show that they coincide. We start by considering \eqref{trafo} 
for a test function $V$ of the shape 
\begin{equation}\label{Happendix}
   V(t) = \int_{\Bbb{R}^{\times}} H(z) |z|^{-it} \frac{dz}{|z|}, \quad H(z) = \phi(z) |z|^{1/2}
 \end{equation}
 as in \eqref{wetachi} with $\eta = |.|^{-it}, \chi = \text{triv}$.  
 The $x$-integral in \eqref{trafo}   equals
 \begin{equation*}
   \pi \int_{(0)}  \int_{-\infty}^{\infty} \phi(x) |x|^{1/2 - u} \frac{dx}{|x|} \sum_{\pm}  \Big(1 + \frac{1}{y}\Big)^{\pm u} \frac{du}{2\pi i}.
 \end{equation*}
 We write $\phi(x) = \phi_+(x) + \phi_-(-x)$ with $\phi_+, \phi_-$ supported on $(0, \infty)$.   Then by Mellin inversion the previous expression equals 
 \begin{displaymath}
   \begin{split}
     & \pi   \int_{(0)} \big(\widehat{\phi}_{+} + \widehat{\phi}_{-}\big) (1/2 - u)  \sum_{\pm}  \Big(1 + \frac{1}{y}\Big)^{\pm u}  \frac{du}{2\pi i} \\
     & = \pi   \int_{(1/2)} \big(\widehat{\phi}_{+} + \widehat{\phi}_{-}\big) (u)  \sum_{\pm}  \Big(1 + \frac{1}{y}\Big)^{\pm (1/2 - u)}  \frac{du}{2\pi i}  = \pi \sum_{\sigma_1, \sigma_2 \in \{\pm 1\}}  \Big( 1 + \frac{1}{y}\Big)^{\frac{1}{2}\sigma_1} \phi \Big(\sigma_2  \Big(  1+ \frac{1}{y}\Big)^{\sigma_1}\Big).
   \end{split}
 \end{displaymath}
 Altogether we obtain
 \begin{equation*}
   \begin{split}
  \widecheck{V}(r) =\pi \sum_{\sigma_1, \sigma_2, \sigma_3 \in \{\pm 1\}} & \int_0^{\infty} \Big( 1 + \frac{1}{y}\Big)^{\frac{1}{2}\sigma_1} \phi \Big(\sigma_2  \Big(  1+ \frac{1}{y}\Big)^{\sigma_1}\Big)y^{-\sigma_3 ir} \frac{\Gamma(\tfrac{1}{2} +\sigma_3  ir)^2}{\Gamma(1+\sigma_3 2 i r)} \\
                                                                          &\times F(\tfrac{1}{2}+\sigma_3 ir, \tfrac{1}{2} +\sigma_3 ir, 1 +\sigma_3  2ir, - 1/y) \Big(1 + \sigma_3 \frac{i}{\sinh( \pi r)}\Big) \frac{dy}{y \sqrt{1+y}}. 
\end{split}
\end{equation*}
Gluing together the $y$-integral with the $\sigma_1, \sigma_2$ sums and writing $t = \sigma_2 (1 + 1/y)^{\sigma_1}$, 
we rewrite this as 
 \begin{displaymath}
\begin{split}
 & \pi \sum_{  \pm }  \int_{-1}^1   \phi (t)\Big(\frac{1}{|t|} - 1\Big)^{\pm  ir} \frac{\Gamma(\tfrac{1}{2} \pm  ir)^2}{\Gamma(1\pm  2 i r)} F\Big(\tfrac{1}{2}\pm  ir, \tfrac{1}{2} \pm ir, 1 \pm  2ir, 1-\frac{1}{ |t|}  \Big) \Big(1 \pm \frac{i}{\sinh( \pi r)}\Big) \frac{dt}{  \sqrt{|t( |t| - 1) |}}\\
 &+ \pi  \sum_{  \pm }  \int_{|t| > 1}   \phi (t) \big(|t| - 1\big)^{\pm  ir} \frac{\Gamma(\tfrac{1}{2} \pm  ir)^2}{\Gamma(1\pm  2 i r)} F\big(\tfrac{1}{2}\pm  ir, \tfrac{1}{2} \pm ir, 1 \pm  2ir, 1-|t|  \big) \Big(1 \pm \frac{i}{\sinh( \pi r)}\Big) \frac{dt}{  \sqrt{| |t| - 1 |}}
 \end{split}
\end{displaymath}
Here we can use \cite[9.131.1]{GR} to write
\begin{equation}\label{hyp}
F\big(\tfrac{1}{2}\pm  ir, \tfrac{1}{2} \pm ir, 1 \pm  2ir, 1-|t|  \big)  = |t|^{-1/2 \mp ir} F\Big(\tfrac{1}{2}\pm  ir, \tfrac{1}{2} \pm ir, 1 \pm  2ir, 1-\frac{1}{|t|}\Big) 
\end{equation}
and simplify
 \begin{equation}\label{wcheck}
\begin{split}
  \widecheck{V}(r) = \pi \sum_{  \pm }  \int_{-\infty}^{\infty}   \phi (t) \Big|\frac{1}{|t|} - 1\Big|^{\pm  ir} &\frac{\Gamma(\tfrac{1}{2} \pm  ir)^2}{\Gamma(1\pm  2 i r)} F\Big(\tfrac{1}{2}\pm  ir, \tfrac{1}{2} \pm ir, 1 \pm  2ir, 1-\frac{1}{ |t|}  \Big)\\
                                                                                                                 &\times \Big(1 \pm \frac{i}{\sinh( \pi r)}\Big) \frac{dt}{  \sqrt{|t( |t| - 1) |}}.  
\end{split}
\end{equation}

We now consider the moment formula in Theorem \ref{thm-moment} for $\chi = \text{triv}$ $F = \Bbb{Q}$, $S = \{\infty\}$, $\pi_1$, $\pi_2$ unramified principal series at infinity with trivial Langlands parameter (i.e.\ generated by two even Maa{\ss} forms with spectral parameters  $0$), and $w$ as in \eqref{wetachi} with $H$ of the form \eqref{Happendix}.    We have $\text{res}_{s=1}\zeta(s) = 1$ and $\eta$ runs over characters of the form $  |.|^{it}$, $t \in \Bbb{R}$  with measure $dt/(2\pi i)$.   We recall that the measure $d\sigma$  is half the counting measure. Therefore an appropriate comparison with \eqref{trafo} is $\pi \cdot h^{\sharp}(\sigma, \text{triv})$ for $\sigma$ generated by an even Maa{\ss} form of spectral parameter $r$ (for odd Maa{\ss} forms, the central $L$-value in \eqref{fourth} vanishes), $\chi_2$ trivial and $\pi_1$ having spectral parameter $0$.

Writing $H$ in terms of $\phi$ and changing variables $t  = 1-y$, the $y$-integral in \eqref{hsh} equals
\begin{equation*}
  \int_{\Bbb{R}^{\times}} \phi(1-t) \chi(t) \frac{dt}{|t|}.
\end{equation*}
Then spelling out the gamma factors and integrating over characters $\chi = \text{sgn}^{\delta} |.|^s$, $\Re s = 0$, $\delta \in \{0, 1\}$, we have (cf.\ \eqref{measurechi} for the normalization) 
\begin{equation}\label{hast}
\begin{split}
 \pi h^{\sharp}(\sigma, \text{triv}) =  & \frac{\pi}{2} \int_{(1/4)} \int_{-\infty}^{\infty} \phi(1-t) |t|^{s} \frac{dt}{|t|}  \frac{\Gamma_{\Bbb{R}}(s)^2  }{\Gamma_{\Bbb{R}}(1 - s)^2} \prod_{\pm} \frac{ \Gamma_{\Bbb{R}}(\tfrac{1}{2} \pm ir - s)}{\Gamma_{\Bbb{R}}(\tfrac{1}{2} \pm ir + s)} \frac{ds}{2\pi i}\\
  & + \frac{\pi}{2}  \int_{(1/4)} \int_{-\infty}^{\infty}\text{sgn}(t) \phi(1-t) |t|^{s} \frac{dt}{|t|}  \frac{\Gamma_{\Bbb{R}}(s+1)^2  }{\Gamma_{\Bbb{R}}(2 - s)^2} \prod_{\pm} \frac{ \Gamma_{\Bbb{R}}(\tfrac{3}{2} \pm ir - s)}{\Gamma_{\Bbb{R}}(\tfrac{3}{2} \pm ir + s)} \frac{ds}{2\pi i}.
\end{split}
\end{equation}

We now distinguish several cases. Let us first assume that $\phi$ is supported on $(0, 1)$. Then the previous formula simplifies as
\begin{displaymath}
\begin{split}
 \frac{\pi}{2} \int_{0}^{1} \phi(1-t)\Big( \int_{(3/4)} t^{-s}G(s)   \frac{ds}{2\pi i} \Big) \, dt, \quad G(s) =   \sum_{\delta \in \{0,1\}} \frac{\Gamma_{\Bbb{R}}(1-s+\delta)^2  }{\Gamma_{\Bbb{R}}(s+\delta)^2} \prod_{\pm} \frac{ \Gamma_{\Bbb{R}}(-\tfrac{1}{2}+\delta \pm ir + s)}{\Gamma_{\Bbb{R}}(\tfrac{3}{2}+\delta \pm ir - s)}.
\end{split}
\end{displaymath}
The key point is that the sum over $\delta$ can be re-written as
\begin{equation*}
G(s) = \frac{2}{\pi} (2\cosh(\pi r)- \sin(2\pi s))\Gamma(1-s)^2 \prod_{\pm} \Gamma ( - \tfrac{1}{2} \pm ir + s ).
\end{equation*}
Shifting the contour to the far left and picking up the residues, we see that
\begin{displaymath}
  \begin{split}
    \int_{(3/4)} t^{-s}G(s)   \frac{ds}{2\pi i} &=  \frac{2}{\pi} \sum_{\pm} \sum_{n= 0}^{\infty}(-1)^n (2\cosh(\pi r) \pm i \sinh(2\pi r))\frac{\Gamma(\tfrac{1}{2} - ir + n)^2 \Gamma(2 i r - n)}{n!} t^{-1/2 -ir + n}\\
                                                & = 2 \sum_{\pm}  t^{-1/2 \pm ir}   \Big(1 \pm  \frac{i }{\sinh( \pi r)}\Big)\frac{ \Gamma(\tfrac{1}{2} \pm ir)^2}{\Gamma(1 \pm 2 i r)}  F(\tfrac{1}{2} \pm ir, \tfrac{1}{2} \pm ir , 1 \pm 2 ir, t),
  \end{split}
\end{displaymath}
so that
 \begin{displaymath}
 \begin{split}
\pi h^{\sharp}(\sigma, \text{triv}) =   \int_{0}^{1} \phi( t)  \sum_{\pm} |1-t|^{-1/2 \pm ir}    \Big(1  \pm  \frac{i }{\sinh( \pi r)}\Big)\frac{ \Gamma(\tfrac{1}{2}  \pm ir)^2}{\Gamma(1 \pm 2 i r)} 
     F(\tfrac{1}{2}  \pm ir, \tfrac{1}{2}  \pm ir, 1 \pm 2 ir, 1-t) dt.
     \end{split}
\end{displaymath}  
An application of \eqref{hyp} shows that this matches \eqref{wcheck}, as desired. 

Next we treat the case when $\text{supp}(\phi) \subseteq (1, \infty)$. Here \eqref{hast} simplifies as 
\begin{displaymath}
\begin{split}
&\frac{\pi}{2} \int_{0}^{\infty} \phi(1 + t) \int_{(1/4)} |t|^{s-1} \sum_{\delta\in \{0, 1\}} (-1)^{\delta}\frac{\Gamma_{\Bbb{R}}(s+\delta)^2  }{\Gamma_{\Bbb{R}}(1+\delta - s)^2} \prod_{\pm} \frac{ \Gamma_{\Bbb{R}}(\tfrac{1}{2}+\delta \pm ir - s)}{\Gamma_{\Bbb{R}}(\tfrac{1}{2}+\delta \pm ir + s)} \frac{ds}{2\pi i} \, dt\\
& = 2 \int_{0}^{\infty} \phi(1 + t) \int_{(3/4)} |t|^{-s} \Gamma(1-s)^2 (-\cosh(\pi r) \cos(\pi s) + \sin(\pi s))\prod_{\pm} \Gamma ( - \tfrac{1}{2} \pm ir +s )  \frac{ds}{2\pi i} \, dt. 
\end{split}
\end{displaymath}
Note that in contrast to the previous case, the $s$-integral is now rapidly decaying everywhere. By the same computation we recast this as
\begin{equation*}
  \pi \int_{0}^{\infty} \phi(1+t) \sum_{\pm} t^{-1/2 \pm ir}   \Big(1 \pm  \frac{i}{\sinh( \pi r)}\Big)\frac{ \Gamma(\tfrac{1}{2}  \pm ir)^2}{\Gamma(1 \pm 2 i r)}  F(\tfrac{1}{2}  \pm ir, \tfrac{1}{2}  \pm ir, 1 \pm 2 ir, -t) dt
\end{equation*}
Again a change of variables along with \eqref{hyp} yields \eqref{wcheck}. 

The case $\text{supp}(\phi) \subseteq (-\infty, 0)$ is similar.\\

\textbf{Acknowledgement.} The authors  would like to thank Gergely Harcos and Philippe Michel for helpful feedback on an earlier draft and the referees for a careful reading of the manuscript.

\end{document}